\documentclass{article}
\usepackage {amsmath, amsfonts, amssymb, mathrsfs, amsthm,stmaryrd, sectsty,hyphenat,enumitem}
\usepackage[usenames]{color}
\usepackage{palatino}
\usepackage[all]{xy}
\usepackage{etoolbox}
\def\mf#1{\mathfrak{#1}}
\def\mc#1{\mathcal{#1}}
\def\mb#1{\mathbb{#1}}
\def\tx#1{\textrm{#1}}

\def\tr{\tx{tr}\,}
\def\R{\mathbb{R}}

\def\C{\mathbb{C}}
\def\Q{\mathbb{Q}}
\def\A{\mathbb{A}}
\def\Z{\mathbb{Z}}
\def\N{\mathbb{N}}
\def\F{\mathbb{F}}
\def\lmod{\setminus}

\def\ol#1{\overline{#1}}
\def\ul#1{\underline{#1}}

\def\hat{\widehat}

\def\rw{\rightarrow}
\def\lw{\leftarrow}

\def\lrw{\longrightarrow}
\def\llw{\longleftarrow}
\def\hrw{\hookrightarrow}

\def\lw{\leftarrow}

\def\<{\langle}
\def\>{\rangle}

\newenvironment{mytitle}
{\begin{center}\large\sc}
{\end{center}}

\theoremstyle{definition}
\newtheorem{thm}{Theorem}[subsection]
\newtheorem{lem}[thm]{Lemma}
\newtheorem{pro}[thm]{Proposition}
\newtheorem{cor}[thm]{Corollary}
\newtheorem{fct}[thm]{Fact}

\newtheorem{cnd}[thm]{Condition}

\newtheorem{rem}[thm]{Remark}
\newtheorem{exa}[thm]{Example}

\newcommand{\qedwhite}{\hfill\ensuremath{\Box}}

\AtEndEnvironment{thm}{\qedwhite}
\AtEndEnvironment{lem}{\qedwhite}
\AtEndEnvironment{pro}{\qedwhite}
\AtEndEnvironment{cor}{\qedwhite}
\AtEndEnvironment{fct}{\qedwhite}
\AtEndEnvironment{clm}{\qedwhite}
\AtEndEnvironment{cnd}{\qedwhite}
\AtEndEnvironment{cnj}{\qedwhite}
\AtEndEnvironment{exn}{\qedwhite}
\AtEndEnvironment{dfn}{\qedwhite}
\AtEndEnvironment{rem}{\qedwhite}
\AtEndEnvironment{exa}{\qedwhite}

\sectionfont{\center\sc\normalsize}
\subsectionfont{\bf\normalsize}

\numberwithin{equation}{section}

\begin{document}

\begin{mytitle}
Global rigid inner forms vs isocrystals
\end{mytitle}

\begin{center} Tasho Kaletha and Olivier Ta\"ibi \end{center}

\begin{abstract}
We compare the cohomology of the global Galois gerbes constructed in \cite{KotBG} and \cite{KalGRI}, respectively, and give applications to the theory of endoscopy.
\end{abstract}
{\let\thefootnote\relax\footnotetext{T.K. is supported in part by NSF grant DMS-1801687 and a Sloan Fellowship.}}


\section{Introduction}

The statement of the refined local and global Arthur-Langlands conjectures for non-quasi-split reductive groups involves the cohomology of certain Galois gerbes \cite{KalSimons}, where the notion of a Galois gerbe is that of \cite{LR87}. In summary, every such group $G$ is an inner form of its quasi-split form $G^*$, but it was observed by Vogan \cite{Vog93} that this relationship does not suffice for the normalization of various objects involved in the statement of the local Langlands conjecture.
The cohomology of a local gerbe is used to provide the necessary additional
data.
The cohomology of a global gerbe is used to organize the local data at all
places coherently, so that the local conjecture can be used in global
applications.

The gerbes constructed in \cite{KotBG} can be used for this purpose. However, not all local inner forms can be treated when $G^*$ does not have connected center, and not all global inner forms can be treated when $G^*$ does not satisfy the Hasse principle. We shall refer to the formulation of the local and global conjectures involving the gerbes of \cite{KotBG} as the \emph{isocrystal version}. The gerbes constructed in \cite{KalRI} and \cite{KalGRI} can be used without these technical hypotheses on $G^*$, but they are at the moment available only in characteristic zero. We shall refer to the resulting local and global conjectures as the \emph{rigid version}.

Assume first that the ground field $F$ is a finite extension of $\Q_p$ and that $G^*$ has connected center. Then both the isocrystal and the rigid version of the refined local Langlands conjecture are available for $G$. It was shown in \cite{KalRIBG} that these two versions are equivalent. Moreover, it was shown that the validity of the isocrystal version for all connected reductive groups with connected center is equivalent to the validity of the rigid version for all connected reductive groups without assumption on the center. These results were derived from a comparison theorem for the cohomology of the two local gerbes.

The current paper provides a comparison theorem for the cohomology of the two
global gerbes. We give two applications to this comparison. First, when the ground field $F$ is a finite extension
of $\Q$ and $G^*$ has connected center and satisfies the Hasse principle, so
that both the isocrystal version and the rigid version of the global
multiplicity formula are available, it is natural to ask if these two versions are equivalent. A formal argument, based on the canonicity of global transfer
factors, gives an affirmative answer, but sheds no light on the relation between the normalized local pairings at each place of $F$. Our cohomological result allows for this finer comparison.

Second, we generalize \cite[Proposition 4.4.1]{KalGRI}, which states that the product of local normalized transfer factors equal the canonical adelic transfer factor. In \cite{KalGRI} this was proved under the assumption that there exists a pair of related $F$-points in the group and its endoscopic group. While this assumption was also made in \cite{LS87}, where transfer factors were originally defined, it was later dropped in \cite{KS99} and replaced with the weaker assumption on the existence of an $F$-point in the endoscopic group that is related to an $\A$-point in the group. We use the results of the current paper to show that \cite[Proposition 4.4.1]{KalGRI} is valid under this weaker hypothesis.

We hope that the comparison of the cohomology of the two global gerbes will be useful beyond these applications, in light of Scholze's recent conjecture \cite[Conjecture 9.5]{ScholzeICM18} on the existence of a Weil cohomology theory for varieties over $\bar\F_p$ valued in the category of representations of the global gerbe of \cite{KotBG}.

Before we outline the comparison theorem in the global case, let us review it in the local case. Let $F$ be a finite extension of $\Q_p$ and $\Gamma$ the absolute Galois group of $F$. The local gerbe of \cite{KotBG}, which we shall denote by $\mc{E}^\tx{iso}$ here, is bound by the pro-torus $\mb{T}^\tx{iso}$ whose character module is the trivial $\Gamma$-module $\Q$. The local gerbe of $\cite{KalRI}$, which we shall denote by $\mc{E}^\tx{rig}$ here, is bound by the pro-finite algebraic group $P^\tx{rig}$ whose character module is the group of smooth functions $\Gamma \to \Q/\Z$, endowed with the obvious action of $\Gamma$. The map $X^*(\mb{T}^\tx{iso}) \to X^*(P^\tx{rig})$ sending $q \in \Q$ to the constant function with value $q$ provides a homomorphism $P^\tx{rig} \to \mb{T}^\tx{iso}$ defined over $F$. One proves that the push-out of $\mc{E}^\tx{rig}$ along this homomorphism equals $\mc{E}^\tx{iso}$. The resulting map of gerbes $\mc{E}^\tx{rig} \to \mc{E}^\tx{iso}$ induces a map between their cohomology. For example, when $T$ is an algebraic torus defined over $F$, we obtain a homomorphism of abelian groups $H^1(\mc{E}^\tx{iso},T) \to H^1(\mc{E}^\tx{rig},T)$. Both of these abelian groups have a description in terms of linear algebra. In the first case, we have the functorial isomorphism $H^1(\mc{E}^\tx{iso},T) \to X_*(T)_\Gamma$, where $X_*(T)$ is the co-character module of $T$. In the second case, we have the functorial isomorphism $H^1(\mc{E}^\tx{rig},T) \to \frac{X_*(T) \otimes \Q}{IX_*(T)}[\tx{tor}]$, where $I \subset \Z[\Gamma]$ is the augmentation ideal, and $[\tx{tor}]$ refers to the torsion subgroup. Let $E/F$ be any finite Galois extension splitting $T$. Let $N^\natural$ denote the renormalized norm map $[E:F]^{-1}\sum_{\sigma \in \Gamma_{E/F}}\sigma : X_*(T) \otimes \Q \to X_*(T) \otimes \Q$. Then the map $X_*(T)_\Gamma \to \frac{X_*(T) \otimes \Q}{IX_*(T)}[\tx{tor}]$ given by $y \mapsto y- N^\natural(y)$ makes the following diagram commute
\[ \xymatrix{
H^1(\mc{E}^\tx{iso},T)\ar[r]\ar[d]&X_*(T)_\Gamma\ar[d]\\
H^1(\mc{E}^\tx{rig},T)\ar[r]&\frac{X_*(T)\otimes\Q}{IX_*(T)}[\tx{tor}]
}
\]

We now come to the global case treated in this paper. Let $F$ be a finite extension of $\Q$. The global gerbe of \cite{KotBG}, which we shall denote by $\mc{E}^\tx{iso}$ here, is bound by a pro-torus $\mb{T}^\tx{iso}$, while the global gerbe of \cite{KalGRI}, which we shall denote by $\mc{E}^\tx{rig}$ here, is bound by a pro-finite algebraic group $P^\tx{rig}$. The description of the character modules is more technical and we will not discuss it in the introduction. Unlike in the local case, we do not know of a natural map $P^\tx{rig} \to \mb{T}^\tx{iso}$. In fact, there is good reason to believe that one cannot expect a natural map like that to exist. The comparison of the cohomology of the two gerbes $\mc{E}^\tx{iso}$ and $\mc{E}^\tx{rig}$ proceeds via an intermediary. We define a new pro-torus $\mb{T}^\tx{mid}$ and natural maps $\mb{T}^\tx{iso} \to \mb{T}^\tx{mid} \lw P^\tx{rig}$. We prove that the classes in $H^2(\Gamma,\mb{T}^\tx{iso})$ and $H^2(\Gamma,P^\tx{rig})$ of the gerbes $\mc{E}^\tx{iso}$ and $\mc{E}^\tx{rig}$ meet in $H^2(\Gamma,\mb{T}^\tx{mid})$. This leads to a gerbe $\mc{E}^\tx{mid}$ bound by $\mb{T}^\tx{mid}$ and equipped with homomorphisms $\mc{E}^\tx{iso} \to \mc{E}^\tx{mid} \lw \mc{E}^\tx{rig}$. We then prove that, for every algebraic torus $T$ defined over $F$, the two diagrams
\[ \xymatrix{
  H^1(\mc{E}^\tx{mid},T)\ar[r]\ar[d]&\tx{Hom}_F(\mb{T}^\tx{mid},T)\ar[d]&H^1(\mc{E}^\tx{mid},T)\ar[r]\ar[d]&\tx{Hom}_F(\mb{T}^\tx{mid},T)\ar[d]\\
  H^1(\mc{E}^\tx{iso},T)\ar[r]&\tx{Hom}_F(\mb{T}^\tx{iso},T)&H^1(\mc{E}^\tx{rig},T)\ar[r]&\tx{Hom}_F(P^\tx{rig},T)
}
\]
are Cartesian and the vertical arrows in the left diagram are surjective.

An analogous discussion holds locally at each place $v$ of $F$: There are maps of gerbes $\mc{E}^\tx{iso}_v \to \mc{E}^\tx{mid}_v \lw \mc{E}^\tx{rig}_v$ over $F_v$ that are compatible with the analogous global maps via localization maps $\mc{E}^*_v \to \mc{E}^* \times_F F_v$. The Cartesian square relating $\mc{E}^\tx{iso}_v$ to $\mc{E}^\tx{mid}_v$ shows that there is a functorial isomorphism from $H^1(\mc{E}^\tx{mid}_v,T)$ to the group $\{(\lambda,\mu)|\lambda \in X_*(T)_\Gamma,\mu \in X_*(T)\otimes\Q,N^\natural(\lambda)=N^\natural(\mu)\}$.

Recalling the comparison map $\mc{E}^\tx{rig}_v \to \mc{E}^\tx{iso}_v$ constructed in \cite{KalRIBG} and reviewed above, we now obtain the following triangle
\begin{equation} \label{eq:locnoncom}
\xymatrix{
&\mc{E}^\tx{mid}_v&\\
\mc{E}^\tx{rig}_v\ar[ur]\ar[rr]&&\mc{E}^\tx{iso}_v\ar[ul].
}\end{equation}
This triangle does \emph{not} commute. In order to relate the global comparison results of this paper, which concern (via the localization maps) the two diagonal arrows, to the local comparison results of \cite{KalRIBG}, which concern the bottom horizontal arrow, we need to understand the failure of commutativity.

We construct a canonical splitting $\mc{E}^\tx{mid}_v \to \mc{E}^\tx{iso}_v$ of the map $\mc{E}^\tx{iso}_v \to \mc{E}^\tx{mid}_v$ and show that composing this splitting with $\mc{E}^\tx{rig}_v \to \mc{E}^\tx{mid}_v$ equals the bottom horizontal map in \eqref{eq:locnoncom}, i.e. the comparison map $\mc{E}^\tx{rig}_v \to \mc{E}^\tx{iso}_v$ of \cite{KalRIBG}. The non-commutativity of the above triangle is then encoded in the difference between the left diagonal map $\mc{E}^\tx{rig}_v \to \mc{E}^\tx{mid}_v$ and
the composition
\[ \mc{E}^\tx{rig}_v \to \mc{E}^\tx{mid}_v \to \mc{E}^\tx{iso}_v \to \mc{E}^\tx{mid}_v. \]
We show that the difference between the two homomorphisms $H^1(\mc{E}^\tx{mid}_v,T) \to H^1(\mc{E}^\tx{rig}_v,T)$ induced by these two maps $\mc{E}^\tx{rig}_v \to \mc{E}^\tx{mid}_v$ is given on the linear algebraic side by the map that sends $(\lambda,\mu)$ to $\mu-N^\natural(\mu) \in \frac{X_*(T)\otimes\Q}{IX_*(T)}[\tx{tor}]$.

These cohomological results allow us to compare the two isocrystal and rigid versions of the multiplicity conjecture for discrete automorphic representations. More precisely, let $G^*$ be a quasi-split connected reductive group defined over $F$ and let $G$ be an inner form of $G^*$. Assuming the existence of the global Langlands group $L_F$, as well as the validity of the rigid version of the refined local Langlands correspondence, we constructed in \cite[\S4.5]{KalGRI} a pairing between the group $\mc{S}_\varphi$ associated to a discrete generic global parameter $\varphi : L_F \to {^LG}$ and the adelic $L$-packet $\Pi_\varphi(G)$. This pairing is an ingredient in the conjectural multiplicity formula \cite[(12.3)]{Kot84}. Its construction uses the cohomology of $\mc{E}^\tx{rig}$, but the result is independent of the cohomology classes used.

Assuming that $G^*$ has connected center and satisfies the Hasse principle, another such pairing can be constructed if one assumes the isocrystal version of the refined local Langlands correspondence and uses the cohomology of $\mc{E}^\tx{iso}$. This construction does not yet appear in the literature.
It is fairly analogous to that in \cite[\S4.5]{KalGRI} and we give the details
in Section \ref{sub:mult_formula_BG}.

As an application of our cohomological results, we show that the two
constructions -- using $\mc{E}^\tx{rig}$ and $\mc{E}^\tx{iso}$, respectively --
produce the same pairing between $\mc{S}_\varphi$ and $\Pi_\varphi(G)$.
More precisely, given an inner twist $\psi : G^* \to G$ we fix a cocycle
$z^\tx{mid} \in Z^1(\mc{E}^\tx{mid},G^*)$ that lifts the cocycle $\sigma \mapsto
\psi^{-1}\sigma(\psi)$ and use it to produce cocycles $z^\tx{iso} \in
Z^1(\mc{E}^\tx{iso},G^*)$ and $z^\tx{mid} \in Z^1(\mc{E}^\tx{mid},G^*)$.
The two global pairings are constructed as products of local pairings, each
normalized by the localization $z^\tx{iso}_v$ and $z^\tx{rig}_v$, respectively.
At each place $v$, the local pairings do depend on the choice of $z^\tx{mid}$,
but the resulting global pairings do not.
Even though conjectural, the local pairings are related by an explicit non-conjectural factor that is a result of the normalized character identities the pairings are required to satisfy. This follows from the local comparison results of \cite{KalRIBG}. However, due to the non-commutativity of \eqref{eq:locnoncom} the local comparison map $H^1(\mc{E}^\tx{iso}_v,G^*) \to H^1(\mc{E}^\tx{rig}_v,G^*)$ does not map $[z^\tx{iso}_v]$ to $[z^\tx{rig}_v]$. Thus the local comparison results of \cite{KalRIBG} need to be supplemented with the quantification of the non-commutativity of \eqref{eq:locnoncom} discussed above. Combining these results, we obtain an explicit factor relating the two local pairings at a given place $v$. The global comparison results of this paper imply that the product over all $v$ of these factors equals $1$ and therefore the two global pairings are equal.

Alongside this comparison result, we introduce in this paper a simplification of the construction of the global gerbe $\mc{E}^\tx{rig}$. In \cite{KalGRI} this construction involved choosing a sequence $(E_i,S_i,\dot S_i)$, where $E_i$ is an exhaustive tower of finite Galois extensions of $F$, $S_i$ is an exhaustive tower of finite sets of places of $F$, and $\dot S_i$ is a set of lifts of $S_i$ to places of $E_i$. Each triple $(E_i,S_i,\dot S_i)$ was required to satisfy a list of four conditions \cite[Conditions 3.3.1]{KalGRI}. In this paper we show that the resulting gerbe depends only on the set $\dot V$ of lifts to $\bar F$ of the places of $F$ that is defined by $\dot V=\varprojlim \dot S_i$.
That is, $\mc{E}^\tx{rig}$ is independent of the choices of $E_i$ and $S_i$.
Furthermore, we show that \cite[Conditions 3.3.1]{KalGRI} for each triple $(E_i,S_i,\dot S_i)$ are equivalent to one simple condition on $\dot V$, namely Condition \ref{cnd:dense} stating that $\bigcup_{v \in V} \Gamma_{\dot v}$ is dense in $\Gamma$.

\textbf{Acknowledgement:} T.K. wishes to thank ENS Lyon for the hospitality and excellent working conditions during the special program on the Geometrization of the Langlands Program in 2018, where the essential part of this work was completed.
This program was funded by Labex Milyon, ANR Project PerCoLaTor
ANR-14-CE25-0002-01 and ERC Advanced Grant GeoLocLang 742608.

\tableofcontents


\section{Definition of some local and global Galois modules}

In this section we shall define some modules for the Galois group of a finite Galois extension of a ground field $F$ that is either a number field or a local field. Taking colimits over all finite extensions of $F$ we shall obtain modules for the absolute Galois group of a number field or a local field. These will be the character modules of $\mb{T}^\tx{iso}$, $\mb{T}^\tx{mid}$, and $P^\tx{rig}$. We shall also discuss the transition maps with respect to which we take these colimits -- we call these \emph{inflation maps}. We shall also discuss \emph{localization} maps, which relate the global Galois modules to their local conterparts.


\subsection{The local modules} \label{sub:locmod}

Let $F$ be a local field, $E/F$ a finite Galois extension, $N$ natural number. We define the following $\Gamma_{E/F}$-modules:

\begin{enumerate}
  \item $M_E^\tx{iso}:=\Z$.
  \item $M^\tx{mid}_{E,N}$ consists of maps $f : \Gamma_{E/F} \to \frac{1}{N}\Z$ satisfying $\sum_\sigma f(\sigma) \in \Z$.
  \item $M^\tx{rig}_{E,N}$ consists of maps $f : \Gamma_{E/F} \to \frac{1}{N}\Z/\Z$ satisfying $\sum_\sigma f(\sigma)=0$.
\end{enumerate}

The module $M_E^\tx{iso}$ is the module $X$ of \cite[\S 5]{KotBG}, while
$M^\tx{rig}_{E,N}$ is the module $X^*(u_{E/F,N})$ of \cite[\S3.1]{KalRI}.

We define $\Gamma_{E/F}$-equivariant maps
\begin{equation} \label{eq:loccompmap}
M_E^\tx{iso} \stackrel{c^\tx{iso}}{\llw} M_{E,N}^\tx{mid} \stackrel{c^\tx{rig}}{\lrw} M_{E,N}^\tx{rig}
\end{equation}
via the formulas
\[ f^\tx{iso} = \sum_\sigma f^\tx{mid}(\sigma),\qquad f^\tx{rig}(\sigma)=-f^\tx{mid}(\sigma) \mod \Z. \]

\begin{fct} \label{fct:loccompsurj}
The maps $c^\tx{iso}$ and $c^\tx{rig}$ are surjective. The kernel of $c^\tx{rig}$ is induced.
\end{fct}
\begin{proof}
Immediate.
\end{proof}

If $N$ is divisible by $[E:F]$, there is a canonical splitting of $c^\tx{iso}$ defined by
\begin{equation} \label{eq:siso}
s^\tx{iso} : M_E^\tx{iso} \to M_{E,N}^\tx{mid},\quad s^\tx{iso}(f^\tx{iso})=f^\tx{mid},\quad f^\tx{mid}(\sigma) =
[E:F]^{-1} f^\tx{iso}.
\end{equation}
The image of $s^\tx{iso}$ is precisely $(M_{E,N}^\tx{mid})^\Gamma$.


\subsection{Local inflation maps} \label{sub:locinf}

We continue with the notation of \S\ref{sub:locmod}. Let $K/F$ be another finite Galois extension with $E \subset K$, $M$ a natural number divisible by $N$.

We define three maps, which we refer to as inflation maps:
\begin{enumerate}
  \item $M_E^\tx{iso} \to M_K^\tx{iso}$, given by multiplication by $[K:E]$.
  \item $M_{E,N}^\tx{mid} \to M_{K,M}^\tx{mid}$, given by $f^{\tx{mid},K}(\sigma)=f^{\tx{mid},E}(\sigma)$.
  \item $M_{E,N}^\tx{rig} \to M_{K,M}^\tx{rig}$, also given by $f^{\tx{rig},K}(\sigma)=f^{\tx{rig},E}(\sigma)$.
\end{enumerate}

These inflation maps fit into the commutative diagram
\[ \xymatrix{
  M_E^\tx{iso}\ar[d]&M_{E,N}^\tx{mid}\ar[l]_{c^\tx{iso}_{E,N}}\ar[r]^{c^\tx{rig}_{E,N}}\ar[d]&M_{E,N}^\tx{rig}\ar[d]\\
  M_K^\tx{iso}&M_{K,M}^\tx{mid}\ar[l]^{c^\tx{iso}_{K,M}}\ar[r]_{c^\tx{rig}_{K,M}}&M_{K,M}^\tx{rig}
}
\]

Using these inflation maps we can take in each case the colimit over all finite Galois extensions $E/F$ and all natural numbers $N$ and obtain the following:

\begin{align*}
  M^\tx{iso} := \varinjlim M^\tx{iso}_E &= \begin{cases}
    \Q & \text{ if } F \text{ is non-Archimedean,} \\
    \Z & \text{ if } F \simeq \C,\\
    \frac{1}{2}\Z & \text{ if } F=\R
  \end{cases} \\
  M^\tx{mid} := \varinjlim M^\tx{mid}_{E,N} &= \begin{cases}
    \Q[\Gamma] & \text{ if } F \text{ is non-Archimedean,} \\
    \Z[\Gamma] = \Z & \text{ if } F \simeq \C, \\
    \{ f : \Gamma \to \Q | \sum_{\sigma \in \Gamma} f(\sigma) \in \Z \} &
      \text{ if } F=\R
  \end{cases} \\
  M^\tx{rig} := \varinjlim M^\tx{rig}_{E,N} &= \begin{cases}
    \Q/\Z[\Gamma] & \text{ if } F \text{ is non-Archimedean,} \\
    0 & \text{ if } F \simeq \C, \\
    \{ f : \Gamma \to \Q/\Z | \sum_{\sigma \in \Gamma} f(\sigma) = 0 \} &
      \text{ if } F=\R
  \end{cases}
\end{align*}

Here $R[\Gamma]$ denotes the set of smooth functions $\Gamma \to R$.

The local comparison maps splice together to maps
\[ M^\tx{iso} \stackrel{c^\tx{iso}}{\llw} M^\tx{mid} \stackrel{c^\tx{rig}}{\lrw}
M^\tx{rig}, \]
the left being given by integrating over $\Gamma$ with respect to the normalized
Haar measure, and the right being induced by the \emph{negative} of the natural
projection $\Q \to \Q/\Z$.
The map $c^\tx{iso}$ has a canonical splitting $s^\tx{iso}$ whose image consists
of constant functions $\Gamma \to \Q$ in the non-Archimedean case (resp.\
$\Gamma \to \Z$ in the complex case, resp.\ $\Gamma \to \frac{1}{2}\Z$ in the
real case).
In the non-Archimedean case the composition $c^\tx{rig}\circ s^\tx{iso}$ equals
the map $X^*(\phi)$, where $\phi$ is the map defined in \cite[(3.13)]{KalRIBG},
as we see by dualizing Lemma 3.1 loc.\ cit.
In the Archimedean case we take this equality as the definition of $\phi$.


\subsection{A discussion of $M^{\tx{mid},\vee}_{E,N}$} \label{sub:discmmidloc}

We now describe the $\Gamma_{E/F}$-module $M^{\tx{mid},\vee}_{E,N}=\tx{Hom}_\Z(M^\tx{mid}_{E,N},\Z)$ and record some of its properties.

The obvious inclusions $\Z[\Gamma_{E/F}] \to M_{E,N}^\tx{mid} \to N^{-1}\Z[\Gamma_{E/F}]$ fit into the exact sequences
\begin{equation} \label{eq:locmides1}
0 \to M_{E,N}^\tx{mid} \to N^{-1}\Z[\Gamma_{E/F}]
\stackrel{\sum_\sigma}{\lrw} N^{-1}\Z/\Z \to 0
\end{equation}
and
\begin{equation} \label{eq:locmides2}
0 \to \Z[\Gamma_{E/F}] \to M_{E,N}^\tx{mid} \to N^{-1}\Z/\Z[\Gamma_{E/F}]_0
\to 0.
\end{equation}

We can identify $\Z[\Gamma_{E/F}]$ with its own dual via the pairing $(x,y)
\mapsto \sum_\sigma x(\sigma)y(\sigma)$.
Then $N\Z[\Gamma_{E/F}]$ dualizes to $N^{-1}\Z[\Gamma_{E/F}]$, the
inclusion $\Z[\Gamma_{E/F}] \to N^{-1}\Z[\Gamma_{E/F}]$ dualizes to the
inclusion $N\Z[\Gamma_{E/F}] \to \Z[\Gamma_{E/F}]$.
For a finite Galois extension $K$ of $F$ containing $E$, the inflation map
$\Z[\Gamma_{E/F}] \to \Z[\Gamma_{K/F}]$ dualizes to the map $\Z[\Gamma_{K/F}]
\to \Z[\Gamma_{E/F}]$ given by summing over $\Gamma_{K/E}$-cosets.

The inclusions $\Z[\Gamma_{E/F}] \to M_{E,N}^\tx{mid} \to N^{-1}\Z[\Gamma_{E/F}]$ dualize to the inclusions $N\Z[\Gamma_{E/F}] \to M^{\tx{mid},\vee}_{E,N} \to \Z[\Gamma_{E/F}]$ and describe $M^{\tx{mid},\vee}_{E,N}$ as the submodule of $\Z[\Gamma_{E/F}]$ given by $N\Z[\Gamma_{E/F}]+\Z$, where $\Z=\Z[\Gamma_{E/F}]^\Gamma$ is the subgroup consisting of constant functions. Note that $\Z[\Gamma_{E/F}]^\Gamma$ coincides with $[M^{\tx{mid},\vee}_{E,N}]^\Gamma$.

In terms of this description of $M^{\tx{mid},\vee}_{E,N}$ the exact sequences dual to \eqref{eq:locmides1} and \eqref{eq:locmides2} are described as follows. The dual of \eqref{eq:locmides1} is
\begin{equation} \label{eq:locmides3}
0 \to N\Z[\Gamma_{E/F}] \to M^{\tx{mid},\vee}_{E,N} \to \Z/N\Z \to 0,
\end{equation}
with the map $M^{\tx{mid},\vee}_{E,N} \to \Z/N\Z$ given by the natural projection on $\Z$ and trivial on $N\Z[\Gamma_{E/F}]$. The dual of \eqref{eq:locmides2} is
\begin{equation} \label{eq:locmides4}
0 \to M^{\tx{mid},\vee}_{E,N} \to \Z[\Gamma_{E/F}] \to \frac{\Z/N\Z[\Gamma_{E/F}]}{\Z/N\Z} \to 0,
\end{equation}
where now the last map is the natural projection.

The map $c^\tx{iso} : M_{E,N}^\tx{mid} \to \Z$ defined in \eqref{eq:loccompmap}
dualizes to the inclusion map $\Z = [M^{\tx{mid},\vee}_{E,N}]^\Gamma \to
M^{\tx{mid},\vee}_{E,N}$.
If $[E:F]$ divides $N$ then its splitting $s^\tx{iso} : \Z \to M_{E,N}^\tx{mid}$
defined in \eqref{eq:siso} dualizes to $M^{\tx{mid},\vee}_{E,N} \to \Z$ given by
$y \mapsto [E:F]^{-1}\sum_\sigma y(\sigma)$.

The inflation map $M_{E,N}^\tx{mid} \to M_{K,M}^\tx{mid}$ dualizes to the map sending $y^K \in M_{K,M}^{\tx{mid},\vee} \subset \Z[\Gamma_{K/F}]$ to $y^E \in M_{E,N}^{\tx{mid},\vee} \subset \Z[\Gamma_{E/F}]$ given by $y^E(\sigma) = \sum_{\tau \mapsto \sigma} y^K(\tau)$.


\subsection{The global modules} \label{sub:globmod}

Let $F$ be a number field, $E/F$ finite Galois extension, $S$ a (finite or infinite) set of places of $F$, $\dot S_E$ a set of lifts of the places in $S$ to places of $E$. We assume that $(S,\dot S_E)$ satisfies \cite[Conditions 3.3.1]{KalGRI}. We define the following $\Gamma_{E/F}$-modules:

\begin{enumerate}
	\item $M_{E,S}^\tx{iso}:=\Z[S_E]_0$ consists of finitely supported functions $f : S_E \to \Z$ satisfying $\sum_w f(w)=0$.
	\item $M_{E,\dot S_E}^\tx{mid}$ consists of finitely supported functions $f : \Gamma_{E/F} \times S_E \to \frac{1}{[E:F]}\Z$ satisfying $\sum_w f(\sigma,w)=0$, $\sum_\sigma f(\sigma,w) \in \Z$, $\sigma^{-1}w \notin \dot S_E \Rightarrow f(\sigma,w)=0$.
	\item $M_{E,\dot S_E}^\tx{rig}$ consists of finitely supported functions $f : \Gamma_{E/F} \times S_E \to \frac{1}{[E:F]}\Z/\Z$ satisfying $\sum_w f(\sigma,w)=0$, $\sum_\sigma f(\sigma,w)=0$, $\sigma^{-1}w \notin \dot S_E \Rightarrow f(\sigma,w)=0$.
\end{enumerate}

We shall write $f^\tx{iso}$ or $f^{\tx{iso},E}$ in the first case if we want to be more precise, and use the analogous notation in the other two cases.

The module $M_{E,S}^\tx{iso}$ was defined by Tate \cite{Tate66}, and later by
Kottwitz in \cite[\S 6]{KotBG}, where it was denoted by $X_3$.
The module $M_{E,\dot S_E}^\tx{rig}$ was defined in \cite{KalGRI}, where it was
denoted by $M_{E,\dot S_E,[E:F]}$.

We define $\Gamma_{E/F}$-equivariant maps
\begin{equation} \label{eq:globcompmap}
M_{E,S}^\tx{iso} \stackrel{c^\tx{iso}}{\llw} M_{E,\dot S_E}^\tx{mid} \stackrel{c^\tx{rig}}{\lrw} M_{E,\dot S_E}^\tx{rig}
\end{equation}
by the formulas
\[ f^\tx{iso}(w)=\sum_\sigma f^\tx{mid}(\sigma,w),\qquad f^\tx{rig}(\sigma,w) =
-f^\tx{mid}(\sigma,w) \mod \Z. \]

\begin{pro} \label{pro:globcrigsurj}
  The map $c^\tx{rig}$ is surjective.
\end{pro}
\begin{proof}
  We can assume that $S \neq \emptyset$.
  Let $f^\tx{rig} \in M^\tx{rig}_{E, \dot{S}_E}$.
  For each $\sigma \in \Gamma_{E/F}$ choose $w_{\sigma} \in S_E$ such that
  $\sigma^{-1} w_{\sigma} \in \dot{S}_E$.
  Define $f^\tx{mid}$ as follows:
  \begin{enumerate}
    \item For $(\sigma, w)$ such that $\sigma^{-1} w \not\in \dot{S}_E$,
      $f^\tx{mid}(\sigma, w) = 0$.
    \item For $(\sigma, w)$ such that $\sigma^{-1} w \in \dot{S}_E$ but $w \neq
      w_{\sigma}$, choose an arbitrary lift $f^\tx{mid}(\sigma, w) \in \Q$ of
      $- f^\tx{rig}(\sigma, w) \in \Q/\Z$.
    \item Finally for $\sigma \in \Gamma_{E/F}$ let $f^\tx{mid}(\sigma,
      w_{\sigma}) = - \sum_{w \in S_E \smallsetminus \{ w_\sigma \}}
      f^\tx{mid}(\sigma, w) \in \Q$.
  \end{enumerate}
  Then $f^\tx{mid} \in M^\tx{mid}_{E, \dot{S}_E}$ and $c^\tx{rig}(f^\tx{mid}) =
  f^\tx{rig}$.
\end{proof}

\begin{fct} \label{fct:globcrigind}
The kernel of $c^\tx{rig}$ is an induced $\Gamma_{E/F}$-module.
\end{fct}
\begin{proof}
After making the change of variables $\phi(\sigma,w)=f(\sigma,\sigma w)$ we see that this kernel is given by the set of functions $\phi : \Gamma_{E/F} \times S_E \to \Z$ satisfying $\sum_w \phi(\sigma,w)=0$ and $w \notin \dot S_E\Rightarrow \phi(\sigma,w)=0$, with $\Gamma_{E/F}$ acting by left translation on the first factor.
This $\Gamma_{E/F}$-module is isomorphic to
$\tx{Ind}_{\{1\}}^{\Gamma_{E/F}}\Z[S]_0$.
\end{proof}

For a $\Z[\Gamma_{E/F}]$-module $Y$ denote $I_{E/F}(Y) = \sum_{\sigma \in
\Gamma_{E/F}} (\sigma-1)(Y)$.

\begin{lem} \label{lem:repr_Sdot_supp}
	Assume that for any $\sigma \in \Gamma_{E/F}$ there exists $w \in \dot{S}_E$
	such that $\sigma w = w$. For any $\Z[\Gamma_{E/F}]$-module $Y$ we have
	$Y[S_E]_0 = Y[\dot{S}_E]_0 + I_{E/F}(Y[S_E]_0)$.
\end{lem}
\begin{proof}
  Let $f : S_E \rightarrow Y$ be such that $\sum_{w \in S_E} f(w) = 0$.
  For each $w \in S_E \smallsetminus \dot S_E$ choose $\sigma_w \in
  \Gamma_{E/F}$ such that $\sigma_w w \in \dot S_E$ and $\dot{v}_w \in
  \dot{S}_E$ such that $\sigma_w \dot{v}_w = \dot{v}_w$.
  Then $f + \sum_{w \in S_E \smallsetminus \dot{S}_E} (\sigma_w - 1) (f(w)
  \delta_w - f(w) \delta_{\dot{v}_w})$ is supported on $\dot{S}_E$.
\end{proof}

\begin{pro} \label{pro:globcisosplit}
	Assume that for any $\sigma \in \Gamma_{E/F}$ there exists $w \in \dot{S}_E$
	such that $\sigma w = w$. The morphism $c^\tx{iso}$ of
	$\Z[\Gamma_{E/F}]$-modules splits.
\end{pro}
\begin{proof}
  It is enough to show that for any $\Gamma_{E/F}$-module $X$ that is a finitely
  generated free abelian group, the map $\tx{Hom}(X, M^\tx{mid}_{E,
  \dot{S}_E})^{\Gamma_{E/F}} \rightarrow \tx{Hom}(X,M^\tx{iso}_{E,
  S_E})^{\Gamma_{E/F}}$ induced by $c^\tx{iso}$ is surjective, for then we can
  take $X=M^\tx{iso}_{E, S_E}$ and lift the identity map.
  Writing $Y = \tx{Hom}_\Z(X,\Z)$ we have $\tx{Hom}(X,M^\tx{iso}_{E, S_E})
  =Y[S_E]_0^{\Gamma_{E/F}}$.
  Let $f^\tx{iso} \in Y[S_E]_0^{\Gamma_{E/F}}$.
  By Lemma \ref{lem:repr_Sdot_supp} we can write $f^\tx{iso} = [E:F]^{-1}
  N_{E/F}(\dot{f})$ where $\dot{f} \in Y[\dot{S}_E]_0$.
  Define
  \[ f^\tx{mid}(\sigma, w) =
  \begin{cases}
    [E:F]^{-1} \sigma(\dot{f}(\sigma^{-1} w)) & \text{ if } \sigma^{-1} w \in
      \dot{S}_E \\
    0 & \text{ otherwise.}
  \end{cases} \]
\end{proof}

\begin{cor} \label{cor:globcisosurj}
Under the assumption of the proposition the map $c^\tx{iso}$ is
surjective.
\end{cor}


\subsection{Global inflation maps} \label{sub:globinf}

We keep the notation of \S\ref{sub:globmod}.
Let $K/F$ be a finite Galois extension with $E \subset K$, $S'$ a set of places
of $F$ containing $S$, $\dot S'_K$ a set of lifts of $S'$ to places of $K$ such
that for each $v \in S$ with lift $\dot v \in \dot S'_K$, the image of $\dot v$
in $S_E$ lies in $\dot S_E$.

We define three inflation maps. First assume $S=S'$.

\begin{enumerate}
  \item $\Z[S_E]_0 \to \Z[S'_K]_0$ by $f^{\tx{iso},K}(u)=[K_u:E_w]f^{\tx{iso},E}(w)$, where $w \in S_E$ is the unique place under $u \in S_K$.
  \item $M_{E,\dot S_E}^\tx{mid} \to M_{K,\dot S_K}^\tx{mid}$ by $f^{\tx{mid},K}(\sigma,u)=f^{\tx{mid},E}(\sigma,w)$ provided $\sigma^{-1}u \in \dot S_K$, and $f^{\tx{mid},K}(\sigma,u)=0$ otherwise.
  \item $M_{E,\dot S_E}^\tx{rig} \to M_{K,\dot S_K}^\tx{rig}$ by $f^{\tx{rig},K}(\sigma,u)=f^{\tx{rig},E}(\sigma,w)$ provided $\sigma^{-1}u \in \dot S_K$, and $f^{\tx{rig},K}(\sigma,u)=0$ otherwise.
\end{enumerate}

We now drop the assumption $S=S'$ and extend all maps defined above from $S_K$ to $S'_K$ by zero outside of $S_K$.

These maps are well-defined and fit in the following commutative diagram
\[ \xymatrix{
  M_{E,S}^\tx{iso}\ar[d]&M_{E,\dot S_E}^\tx{mid}\ar[l]_{c^\tx{iso}_{E,S}}\ar[r]^{c^\tx{rig}_{E,S}}\ar[d]&M_{E,\dot S_E}^\tx{rig}\ar[d]\\
  M_{K,S'}^\tx{iso}&M_{K,\dot S'_K}^\tx{mid}\ar[l]^{c^\tx{iso}_{K,S'}}\ar[r]_{c^\tx{rig}_{K,S'}}&M_{K,\dot S'_K}^\tx{rig}\\
}
\]
The commutativity of the right square is immediate. The commutativity of the left square follows from the condition that if $(\sigma,w)$ is in the support of $f^E$ then $\sigma^{-1}w \in \dot S_E$.


\subsection{Localization maps} \label{sub:locmap}

Continue with the notation of \S\ref{sub:globmod}. Fix $w \in \dot S_E$. For each of the three global modules we define localization maps $\tx{loc}_w : f \mapsto f_w$ as follows:

\begin{enumerate}
	\item $\tx{loc}_w : M_{E,S}^\tx{iso} \to M_{E_w}^\tx{iso}$, defined by $f_w := f(w)$.
	\item $\tx{loc}_w : M_{E,\dot S_E}^\tx{mid} \to M_{E_w,[E:F]}^\tx{mid}$, defined by $f_w(\sigma) := f(\sigma,w)$.
	\item $\tx{loc}_w : M_{E,\dot S_E}^\tx{rig} \to M_{E_w,[E:F]}^\tx{rig}$,
    defined by $f_w(\sigma) := f(\sigma,w)$.
\end{enumerate}

These maps fit into the following commutative diagram
\[ \xymatrix{
	M_{E,S}^\tx{iso}\ar[d]&M_{E,\dot S_E}^\tx{mid}\ar[l]\ar[r]\ar[d]&M_{E,\dot S_E}^\tx{rig}\ar[d]\\
	M_{E_w}^\tx{iso}&M_{E_w,[E:F]}^\tx{mid}\ar[l]\ar[r]&M_{E_w,[E:F]}^\tx{rig}
}
\]
The commutativity of the right square is immediate, while that of the left is implied by the support condition and the assumption $w \in \dot S_E$.

\begin{fct}
  The localization maps are compatible with the local and global inflation maps,
  i.e.\ in the setting of \S \ref{sub:globinf}, for $? \in \{ \tx{iso},
  \tx{mid}, \tx{rig} \}$ and $w \in \dot S'_K \cap S_K$, the following diagram
  commutes.
  \[ \xymatrix{
    M_{E,\dot S_E}^\tx{?} \ar[r]\ar[d] & M_{K,\dot S'_K}^\tx{?} \ar[d] \\
    M_{E_w,[E:F]}^\tx{?} \ar[r] & M_{K_w,[K:F]}^\tx{?}
  } \]
\end{fct}
\begin{proof}
Immediate.
\end{proof}


\section{Cohomology}


\subsection{Preliminary discussion} \label{sub:prelim}

Let $F$ be a local or global field of characteristic zero.
Assume given an inverse system $(D_n)_{n \in \N}$ of diagonalizable groups
defined over $F$, with surjective transition maps, and an inverse system of
classes in $H^2(\Gamma,D_n)$.
Let $D=\varprojlim D_n$. We endow $D_n(\bar F)$ with the discrete topology and
$D(\bar F) = \varprojlim D_n(\bar F)$ with the inverse limit topology.
Then \cite[Theorem 3.5.8]{Weibel94} gives the exact sequences
\[ 1 \to R^1 \varprojlim H^0(\Gamma,D_n) \to H^1(\Gamma,D) \to \varprojlim
H^1(\Gamma,D_n) \to 1 \]
and
\[ 1 \to R^1 \varprojlim H^1(\Gamma,D_n) \to H^2(\Gamma,D) \to \varprojlim
H^2(\Gamma,D_n) \to 1. \]
If $R^1 \varprojlim H^1(\Gamma,D_n)$ vanishes, the inverse system of classes in $H^2(\Gamma,D_n)$ gives an element of $H^2(\Gamma,D)$.

Assume now that we have a class $\xi \in H^2(\Gamma,D)$ and let
\[ 1 \to D \to \mc{E} \to \Gamma \to 1 \]
be an extension belonging to the corresponding isomorphism class.
Following Kottwitz we define for any affine algebraic group $G$ defined over $F$
the set $Z^1_\tx{alg}(\mc{E},G)$ to be the set of those continuous 1-cocycles
$\mc{E} \to G(\bar F)$ whose restriction to $D$ factors as the projection $D \to
D_n$ for some $n$ followed by an algebraic homomorphism $D_n \to G$. In general this homomorphism is only defined over $\bar F$, but its $G(\bar F)$-conjugacy class is invariant under $\Gamma$.
Further, for any central algebraic subgroup $Z \subset G$ we define $Z^1(D \to
\mc{E},Z \to G)$ to consist of those elements of $Z^1_\tx{alg}(\mc{E},G)$ whose
restriction to $D$ takes image in $Z$. In that case the resulting homomorphism $D \to Z$ is defined over $F$.
Finally, we set $Z^1_\tx{bas}(\mc{E},G) = Z^1(D \to \mc{E},Z(G) \to G)$.
We also define the corresponding cohomology sets
\[ H^1(D \to \mc{E},Z \to G) \subset H^1_\tx{bas}(\mc{E},G) \subset
H^1_\tx{alg}(\mc{E},G) \]
to be the quotients by the action of $G(\bar F)$ by coboundaries, i.e.\ $g$
sends $z \in Z^1_\tx{alg}(\mc{E},G)$ to $e \mapsto g^{-1}z(e)\sigma_e(g)$, where
$\sigma_e \in \Gamma$ is the image of $e \in \mc{E}$.

A priori the set $H^1_\tx{alg}(\mc{E},G)$ depends on the choice of the particular extension $\mc{E}$ in its isomorphism class. Indeed, if $\mc{E}'$ is another extension in the same class, then choosing an isomorphism of extensions $i : \mc{E}' \to \mc{E}$ provides an isomorphism $H^1_\tx{alg}(\mc{E},G) \to H^1_\tx{alg}(\mc{E}',G)$ which depends only on the $D$-conjugacy class of $i$. The $D$-conjugacy classes of isomorphisms $\mc{E}' \to \mc{E}$ are parameterized by $H^1(\Gamma,D)$. This group acts on $H^1_\tx{alg}(\mc{E},G)$ by $\alpha \in Z^1(\Gamma,D)$, $z \in Z^1_\tx{alg}(\mc{E},G)$, $(\alpha \cdot z)(e) = z(\alpha(\sigma_e) \cdot e)$, and replacing $i$ by $\alpha \cdot i$ composes the isomorphism $H^1_\tx{alg}(\mc{E},G) \to H^1_\tx{alg}(\mc{E}',G)$ with the action of $\alpha$.

It is thus clear that when $H^1(\Gamma,D)=1$ the set $H^1_\tx{alg}(\mc{E},G)$ is independent of the choice of extension $\mc{E}$ in its isomorphism class. In fact, the weaker condition $\varprojlim H^1(\Gamma,D_n)=1$ turns out to be sufficient. Indeed, by assumption for any $z \in Z^1_\tx{alg}(\mc{E},G)$ the restriction $z|_D$ factors through the projection $D \to D_n$ for some $n$ and therefore $z(\alpha(e) \cdot e)=z_n(\alpha_n(e)) \cdot z(e)$, where $z_n : D_n \to G$ composed with $D \to D_n$ equals $z|_D$, and $\alpha_n \in H^1(\Gamma,D_n)$ is the image of $\alpha$.

Assume now that $D'_n$ is another inverse system of diagonalizable groups
defined over $F$ with surjective transition maps and that we are given
homomorphisms $D'_n \to D_n$ compatible with the transition maps.
These splice to a homomorphism $D' \to D$, where $D'=\varprojlim D$.
Assume that we are given a class $\xi' \in H^2(\Gamma,D')$ and let $\mc{E}'$ be
the corresponding extension.
If $\xi'$ maps to $\xi$ under the homomorphism $D' \to D$ then there exists a
homomorphism of extensions $\mc{E}' \to \mc{E}$.
This homomorphism induces a map $H^1_\tx{alg}(\mc{E}',G) \to
H^1_\tx{alg}(\mc{E},G)$.
As above, this map is well-defined if $\varprojlim H^1(\Gamma,D_n)=1$.

We have thus seen that the vanishing of $R^i\varprojlim H^1(\Gamma,D_n)$ for $i=0,1$ has desirable consequences. A sufficient condition for the vanishing of both of these is the following: For any $n$ there exists $m>n$ such that the map $H^1(\Gamma,D_m) \to H^1(\Gamma,D_n)$ is zero.

\begin{fct} \label{fct:infres}
  We have the inflation-restriction exact sequence of pointed sets (abelian
  groups if $G$ is abelian)
  \[ 1 \to H^1(\Gamma,G) \to H^1(D \to \mc{E},Z \to G) \to \tx{Hom}_F(D,Z) \to
  H^2(\Gamma,G), \]
  where $H^2(\Gamma,G)$ is considered only when $G$ is abelian and in this case
  the last arrow is $\phi \mapsto \phi \circ \xi$.
\end{fct}

\begin{fct} \label{fct:iso_hom_D}
For any torus $T$ with co-character module $Y$ we have the isomorphism
\begin{equation} \label{eq:iso_hom_D}
  (Y \otimes X^*(D))^{\Gamma} \to
  \tx{Hom}_\Z(\tx{Hom}_\Z(Y,\Z),X^*(D)) \to \tx{Hom}_F(D,T),
\end{equation}
where the second map is obvious and the first sends $y \otimes a$ to $f_{y\otimes a}$ defined by $f_{y\otimes a}(\varphi)=\varphi(y)\cdot a$. 
\end{fct}

\begin{fct} \label{fct:cart}
  Let $\mc{E}' \to \mc{E}$ be a morphism of extensions of $\Gamma$ as considered
  above.
  For any algebraic group $G$ and any central algebraic subgroup $Z$ the square
  \[ \xymatrix{
    H^1(D \rightarrow \mc{E}, Z \rightarrow G) \ar[r] \ar[d] &
      \tx{Hom}_F(D, Z) \ar[d] \\
    H^1(D' \rightarrow \mc{E}', Z \rightarrow G) \ar[r] & \tx{Hom}_F(D', Z)
    } \]
  is Cartesian.
\end{fct}
\begin{proof}
  This follows directly from the fact that $\mc{E}$ is generated by $D$ and the
  image of $\mc{E}'$ which have intersection the image of $D'$.
\end{proof}


\subsection{Definition of $\mb{T}^\tx{iso}$, $\mb{T}^\tx{mid}$, and $P^\tx{rig}$ in the local case}
\label{sub:def_pro_loc}

Let $F$ be a local field of characteristic zero, $E/F$ a finite Galois
extension, $N$ a natural number.

Let $\mb{T}_E^\tx{iso}$ and $\mb{T}_{E,N}^\tx{mid}$ be the tori with character modules $M_E^\tx{iso}$ and $M_{E,N}^\tx{mid}$. Let $P_{E,N}^\tx{rig}$ be the finite multiplicative group with character module $M_{E,N}^\tx{rig}$. The torus $\mb{T}^\tx{iso}_E$ is simply $\mb{G}_m$. The finite multiplicative group $P_{E,N}^\tx{rig}$ was defined in \cite[\S3.1]{KalRI}, where it was denoted by $u_{E/F,N}$.

Let $\mb{T}^\tx{iso}$ and $\mb{T}^\tx{mid}$ be the pro-tori obtained as inverse limits of the systems $\mb{T}^\tx{iso}_E$ and $\mb{T}^\tx{mid}_{E,N}$ respectively, where the transition maps are induced by the inflation maps defined in \S\ref{sub:locinf}. Let $P^\tx{rig}$ be the pro-finite multiplicative group  obtained as the inverse limit of $P^\tx{rig}_{E,N}$ in the same manner. The pro-torus $\mb{T}^\tx{iso}$ was denoted by $\mb{D}$ in \cite[\S3]{Kot85}, while the group $P^\tx{rig}$ was denoted by $u$ in \cite[\S3.1]{KalRI}.


\subsection{Definition of $\mb{T}^\tx{iso}$, $\mb{T}^\tx{mid}_{\dot V}$, and
$P^\tx{rig}_{\dot V}$ in the global case}
\label{sub:def_pro_glob}

Let $F$ be a global field, $E/F$ a finite Galois extension, $S$ a finite set of
places of $S$, $\dot S_E \subset S_E$ a set of lifts for the elements of $S$.
Let $\mb{T}_{E,S}^\tx{iso}$ and $\mb{T}_{E,\dot S_E}^\tx{mid}$ be the tori over
$F$ with character modules $M_{E,S}^\tx{iso}$ and $M_{E,\dot S_E}^\tx{mid}$.
Let $P_{E,\dot S_E}^\tx{rig}$ be the finite multiplicative group with character
module $M_{E,\dot S_E}^\tx{rig}$.
In \cite{KalGRI} this was denoted by $P_{E, \dot S_E, [E:F]}$.
Note that in \cite{KalGRI} $P_{E, \dot{S}_E}$ was used to denote $\varprojlim_N
P_{E, \dot{S}_E, N}$ where $P_{E, \dot{S}_E, [E:F]}$ is the finite
multiplicative group denoted by $P^\tx{rig}_{E, \dot{S}_E}$ in the present
paper.
Since for comparison with $?^\tx{mid}$ we usually impose that this integer $N$
equal $[E:F]$ in the present paper, we hope that this will not cause confusion.

We now choose as in \cite[\S 3.3, p.\ 306]{KalGRI} an exhaustive tower $(E_i)_{i
\geq 0}$ of finite Galois extensions of $F$, exhaustive tower of finite sets of
places of $F$, $\dot{S}_i \subset S_{i, E_i}$ a choice of lifts of $S_i$ to
places of $E_i$ so that $\dot{S}_{i+1} \subset (\dot{S}_i)_{E_{i+1}}$ and each
$(E_i/F, S_i, \dot{S}_i)$ satisfies \cite[Conditions 3.3.1]{KalGRI}.
Let $\dot V$ be the set of places of $\bar F$ defined as the inverse limit of
the sets $\dot S_i$.
It is natural to ask if it is possible to formulate a condition on $\dot V$ that
is equivalent to the fact that it arises as an inverse limit of a sequence
$(E_i,S_i,\dot S_i)$ all of whose terms satisfy \cite[Conditions 3.3.1]{KalGRI}.
This is indeed possible.
For $v \in V$ denote by $\dot{v} \in \dot{V}$ its unique lift, and by
$\Gamma_{\dot v}$ the stabilizer of $\dot v$ in $\Gamma$, i.e.\ the
decomposition subgroup at $\dot v$.

\begin{cnd} \label{cnd:dense}
  $\bigcup_{v \in V} \Gamma_{\dot v}$ is dense in $\Gamma$.	
\end{cnd}

\begin{lem}
  Let $(E_i)_{i \geq 0}$ be an exhaustive tower of finite Galois extensions of
  $F$ as in \cite[\S 3.3, p.\ 306]{KalGRI}.
  \begin{enumerate}
    \item Let  $(S_i)_{i \geq 0}$ be an exhaustive tower of finite sets of
      places of $F$, $\dot{S}_i \subset S_{i, E_i}$ a choice of lifts of $S_i$
      to places of $E_i$ so that $\dot{S}_{i+1} \subset (\dot{S}_i)_{E_{i+1}}$
      and each $(E_i/F, S_i, \dot{S}_i)$ satisfies \cite[Conditions
      3.3.1]{KalGRI}.
      Let $\dot V$ be the set of places of $\bar F$ defined as the inverse limit
      of the sets $\dot S_i$.
      Then Condition \ref{cnd:dense} holds for $\dot V$.
    \item If $\dot V$ satisfies Condition \ref{cnd:dense} we can choose a finite
      increasing sequence $(S_i)_{i \geq 0}$ of subsets of $V$ such that letting
      $\dot S_i$ be the intersection of $(S_i)_{E_i}$ with the image of $\dot V$
      in $V_{E_i}$, the tower $(E_i, S_i, \dot S_i)_{i \geq 0}$ satisfies
      \cite[Conditions 3.3.1]{KalGRI}.
  \end{enumerate}
\end{lem}
\begin{proof}
  If $(E_i/F, S_i, \dot S_i)$ satisfies \cite[Conditions 3.3.1]{KalGRI} then for
  any $i$ the image of $\bigcup_{v \in V} \Gamma_{\dot v}$ in $\Gamma_{E_i/F}$
  contains $\bigcup_{v \in S_i} \Gamma_{E_{i, \dot v} / F_v} = \Gamma_{E_i/F}$
  by the third point of \cite[Conditions 3.3.1]{KalGRI}.
  Since $\Gamma_{\bar F/E_i}$ is a basis of neighbourhoods of $1$ in
  $\Gamma$, this means that $\bigcup_{v \in V} \Gamma_{\dot v}$ is dense in
  $\Gamma$.

  The proof of the converse is similar: since all $\Gamma_{E_i/F}$ are finite
  any sufficiently large $S_i$ works.
\end{proof}

In particular, this shows that sets $\dot V$ that satisfy Condition
\ref{cnd:dense} do exist.
This condition is however not automatic. Furthermore, two sets $\dot V$ and $\dot V'$ that both satisfy Condition \ref{cnd:dense} need not be conjugate under $\Gamma$. We illustrate both of these points in the following example.

\begin{exa} \label{exa:bad_dot_V}
  Take $F=\Q$ and let $E/F$ be the extension generated by all roots of the
  polynomial $P = X^3 - X^2 + 1$.
  Then $\Gamma_{E/F} \simeq S_3$ and $E/F$ is ramified only at $23$, in fact
  $A := \Z[1/23][X] /(P)$ is finite \'etale over $\Z[1/23]$: $P' = 3X^2-2X =
  (3X-2)X$, $X$ is obviously invertible in $A$ and $(9X^2-3X-2)(3X-2) = 27P -
  23$.
  Modulo $23$ we have $P(-1/3)=0$ and $P'(-1/3) \neq 0$ and so $P$ has a root in
  $\Q_{23}$.
  In particular all decomposition subgroups of $\Gamma_{E/F}$ are Abelian.
  Fix an isomorphism $\Gamma_{E/F} \simeq S_3$.
  One can choose $\dot V$ such that every decomposition group is either trivial, or 
  generated by $(1 2)$, or generated by $(1 2 3)$, and thus $(2 3)$ does not belong to any decomposition group.

  Using the same extension $E$, we can give an example of two sets $\dot V$ and
  $\dot V'$ both satisfying Condition \ref{cnd:dense} but which are not in the
  same $\Gamma$-orbit.
  Namely, choose two places $v_1, v_2$ of $F$ such that the decomposition groups
  in $\Gamma_{E/F}$ both have order two.
  Then we can choose $\dot V_E$ and $\dot V_E'$ such that $\Gamma_{E_{\dot
  v_1}/F_{v_1}} = \Gamma_{E_{\dot v_1'}/F_{v_1}} = \Gamma_{E_{\dot
  v_2}/F_{v_2}}$ but $\Gamma_{E_{\dot v_2'}/F_{v_2}} \neq \Gamma_{E_{\dot
  v_2}/F_{v_2}}$, so that even after conjugating by $\Gamma_{E/F}$ we cannot
  have $\Gamma_{E_{\dot v_i}/F_{v_i}} = \Gamma_{E_{\dot v_i'}/F_{v_i}}$ for
  $i=1,2$ simultaneously.
\end{exa}

For the rest of the paper we fix $\dot V$ satisfying Condition \ref{cnd:dense}.

Let $\mb{T}^\tx{iso}$ be the pro-torus over $F$ obtained as the inverse limits
of $\mb{T}^\tx{iso}_{E, S}$ over all pairs $(E,S)$ as above.
In the other two cases the result depends on $\dot V$.
For each finite Galois extension $E/F$ and each finite set of places $S$ of $F$
we let $\dot S_E = \{ \dot v |_E | v \in S\}$.
Consider the pro-torus $\mb{T}^\tx{mid}_{\dot V} = \varprojlim_{E,S}
\mb{T}^\tx{mid}_{E, \dot S_E}$ and the pro-finite group scheme $P^\tx{rig}_{\dot
V} = \varprojlim_{E,S} P^\tx{rig}_{E, \dot S_E}$.
Note that $M^\tx{mid}_{\dot V} = X^*(\mb{T}^\tx{mid}_{\dot V}) =
\varinjlim_{E,S} M^\tx{mid}_{E, \dot S_E}$ is identified with the
$\Gamma$-module of functions $\phi : \Gamma \times V \to \Q$ continuous in the
first variable and with finite support in the second variable such that for any
$\sigma \in \Gamma$, $\sum_{v \in V} \phi(\sigma, v) = 0$ and for any
Archimedean place $v \in V$, $\sum_{\tau \in \Gamma_{\dot v}} \phi(\sigma \tau,
v) \in \Z$.
This identification is obtained by mapping $f \in M^\tx{mid}_{E, \dot S_E}$ to
$\phi$ defined by $\phi(\sigma, v) = f(\sigma, \sigma \dot v)$.
This description is similar to \cite[Lemma 3.4.1]{KalGRI}.

The set of lifts $\dot V$ being fixed, for $v \in V$ we simply denote $\Gamma_v
= \Gamma_{\dot v}$.
Denote $\ol{F_v} = \varinjlim_E E_{\dot v}$ where we take the limit over all
finite extensions of $F$ in $\ol{F}$.
This is an algebraic closure of $F_v$, strictly smaller than the completion of
$\ol{F}$ for $\dot v$ is $v$ is non-Archimedean.
It is easy to check that the localization maps defined in Section
\ref{sub:locmap} induce localization maps at infinite level
\[  \mb{T}^\tx{iso}_v \rightarrow (\mb{T}^\tx{iso})_{F_v}, \ \ \mb{T}^\tx{mid}_v
\rightarrow (\mb{T}^\tx{mid}_{\dot V})_{F_v} \text{ and } P^\tx{rig}_v
\rightarrow (P^\tx{rig}_{\dot V})_{F_v} \]
where $\mb{T}^\tx{iso}_v$ (resp.\ $\mb{T}^\tx{mid}_v$, $P^\tx{rig}_v$) denotes
the pro-torus (resp.\ pro-torus, pro-finite groupe scheme) over $F_v$ defined in
Section \ref{sub:def_pro_loc} for the local field $F_v$ together with its
algebraic closure $\ol{F_v}$, and a subscript $F_v$ denotes base change from $F$
to $F_v$.
We will denote these three localization maps by $\tx{loc}_v$.


\subsection{The maps $\mb{T}^\tx{iso} \to \mb{T}^\tx{mid} \lw P^\tx{rig}$}

The maps $c^\tx{rig}$ and $c^\tx{iso}$ defined in the local case in
\S\ref{sub:locmod} and in the global case in \S\ref{sub:globmod} splice together
to define Cartier dual maps $c_\tx{iso} : \mb{T}^\tx{iso} \to \mb{T}^\tx{mid}$
and $c_\tx{rig} : P^\tx{rig}_{\dot V} \to \mb{T}^\tx{mid}_{\dot V}$.

Let $F$ be a local or global field.

\begin{pro} \label{pro:surj}
Let $T$ be an algebraic torus defined over $F$.
\begin{enumerate}
  \item The map $c_\tx{iso}: \mb{T}^\tx{iso} \to \mb{T}^\tx{mid}_{\dot V}$ is
    injective.
    The homomorphism $\tx{Hom}_F(\mb{T}^\tx{mid}_{\dot V},T) \to
    \tx{Hom}_F(\mb{T}^\tx{iso},T)$ is surjective.
  \item The map $c_\tx{rig} : P^\tx{rig}_{\dot V} \to \mb{T}^\tx{mid}_{\dot V}$
    is injective.
    The homomorphism $\tx{Hom}_F(\mb{T}^\tx{mid}_{\dot V},T) \to
    \tx{Hom}_F(P^\tx{rig}_{\dot V},T)$ is surjective.
\end{enumerate}
\end{pro}
\begin{proof}
In the local case the injectivity claims follow from Fact \ref{fct:loccompsurj}, while in the global case they follow from Corollary \ref{cor:globcisosurj} and Proposition \ref{pro:globcrigsurj}.

We prove the surjectivity of $\tx{Hom}_F(\mb{T}^\tx{mid}_{\dot V},T) \to \tx{Hom}_F(\mb{T}^\tx{iso},T)$.
In the local case, it follows immediately from the existence of the splitting
\eqref{eq:siso}.
In the global case, Proposition \ref{pro:globcisosplit} implies that
$\tx{Hom}_F(\mb{T}^\tx{mid}_{E,\dot S_E},T) \to
\tx{Hom}_F(\mb{T}^\tx{iso}_{E,S},T)$ is surjective for any $E, \dot S_E$. The surjectivity of $\tx{Hom}_F(\mb{T}^\tx{mid}_{\dot V},T) \to \tx{Hom}_F(\mb{T}^\tx{iso},T)$ follows by taking
the colimit over $E,\dot S_E$.

We prove the surjectivity of $\tx{Hom}_F(\mb{T}^\tx{mid}_{\dot V},T) \to
\tx{Hom}_F(P^\tx{rig}_{\dot V},T)$. Consider first the global case. Let $X=X^*(T)$. We claim that every $\Z[\Gamma]$-homomorphism $f : X \to M^\tx{rig}_{E,\dot S_E}$ lifts to a homomorphism $\dot f : X \to M^\tx{mid}_{E,\dot S_E}$. Since $X$ is $\Z$-free we can choose a lift $\ddot f : X \to M^\tx{mid}_{E,\dot S_E}$ that is a homomorphism of $\Z$-modules, but not necessarily $\Gamma$-equivariant.
Then $\sigma \mapsto \ddot f - \sigma(\ddot f)$ is a 1-cocycle of $\Gamma$ in
$\tx{Hom}_\Z(X,\tx{Ker}(c^\tx{rig}_{E,\dot S_E}))$.
By Fact \ref{fct:globcrigind} and \cite[Chap. IX,\S3,Prop]{SerLF} this is an
induced $\Gamma$-module, so $\sigma \mapsto \ddot f - \sigma(\ddot f)$ is a
coboundary, implying that there exists a $\Gamma$-equivariant lift $\dot f$ of
$f$.
This completes the proof in the global case. The proof in the local case is the
same, but now based on Fact \ref{fct:loccompsurj} in place of Fact
\ref{fct:globcrigind}.
\end{proof}

\begin{pro} \label{pro:ciso_splits_infty}
  In the global case, the map $c_\tx{iso} : \mb{T}^\tx{iso} \to
  \mb{T}^\tx{mid}_{\dot V}$ splits.
\end{pro}
\begin{proof}
  We seek a compatible family of splittings of $c_\tx{iso} :
  \mb{T}^\tx{iso}_{E_i} \to \mb{T}^\tx{mid}_{E_i, \dot V_{E_i}}$.
  As we saw in the proof of Proposition \ref{pro:globcisosplit}, giving such a
  splitting is equivalent to giving, for any torus $T$ defined over $F$ and
  split by $E_i$, a splitting $s_i$ of
  \[ (Y \otimes M^\tx{mid}_{E_i, \dot V_{E_i}})^\Gamma =
  \tx{Hom}(\mb{T}^\tx{mid}_{E_i, \dot V_{E_i}}, T) \to \tx{Hom}(\mb{T}^\tx{iso},
  T) = (Y[V_{E_i}]_0)^\Gamma, \]
  where $Y = X_*(T)$, which is functorial in $T$.

  It is convenient to let $E_{-1} = F$.
  For $k \geq 0$ and $v \in V$ choose $R_{k,v} \subset \Gamma_{\ol{F}/E_{k-1}}$
  representing $\Gamma_{E_k/E_{k-1}} / \Gamma_{E_{k, \dot v} / E_{k-1, \dot
  v}}$.
  For $w \in V_{E_k} \smallsetminus \dot{V}_{E_k}$ such that $w_{E_{k-1}} \in
  \dot{V}_{E_{k-1}}$, let $r(k, w) \in R_{k, w_F}$ be the element such that
  $r(k, w)^{-1} w \in \dot{V}_{E_k}$, and choose $\dot{v}(k, w) \in \dot{V}$
  such that the image of $r(k, w)$ in $\Gamma_{E_k / E_{k-1}}$ belongs to the
  decomposition subgroup for $\dot{v}(k, w)_{E_k}$.
  For $i \geq k \geq 0$ denote $\dot{V}_{i,k} = \{ w \in V_{E_i} | w_{E_k} \in
  \dot{V}_{E_k} \}$.
  For $f \in Y[\dot{V}_{i, k-1}]$ let
  \[ \pi_{i,k}(f) = f + \sum_{w \in \dot{V}_{i, k-1} \smallsetminus \dot{V}_{i,
  k}} (r(k, w) - 1)( f(w) \delta_w - f(w) \delta_{\dot{v}(k, w)_{E_i}}). \]
  It is clear that $\pi_{i,k}(f)$ is supported on $\dot{V}_{i,k}$ and that
  $\pi_{i,k}(f)-f \in I(Y[V_{E_i}]_0)$.
  Define $\pi_i := \pi_{i,i} \circ \pi_{i,i-1} \circ \dots \circ \pi_{i,0} :
  Y[V_{E_i}] \to Y[\dot{V}_{E_i}]$.

  As in \cite[\S 8.3]{KotBG} we denote $p : Y[S_{E_i}] \to Y[S_{E_{i+1}}]$ for
  the inflation map and $j : Y[S_{E_{j+1}}] \to Y[S_{E_j}]$ defined by $j(f)(w)
  = \sum_{u | w} f(u)$, i.e.\ $j(\delta_u) = \delta_{u_{E_i}}$.
  They are both $\Gamma_{E_{i+1}/F}$-equivariant and satisfy $j \circ p =
  [E_{i+1} : E_i]$.
  It is easy to check that for $f \in Y[\dot{V}_{i+1,k-1}]$ we have
  \[ j(\pi_{i+1,k}(f)) = \begin{cases}
    \pi_{i,k}(j(f)) & \text{ if } k \leq i \\
    j(f) & \text{ if } k = i+1
  \end{cases} \]
  and thus $j \circ \pi_{i+1} = \pi_i \circ j$.
  Now $\pi_{i+1} p(f)$ is supported on $\dot{V}_{E_{i+1}}$ and satisfies $j
  \pi_{i+1} p (f) = [E_{i+1}:E_i] \pi_i(f)$, and so for $w \in V_{E_{i+1}}$ we
  have
  \begin{equation} \label{eq:split_infty}
    \pi_{i+1} p(f)(w) = [E_{i+1}:E_i] \pi_i(f)(w_{E_i}).
  \end{equation}
  Now we can resume the proof of Proposition \ref{pro:globcisosplit} with
  $\dot{f} = \pi_i(f)$ for $f \in (Y[V_{E_i}]_0)^\Gamma)$, defining $s_i(f) \in
  (Y \otimes M^\tx{mid}_{E_i, \dot{V}_{E_i}})^\Gamma$ by
  \[ s_i(f)(\sigma, w) = \begin{cases}
      [E_i:F]^{-1} \sigma(\pi_i(f)(\sigma^{-1} w)) & \text{ if } \sigma^{-1} w
      \in \dot{V}_{E_i} \\
      0 & \text{ otherwise}
  \end{cases} \]
  for $\sigma \in \Gamma_{E_i/F}$ and $w \in V_{E_i}$.
  Now \eqref{eq:split_infty} implies that $s_{i+1}(p(f)) \in (Y \otimes
  M^\tx{mid}_{E_{i+1}, \dot{V}_{E_{i+1}}})^\Gamma$ is the inflation of
  $s_i(f) \in (Y \otimes M^\tx{mid}_{E_i, \dot{V}_{E_i}})^\Gamma$.
\end{proof}

\begin{rem} \label{rem:no_can_rig_iso}
  Composing such a splitting $\mb{T}^\tx{mid}_{\dot V} \to \mb{T}^\tx{iso}$ with
  $c_\tx{rig}$, we obtain a map $P^\tx{rig}_{\dot V} \to \mb{T}^\tx{iso}$.
  Unfortunately, this splitting is not canonical, since we had to
  choose sets of representatives $R_{k,v}$.
  Therefore we cannot use it to compare $?^\tx{iso}$ and $?^\tx{rig}$ directly,
  as in the local case, which is why $?^\tx{mid}$ was introduced. 
\end{rem}


\subsection{Review of the Tate-Nakayama isomorphism} \label{sub:tniso}

Let $E/F$ be a Galois extension of local fields of characteristic zero, $T$ an
algebraic torus defined over $F$ and split over $E$, $Y=X_*(T)$.
We have the Tate-Nakayama isomorphism $\hat H^i(\Gamma_{E/F},Y) \to \hat H^{i+2}(\Gamma_{E/F},T(E))$ defined by cup product against the fundamental class in $H^2(\Gamma_{E/F},E^\times)$. Combining with the inflation $H^i(\Gamma_{E/F},T(E)) \to H^i(\Gamma,T)$ we obtain the isomorphism $\hat H^{-1}(\Gamma_{E/F},Y) \to H^1(\Gamma,T)$ and the inclusion $\hat H^0(\Gamma_{E/F},Y) \to H^2(\Gamma,T)$.

Given a finite multiplicative group $Z$ defined over $F$ and split over $E$, we let $A=X^*(Z)$ and $A^\vee=\tx{Hom}(A,\Q/\Z)$, and then have the injective map $\hat H^{-1}(\Gamma_{E/F},A^\vee) \to H^2(\Gamma,Z)$ denoted by $\Theta_{E,v}$ in \cite[\S3.2]{KalGRI}.

If $1 \to Z \to T \to \bar T \to 1$ is an exact sequence of diagonalizable
groups defined over $F$ and split over $E$, where $Z$ is finite and $T$ and
$\bar T$ are tori, then these maps fit in the following commutative diagram,
which is the local analog of Lemma \cite[Lemma 3.2.5]{KalGRI} whose proof is
easier and shall be omitted:
\begin{equation} \label{eq:tnzdiagloc}
\xymatrix{
\hat H^{-1}(\Gamma_{E/F},Y)\ar[r]^\cong_{\tx{TN}}\ar[d]&H^1(\Gamma_{E/F},T(E))\ar[r]^\cong\ar[d]&H^1(\Gamma_S,T(\bar F))\ar[d]\\
\hat H^{-1}(\Gamma_{E/F},\bar Y)\ar[r]^\cong_{\tx{TN}}\ar[d]&H^1(\Gamma_{E/F},\bar T(E))\ar[r]^\cong&H^1(\Gamma_S,\bar T(\bar F))\ar[d]\\
\hat H^{-1}(\Gamma_{E/F},A^\vee)\ar[rr]^{\Theta_{E}}\ar[d]&&H^2(\Gamma_S,Z(\bar F))\ar[d]\\
\hat H^{0}(\Gamma_{E/F},Y)\ar[r]^-\cong_-{-\tx{TN}}\ar[d]&H^2(\Gamma_{E/F},T(E))\ar@{^(->}[r]\ar[d]&H^2(\Gamma_S,T(\bar F))\ar[d]\\
\hat H^{0}(\Gamma_{E/F},\bar Y)\ar[r]^-\cong_-{-\tx{TN}}&H^2(\Gamma_{E/F},\bar T(E))\ar@{^(->}[r]&H^2(\Gamma_S,\bar T(\bar F))\\^{}
}
\end{equation}

Let $E/F$ be a Galois extension of number fields, $S$ a finite set of places of
$F$ containing all Archimedean places and all finite places ramifying in $E/F$
and such that every ideal class of $E$ contains an ideal with support in $S$
(i.e.\ \cite[Conditions 3.1.1]{KalGRI}).
Given an algebraic torus $T$ defined over $F$ and split over $E$ we let $Y=X_*(T)$ and $Y[S_E]_0=Y \otimes \Z[S_E]_0$. We have the Tate-Nakayama isomorphism $\hat H^i(\Gamma_{E/F},Y[S_E]_0) \to \hat H^{i+2}(\Gamma_{E/F},T(O_{E,S}))$ defined by cup product against the fundamental class in $H^2(\Gamma_{E/F},\tx{Hom}_\Z(\Z[S_E]_0,O_{E,S}^\times))$ defined in \cite{Tate66}.
Combining with the inflation $H^i(\Gamma_{E/F},T(O_{E,S})) \to
H^i(\Gamma_S,T(O_S))$ we obtain the isomorphism
$\hat H^{-1}(\Gamma_{E/F},Y) \to H^1(\Gamma_S,T(O_S))$ and the inclusion $\hat
H^0(\Gamma_{E/F},Y) \to H^2(\Gamma_S,T(O_S))$, see \cite[Lemma 3.1.9]{KalGRI}.

Given a finite multiplicative group $Z$ defined over $F$ and split over $E$, we let $A=X^*(Z)$ and $A^\vee=\tx{Hom}(A,\Q/\Z)$, and then have the injective map $\hat H^{-1}(\Gamma_{E/F},A^\vee[S_E]_0) \to H^2(\Gamma,Z)$ denoted by $\Theta_{E,S}$ in \cite[\S3.2]{KalGRI}.

If $1 \to Z \to T \to \bar T \to 1$ is an exact sequence of diagonalizable groups defined over $F$ and split over $E$, where $Z$ is finite and $T$ and $\bar T$ are tori, and the order of $Z$ is an $S$-unit, then these maps fit in the following commutative diagram according to Lemma \cite[Lemma 3.2.5]{KalGRI}:

\begin{equation} \label{eq:tnzdiagglob}
\xymatrix{
\hat H^{-1}(\Gamma_{E/F},Y[S_E]_0)\ar[r]^\cong_{\tx{TN}}\ar[d]&H^1(\Gamma_{E/F},T(O_{E,S}))\ar[r]^\cong\ar[d]&H^1(\Gamma_S,T(O_S))\ar[d]\\
\hat H^{-1}(\Gamma_{E/F},\bar Y[S_E]_0)\ar[r]^\cong_{\tx{TN}}\ar[d]&H^1(\Gamma_{E/F},\bar T(O_{E,S}))\ar[r]^\cong&H^1(\Gamma_S,\bar T(O_S))\ar[d]\\
\hat H^{-1}(\Gamma_{E/F},A^\vee[S_E]_0)\ar[rr]^{\Theta_{E,S}}\ar[d]&&H^2(\Gamma_S,Z(O_S))\ar[d]\\
\hat H^{0}(\Gamma_{E/F},Y[S_E]_0)\ar[r]^\cong_{-\tx{TN}}\ar[d]&H^2(\Gamma_{E/F},T(O_{E,S}))\ar@{^(->}[r]\ar[d]&H^2(\Gamma_S,T(O_S))\ar[d]\\
\hat H^{0}(\Gamma_{E/F},\bar Y[S_E]_0)\ar[r]^\cong_{-\tx{TN}}&H^2(\Gamma_{E/F},\bar T(O_{E,S}))\ar@{^(->}[r]&H^2(\Gamma_S,\bar T(O_S))\\^{}
}
\end{equation}

Consider now a morphism $Z \to T$ from a finite multiplicative group $Z$ to a torus $T$, not assumed injective.
It induces a homomorphism $H^2(\Gamma,Z) \to H^2(\Gamma,T)$ in the local case, and $H^2(\Gamma_S,Z(O_S)) \to H^2(\Gamma_S,T(O_S))$ in the global case. We shall now define a homomorphism $H^{-1}_T(\Gamma_{E/F},A^\vee) \to H^0_T(\Gamma_{E/F},Y)$ in the local case, and $H^{-1}_T(\Gamma_{E/F},A^\vee[S_E]_0) \to
H^0_T(\Gamma_{E/F},Y[S_E]_0)$ in the global case, that intertwines the respective Tate-Nakayama homomorphisms.

We first consider the local case.
Let $\bar T$ be the cokernel of $Z \to T$, an algebraic torus that is a quotient
of $T$.
We consider the complex $X^*(\bar T) \to X^*(T) \to X^*(Z)$.
Given $\lambda : X^*(Z) \to \Q/\Z$ we compose to obtain $\lambda_T : X^*(T) \to
\Q/\Z$.
Since $X^*(T)$ is a free $\Z$-module there exists a lift $\dot\lambda_T : X^*(T)
\to \Q$ whose restriction to $X^*(\bar T)$ necessarily takes image in $\Z$ and
thus is an element of $X_*(\bar T)$, well-defined modulo $X_*(T)$.
Consider $N\dot\lambda_T = \sum_{\sigma \in \Gamma_{E/F}} \sigma(\dot\lambda_T)
\in X_*(\bar T)^\Gamma$.
If we assume that $N\lambda=0$ then we see that $N\dot\lambda_T$ belongs to the
sublattice $Y = X_*(T) \subset X_*(\bar T)$.
This defines a map $H^{-1}_T(\Gamma_{E/F},A^\vee) \to H^0_T(\Gamma_{E/F},Y)$.

In the global case, the definition of $H^{-1}_T(\Gamma_{E/F},A^\vee[S_E]_0) \to
H^0_T(\Gamma_{E/F},Y[S_E]_0)$ is analogous.
Now instead of $\lambda : A \to \Q/\Z$ we have $\lambda : A \times S_E \to
\Q/\Z$ that is a homomorphism in the first variable, and with $\sum_w
\lambda(a,w)=0$.
We can choose a lift $\dot\lambda_T : X^*(T) \times S_E \to \Q$ which is also a
homomorphism in the first variable and with $\sum_w \dot\lambda(x,w)=0$, i.e.\
$\dot \lambda_T \in \Q Y[S_E]_0$.
Again $\dot \lambda_T$ is well-defined modulo $Y[S_E]_0$, and if $N\lambda=0$
then $N\dot\lambda_T \in Y[S_E]_0$.

The fact that these maps are compatible with the Tate-Nakayama homomorphisms is
proved as follows.
Define $Z' \subset T$ to be the image of $Z$ and write $Z \to T$ as the composition of the surjective homomorphism $Z \to Z'$ and the injective homomorphism $Z' \to T$. For $Z \to Z'$ one applies the functoriality of the Tate-Nakayama homomorphism for finite multiplicative groups, and for $Z' \to T$ one uses \cite[Lemma 3.2.5]{KalGRI} and its local analog.


\subsection{The local gerbes $\mc{E}^\tx{iso}$, $\mc{E}^\tx{mid}$, $\mc{E}^\tx{rig}$}
\label{sub:loc_gerbes}

Let $F$ be a local field of characteristic zero, $E/F$ a finite Galois
extension, $N$ a natural number.

There is a canonical element $\xi^\tx{iso}_E \in H^2(\Gamma,\mb{T}^\tx{iso}_E(\bar
F))$: it is the element that, under the identification $\mb{T}^\tx{iso}_E =
\mb{G}_m$ and the invariant map of local class field theory $H^2(\Gamma,\bar
F^\times) \xrightarrow{\sim} \Q/\Z$ corresponds to $[E:F]^{-1}$, i.e.\ the
inflation of the canonical class in $H^2(\Gamma_{E/F}, E^{\times})$.

There is also a canonical element $\xi^\tx{rig}_{E,N} \in
H^2(\Gamma,P^\tx{rig}_{E,N})$.
It is obtained by taking $-1 \in \Z$, using the identification $\hat\Z =
H^2(\Gamma,u)$ of \cite[Theorem 3.1]{KalRI}, and mapping this class under the
map $u \to u_{E/F,N}$.

\begin{pro} \label{pro:loc_can_agree}
  The images of the canonical classes $\xi^\tx{iso}_E$ and $\xi^\tx{rig}_{E,N}$
  under
  \[ H^2(\Gamma,\mb{T}^\tx{iso}_{E}(\bar F)) \rw
  H^2(\Gamma,\mb{T}^\tx{mid}_{E,N}(\bar F)) \lw
  H^2(\Gamma,P^\tx{rig}_{E,N}(\bar F))\]
  are equal.
  We denote this common image by $\xi^\tx{mid}_{E,N}$.
\end{pro}
\begin{proof}
We use Tate-Nakayama duality to describe the canonical elements.
For any algebraic torus $T$ defined over $F$ and split over $E$ we have the
isomorphism $\hat H^0(\Gamma_{E/F},X_*(T)) \to H^2(\Gamma_{E/F},T(E))$ of
\cite{Tate66}, which is cup-product with the canonical class, and the inclusion
$H^2(\Gamma_{E/F},T(E)) \to H^2(\Gamma,T(\bar F))$.
We apply this with $T=\mb{T}^\tx{iso}_E$ and $T=\mb{T}^\tx{mid}_{E,N}$. For a finite multiplicative group $D$ we write $X_*(D)=\tx{Hom}_\Z(X^*(D),\Q/\Z)$. Then we have the injective homomorphism $\hat H^{-1}(\Gamma_{E/F},X_*(D)) \to H^2(\Gamma,D)$ denoted by $\Theta_{E,v}$ in \cite[\S3.2]{KalGRI}. We apply this with $D=P^\tx{rig}_{E,N}$.

We now have the following elements:
\begin{enumerate}
  \item $1 \in \Z$, representing an element of $\Z/[E:F]\Z=\hat H^0(\Gamma_{E/F},\Z)$.
  \item The constant function with value $1$ in $M^{\tx{mid},\vee}_{E,N}$, representing an element of $\hat H^0(\Gamma_{E/F},M^{\tx{mid},\vee}_{E,N})$.
  \item The function $\delta_e \in \frac{\Z/N\Z[\Gamma_{E/F}]}{\Z/N\Z}$ representing an element of $\hat H^{-1}(\Gamma_{E/F},-)$.
\end{enumerate}
The image of the first element is the canonical class $\xi^\tx{iso}_E \in
H^2(\Gamma,\mb{T}_E^\tx{iso})$, while the image of the third element is the
canonical class $\xi^\tx{rig}_{E,N} \in H^2(\Gamma,P^\tx{rig}_{E,N})$.

It is clear that the maps $\Z \to M^{\tx{mid},\vee}_{E,N} \to \Z$ identify the
the first two elements. As for the second and third element, we consider the
exact sequence \eqref{eq:locmides4} that is dual to \eqref{eq:locmides2} and see
that the (-1)-cocycle $\delta_e$ in $\frac{\Z/N\Z[\Gamma_{E/F}]}{\Z/N\Z}$ lifts
to the (-1)-cochain $\delta_e$ in $\Z[\Gamma_{E/F}]$, whose differential, i.e.
$\Gamma_{E/F}$-norm, is the constant function $1$ in $M^{\tx{mid},\vee}_{E,N}$.
Since the map $M^{\tx{mid},\vee}_{E,N} \to N^{-1}\Z/\Z[\Gamma_{E/F}]_0$ in
\eqref{eq:locmides2} is the negative of $c^\tx{rig}$, the claim now follows from
the functoriality of the Tate-Nakayama isomorphism and its anticommutativity
between degrees $-1$ and $0$, i.e.\ the discussion of \S\ref{sub:tniso}.
\end{proof}

Recall from \cite[Theorem 3.5.8]{Weibel94} that the finiteness of
all $H^1(\Gamma,\mb{T}^\tx{mid}_{E,N})$ implies that
$H^2(\Gamma,\mb{T}^\tx{mid}) = \varprojlim H^2(\Gamma,\mb{T}^\tx{mid}_{E,N})$,
and similarly for $\mb{T}^\tx{iso}$ and $P^\tx{rig}$.
We therefore have a unique class $\xi^\tx{iso} \in H^2(\Gamma, \mb{T}^\tx{iso})$
(resp.\ $\xi^\tx{rig} \in H^2(\Gamma, P^\tx{rig})$) lifting
$(\xi^\tx{iso}_E)_{E}$ (resp.\ $(\xi^\tx{rig}_{E,N})_{E,N}$).

\begin{cor} \label{cor:loc_can_agree}
\ \\[-24pt]\begin{enumerate}
  \item Let $K/E/F$ be a tower of finite Galois extensions and $N|M$ natural
    numbers.
    The inflation map $H^2(\Gamma,\mb{T}^\tx{mid}_{K,M}(\bar F)) \to
    H^2(\Gamma,\mb{T}^\tx{mid}_{E,N}(\bar F))$ maps $\xi^\tx{mid}_{K,M}$ to
    $\xi^\tx{mid}_{E,N}$ and therefore yields a canonical class $\xi^\tx{mid} \in
    H^2(\Gamma,\mb{T}^\tx{mid})$.
  \item The images of the canonical classes $\xi^\tx{iso}$ and $\xi^\tx{rig}$
    under
    \[ H^2(\Gamma,\mb{T}^\tx{iso}(\bar F)) \rw H^2(\Gamma,\mb{T}^\tx{mid}(\bar
      F)) \lw H^2(\Gamma,P^\tx{rig}(\bar F))\]
    are equal to $\xi^\tx{mid}$.
\end{enumerate}
\end{cor}
\begin{proof}
  Both points follow from Proposition \ref{pro:loc_can_agree} and the
  compatibility of the canonical classes for the tori $\mb{T}^\tx{iso}_E$.
\end{proof}

\begin{lem} \label{lem:h1lmvan}
  We have $R^i \varprojlim H^1(\Gamma,\mb{T}^\tx{mid}_{E,N})=0$ for $i=0,1$.
\end{lem}
\begin{proof}
  It is enough to show that for each $E,N$ we can find $M$ such that the map
  $\mb{T}^\tx{mid}_{E,M} \to \mb{T}^\tx{mid}_{E,N}$ induces the zero map on
  $H^1$.
  Let $\mb{T}^\circ_{E,N}$ be the cokernel of the injective map $c_\tx{iso} :
  \mb{T}^\tx{iso}_E \to \mb{T}^\tx{mid}_{E,N}$, so that $M^\circ_{E,N} :=
  X^*(\mb{T}^\circ_{E,N})$ is the group of maps $f : \Gamma_{E/F} \to
  \frac{1}{N} \Z$ satisfying $\sum_{\sigma} f(\sigma) = 0$.
  The vanishing of $H^1(\Gamma,\mb{T}^\tx{iso}_E)$ implies that
  $H^1(\Gamma,\mb{T}^\tx{mid}_{E,N}) \to H^1(\Gamma,\mb{T}^\circ_{E,N})$ is
  injective, so it is enough to prove the statement with $\mb{T}^\tx{mid}$
  replaced by $\mb{T}^\circ$.
  Choose $M=[E:F]N$.
  Then the inflation map $M^\circ_{E,N} \to M^\circ_{E,M}$ can be factored as
  \[ M^\circ_{E,N} \to M^\circ_{E,N} \to M^\circ_{E,M} \]
  where the first map is multiplication by $[E:F]$ and the second map is
  division by $[E:F]$.
  Since the torus $\mb{T}^\circ_{E,N}$ splits over $E$, its cohomology groups
  are killed by multiplication by $[E:F]$ by Hilbert's Theorem 90.
\end{proof}

We define the gerbes $\mc{E}^\tx{iso}$, $\mc{E}^\tx{mid}$, and $\mc{E}^\tx{rig}$
to be the extensions of $\Gamma$ by $\mb{T}^\tx{iso}$, $\mb{T}^\tx{mid}$, and
$P^\tx{rig}$, respectively, given by the canonical classes $\xi^\tx{iso}$,
$\xi^\tx{mid}$ and $\xi^\tx{rig}$.
According to Corollary \ref{cor:loc_can_agree} there exist maps of gerbes
\begin{equation} \label{eq:maps_gerbes_loc}
  \xymatrix{
1\ar[r]&\mb{T}^\tx{iso}\ar[d]_{c_\tx{iso}}\ar[r]&\mc{E}^\tx{iso}\ar[r]\ar@{.>}[d]^{c_\tx{iso}}&\Gamma\ar[r]\ar@{=}[d]&1\\
1\ar[r]&\mb{T}^\tx{mid}\ar[r]&\mc{E}^\tx{mid}\ar[r]&\Gamma\ar[r]&1\\
1\ar[r]&P^\tx{rig}\ar[u]^{c_\tx{rig}}\ar[r]&\mc{E}^\tx{rig}\ar@{.>}[u]_{c_\tx{rig}}\ar[r]&\Gamma\ar[r]\ar@{=}[u]&1\\
}
\end{equation}
These dotted maps are not unique. In both cases, the set of
$\mb{T}^\tx{mid}$-conjugacy classes of such maps is a torsor under
$H^1(\Gamma,\mb{T}^\tx{mid})$, which by Lemma \ref{lem:h1lmvan} and
\cite[Theorem 3.5.8]{Weibel94} equals $R^1 \varprojlim
H^0(\Gamma,\mb{T}^\tx{mid}_{E,N})$. This group is uncountable by Lemma \ref{lem:r1limuncount}.
 Nonetheless, the discussion of \S\ref{sub:prelim} and Lemma \ref{lem:h1lmvan}
shows that both the set $H^1_\tx{alg}(\mc{E}^\tx{mid},G)$ and the maps
$H^1_\tx{alg}(\mc{E}^\tx{iso},G) \lw H^1_\tx{alg}(\mc{E}^\tx{mid},G) \to
H^1_\tx{alg}(\mc{E}^\tx{rig},G)$ are independent of the choice of
$\mc{E}^\tx{mid}$ within its isomorphism class and of the dotted maps
$c_\tx{iso}$ and $c_\tx{rig}$, and similarly for $H^1(\mb{T}^\tx{mid} \to
\mc{E}^\tx{mid}, Z \to G)$ etc.

Thanks to the canonical splitting $s_\tx{iso} : \mb{T}^\tx{mid} \to
\mb{T}^\tx{iso}$ of $c_\tx{iso}$, which tautologically maps $\xi^\tx{mid}$ to
$\xi^\tx{iso}$, there is also a map of gerbes $s_\tx{iso} : \mc{E}^\tx{mid} \to
\mc{E}^\tx{iso}$, well-defined up to $Z^1(\Gamma, \mb{T}^\tx{iso})$.
As above, this ambiguity disappears when considering $H^1_\tx{alg}$ groups.
If $F$ is non-Archimedean then the composition $s_\tx{iso} \circ c_\tx{rig} :
\mc{E}^\tx{rig} \to \mc{E}^\tx{iso}$ is the morphism (3.13) of
\cite{KalRIBG}.

\begin{lem} \label{lem:r1limuncount}
The groups $H^1(\Gamma,\mb{T}^\tx{mid}) = R^1 \varprojlim H^0(\Gamma,\mb{T}^\tx{mid}_{E,N})$ and $H^1(\Gamma,\mb{T}^\tx{iso}) = R^1\varprojlim H^0(\Gamma,\mb{T}^\tx{iso}_E)$ are uncountable.
\end{lem}
\begin{proof}
We treat the case of $\mb{T}^\tx{mid}$, that of $\mb{T}^\tx{iso}$ being analogous but simpler. Consider the Kottwitz homomorphism \cite[\S7]{Kot97}
\[ 1 \to \mb{T}^\tx{mid}_{E,N}(F)_0 \to \mb{T}^\tx{mid}_{E,N}(F) \to (M_{E,N}^{\tx{mid},\vee})_I^\tx{Fr} \to 0. \]
Using that $R^1\varprojlim$ is right-exact and \cite[Proposition 1.1]{Mil08} our claim is equivalent to the uncountability of $R^1\varprojlim (M_{E,N}^{\tx{mid},\vee})_I^\tx{Fr}$. The latter is a system of countable groups, so we apply \cite[Proposition 1.4]{Mil08} and reduce to showing that this system fails the Mittag-Leffler condition.

We fix a finite Galois extension $E/F$ and let $K$ traverse a co-final sequence of finite Galois extensions of $F$ containing $E$. We also fix $N$ and let $M$ traverse a co-final sequence of multiples of $N$. If the image of the inflation map
\[ (M_{K,M}^{\tx{mid},\vee})_I^\tx{Fr} \to (M_{E,N}^{\tx{mid},\vee})_I^\tx{Fr} \]
stabilizes, so would the image of its composition with the norm map for the action of $\Gamma_{E/F}$. Recall from \S\ref{sub:discmmidloc} that the inflation map is induced by the map $\Z[\Gamma_{K/F}] \to \Z[\Gamma_{E/F}]$ defined to send $y^K$ to $y^E(\sigma)=\sum_{\tau \mapsto \sigma}y^K(\tau)$. Composing this with the norm map $\Z[\Gamma_{E/F}] \to \Z$ we obtain the norm map $\Z[\Gamma_{K/F}] \to \Z$, i.e. the map sending $y^K$ to $\sum_{\tau \in \Gamma_{K/F}} y^K(\tau)$. Thus we are studying whether the image of this map, restricted to $M^{\tx{mid},\vee}_{K,M}=M\Z[\Gamma_{K/F}]+\Z \subset \Z[\Gamma_{K/F}]$, stabilizes. But the norm map sends $M\Z[\Gamma_{K/F}]$ to $M\Z$ and $\Z$ onto $[K:F]\Z$. Thus, as $K$ and $M$ grow the image of $M^{\tx{mid},\vee}_{K,M}$ in $\Z$ shrinks to $\{0\}$.
\end{proof}


\subsection{Global canonical classes at finite levels}

Let $F$ be a number field, $E/F$ a finite Galois extension, $S$ a finite set of places of $F$, $\dot S_E$ a set of lifts of the places in $S$ to places of $E$. We assume that $(E,\dot S_E)$ satisfies \cite[Conditions 3.3.1]{KalGRI}.

Given a torus $T$ over $F$ split over $E$ with cocharacter module $Y$,
the Tate-Nakayama isomorphism reviewed in \S\ref{sub:tniso} is
\begin{equation} \label{eq:global_TN}
  \hat{H}^{i-2}(\Gamma_{E/F},Y[S_E]_0) \to \hat{H}^i(\Gamma_{E/F},T(O_{E,S})).
\end{equation}
We can apply this to $i=2$ and $Y=\tx{Hom}_\Z(\Z[S_E]_0,\Z)$.
Then the identity element in $Y[S_E]_0=\tx{End}_\Z(\Z[S_E]_0)$  maps to Tate's class $H^2(\Gamma_{E/F},
\mb{T}^\tx{iso}_{E,S}(O_{E,S}))$, which we shall denote by $\xi^\tx{iso}_{E,S}$ (it is denoted by $\alpha_3$ in \cite{Tate66} and
\cite[\S 6]{KotBG}).

We now consider a finite multiplicative group $Z$ with $A=X^*(Z)$ and $|A|$
invertible away from $S$ and have the injection (introduced as $\Theta_{E,S}$ in
\cite[\S 3.2]{KalGRI})
\[ \hat{H}^{-1}(\Gamma_{E/F},A^\vee[S_E]_0) \to H^2(\Gamma_S,Z(O_S)). \]
We have $A^\vee[S_E]_0=\tx{Hom}_\Z(A,\tx{Maps}(S_E,\frac{1}{N}\Z/\Z)_0)$ for any
$N$ multiple of $|A|$.
Assuming that $|A|$ divides $[E:F]$, we have
\[ \tx{Hom}_\Z(A, \tx{Maps}(S_E, \frac{1}{[E:F]}\Z/\Z)_0)^{N_{E/F}} =
\tx{Hom}_\Z(A,M_{E,\dot S_E}^\tx{rig})^\Gamma. \]

We can apply this to $A = M_{E,\dot S_E}^\tx{rig}$, in which case the image of
the identity is the canonical class $\xi^\tx{rig}_{E, \dot S_E, N} \in
H^2(\Gamma_S,P_{E,\dot S_E}^\tx{rig}(O_S))$, denoted by $\xi_{E \dot S_E, N}$ in
\cite[\S 3.3]{KalGRI}.

\begin{pro} \label{pro:glob_can_agree}
  The images of $\xi^\tx{iso}_{E,S}$ and $\xi^\tx{rig}_{E, \dot S_E, N}$ under
  \[ H^2(\Gamma_S,\mb{T}^\tx{iso}_{E,S}(O_S)) \rw
  H^2(\Gamma_S,\mb{T}^\tx{mid}_{E,\dot S_E}(O_S)) \lw
  H^2(\Gamma_S,P^\tx{rig}_{E,\dot S_E}(O_S))\]
  are equal.
\end{pro}
\begin{proof}
This is the global analogue of Proposition \ref{pro:loc_can_agree}, and the proof is analogous, again based on the discussion in \S\ref{sub:tniso}.
\end{proof}


\subsection{Global canonical classes at infinite level}

The canonical classes are compatible under the transition maps in all three
cases.
That is, the transition maps
\begin{enumerate}
	\item $H^2(\Gamma,\mb{T}^\tx{iso}_{K,S'}) \to H^2(\Gamma,\mb{T}^\tx{iso}_{E,S})$,
	\item $H^2(\Gamma,\mb{T}^\tx{mid}_{K,\dot S'_K}) \to H^2(\Gamma,\mb{T}^\tx{mid}_{E,\dot S_E})$, and
	\item $H^2(\Gamma,P^\tx{rig}_{K,\dot S'_K}) \to H^2(\Gamma,P^\tx{rig}_{E,\dot S_E})$
\end{enumerate}
identify the canonical classes.
In the first case the compatibility is \cite[(8.18)]{KotBG}, in the third case
it is \cite[Lemma 3.3.5]{KalGRI}.
The middle case follows by Proposition \ref{pro:glob_can_agree}.

We now want to define a canonical class in each of these three cases.
The first case is relatively straightforward.
We use the exact sequence
\[ 1 \to R^1 \varprojlim H^1(\Gamma,\mb{T}^\tx{iso}_{E,S}) \to
H^2(\Gamma,\mb{T}^\tx{iso}) \to \varprojlim H^2(\Gamma,\mb{T}^\tx{iso}_{E,S})
\to 1.\]
By \cite[Lemma 6.5]{KotBG} and Hilbert's theorem 90 we have
$H^1(\Gamma,\mb{T}^\tx{iso}_{E,S})=0$ and so we have a canonical class
$\xi^\tx{iso} \in H^2(\Gamma,\mb{T}^\tx{iso})$.

For $\mb{T}^\tx{mid}_{\dot V}$ we need the following global analogue of Lemma
\ref{lem:h1lmvan}.

\begin{pro} \label{pro:h1gmvan}
If $K/F$ is a finite Galois extension containing $E$ and s.t. $[E:F]$ divides $[K:E]$, then the map
\[ H^1(\Gamma,\mb{T}^\tx{mid}_{K,\dot S'_K}) \to H^1(\Gamma,\mb{T}^\tx{mid}_{E,\dot S_E}) \]
is zero. For every place $v$ of $F$, the map
\[ H^1(\Gamma_v,\mb{T}^\tx{mid}_{K,\dot S'_K}) \to H^1(\Gamma_v,\mb{T}^\tx{mid}_{E,\dot S_E}) \]
is also zero.
\end{pro}
\begin{proof}
The proofs for $\Gamma$ and $\Gamma_v$ are the same, so we only treat the first case. Let $\mb{T}^\circ_{E,\dot S_E}$ denote the cokernel of the injective morphism
$\mb{T}^\tx{iso}_{E,S} \to \mb{T}^\tx{mid}_{E,\dot S_E}$ dual to $c^\tx{iso}$. The vanishing of $H^1(\Gamma,\mb{T}^\tx{iso}_{E,S})$ implies that $H^1(\Gamma,\mb{T}^\tx{mid}_{E,\dot S_E}) \to H^1(\Gamma,\mb{T}^\circ_{E,\dot S_E})$ is injective. It thus suffices to prove the statement with $\mb{T}^\tx{mid}$ replaced by $\mb{T}^\circ$.

The character module $M^\circ_{E,\dot S_E}$ of $\mb{T}^\circ_{E,\dot S_E}$ is
equal to the kernel of $M^\tx{mid}_{E,\dot S_E} \to M^\tx{iso}_{E,S}$ and hence
is the $\Z[\Gamma_{E/F}]$-module consisting of functions $f : \Gamma_{E/F}
\times S_E \to \frac{1}{[E:F]}\Z$ that satisfy the conditions $\sum_w
f(\sigma,w)=0$ ,$\sum_\sigma f(\sigma,w)=0$, and $\sigma^{-1}w \notin \dot S_E
\Rightarrow f(\sigma,w)=0\}$.

The restriction of the inflation map $f^{\tx{mid},E} \mapsto f^{\tx{mid},K}$ to $M^\circ_{E,\dot S_E}$ factors as the composition
\[ M^\circ_{E,\dot S_E} \to M^\circ_{E,\dot S_E} \to M^\circ_{K,\dot S'_K}, \]
where the first map is just multiplication by $[K:E]$, while the second map is given by $f^{\tx{mid},E} \mapsto [K:E]^{-1}\cdot f^{\tx{mid},K}$. Note that, since $f^{\tx{mid},K}$ takes values in $[E:F]^{-1}\Z$, the function $[K:E]^{-1}\cdot f^{\tx{mid},K}$ takes values in $[K:F]^{-1}\Z$ and is thus a well-defined element of $M^\circ_{K,\dot S'_K}$. On cohomology we obtain the composition
\[ H^1(\Gamma,\mb{T}^\circ_{K,\dot S'_K}) \to H^1(\Gamma,\mb{T}^\circ_{E,\dot S_E}) \to H^1(\Gamma,\mb{T}^\circ_{E,\dot S_E}). \]
The second map is just multiplication by $[K:E]$. Since the torus $\mb{T}^\circ_{E,\dot S_E}$ splits over $E$ the inflation map $H^1(\Gamma_{E/F},\mb{T}^\circ_{E,\dot S_E}) \to H^1(\Gamma,\mb{T}^\circ_{E,\dot S_E})$ is an isomorphism, but its source is killed by multiplication by $[E:F]$.
\end{proof}

\begin{cor} \label{cor:h1gmvan}
The abelian groups $\varprojlim H^1(\Gamma,\mb{T}^\tx{mid}_{E_i,\dot S_i})$, $\varprojlim H^1(\Gamma_v,\mb{T}^\tx{mid}_{E_i,\dot S_i})$, and $R^1 \varprojlim H^1(\Gamma,\mb{T}^\tx{mid}_{E_i,\dot S_i})$ vanish.
\end{cor}
\begin{proof}
This follows immediately Proposition \ref{pro:h1gmvan}.
\end{proof}

The vanishing of $R^1 \varprojlim H^1(\Gamma, \mb{T}^\tx{mid}_{E, \dot S_i})$ asserted in Corollary \ref{cor:h1gmvan} implies that the natural map $H^2(\Gamma,\mb{T}^\tx{mid}_{\dot V}) \to \varprojlim H^2(\Gamma,\mb{T}^\tx{mid}_{E_i,\dot S_i})$ is an isomorphism, so we
obtain a unique class $\xi^\tx{mid}_{\dot V} \in H^2(\Gamma,\mb{T}^\tx{mid}_{\dot
V})$ mapping to $(\xi^\tx{mid}_{E_i, \dot{S}_i})_{i \geq 0}$.

The case of $P^\tx{rig}$ is the most delicate, since $R^1\varprojlim H^1(\Gamma,P^\tx{rig}_{E_i,\dot S_i})$ is known not to vanish by \cite[\S6.3]{TaibGRI}. Nonetheless, in \cite[\S 3.5]{KalGRI} a canonical class $\xi^\tx{rig}_{\dot V} \in
H^2(\Gamma,P_{\dot V})$ is constructed that maps to the inverse system
$(\xi^\tx{rig}_{E_i, \dot S_i})_{i \geq 0}$. Its construction is briefly reviewed in the proof of the following lemma.

As in Section \ref{sub:def_pro_glob} we have not recorded the tower $(E_i, S_i,
\dot{S}_i)_i$ in the notation $\xi^\tx{rig}_{\dot V}$ and $\xi^\tx{mid}_{\dot
V}$.
Again the reason is that this choice does not matter, as the following lemma shows.
Note that to be precise one should also choose a co-final sequence $(N_i)_{i
\geq 0}$ as introduced in the proof of \cite[Corollary 3.3.8]{KalGRI}, but it is
clear that increasing $S_i$ or replacing $N_i$ by a multiple yields the same
objects in the inverse limit.

\begin{lem} \label{lem:towerind}
  If two sequences $(E_i,S_i,\dot S_i)$ lead to the same $\dot V$, then they
  lead to the same class $\xi^\tx{rig}_{\dot V} \in H^2(\Gamma, P^\tx{rig}_{\dot
  V})$.
\end{lem}
\begin{proof}
  In order to obtain the statement about the class $\xi^\tx{rig}_{\dot V}$ we
  need to review its construction given in \cite[\S3.5]{KalGRI}.
  First, an element $x \in H^2(\Gamma,P_{\dot V}(\bar\A))$ is constructed
  from the local canonical classes $\xi^\tx{rig}_v \in
  H^2(\Gamma_v,P^\tx{rig}_v(\bar F_v))$, with the help of Shapiro maps $\dot
  S_v^2 : C^2(\Gamma_v,P(\bar F_v)) \to C^2(\Gamma,P(\bar\A_v))$ (for a suitable
  choice of continuous section $\Gamma_v \backslash \Gamma \rightarrow
  \Gamma$ as in \cite[Appendix B]{KalGRI}).
  These Shapiro maps can be obtained by splicing finite-level Shapiro
  maps $\dot S_v^2 : C^2(\Gamma_v,P_{E_i,\dot S_i}(\bar F_v)) \to
  C^2(\Gamma,P_{E_i,\dot S_i}(\bar\A_v))$.
  However, taking another continuous section yields the same map $H^2(\Gamma_v,
  P(\bar F_v)) \to H^2(\Gamma, P(\bar\A_v))$ (see Lemma B.4 loc.\ cit.).
  From this and the fact that for a given pair $(E,S)$ the projection to $C^2(\Gamma, P^\tx{rig}_v(\ol{\A}))$ of the
  $2$-cocycle $\dot x \in C^2(\Gamma, P^\tx{rig}_{\dot V}(\ol{\A}))$ introduced
  in \cite[\S 3.5]{KalGRI} is trivial
  for any $v \not\in S$ we see that the class $x$ is independent of the chosen
  tower.
  Since $x$ and the inverse system $(\xi^\tx{rig}_{E_i,\dot S_i})$ uniquely
  determine $\xi^\tx{rig}_{\dot V}$, the class $\xi^\tx{rig}_{\dot V}$ is itself
  independent of the chosen tower.
\end{proof}

The argument for $\xi^\tx{mid}_{\dot V}$ is analogous. However, both classes do depend on $\dot V = \varprojlim \dot S_i$.

\begin{rem}
  Let us briefly discuss the choice of $\dot V$.
  It is formal to check that the formation of $\mb{T}^\tx{rig}_{\dot V}$,
  $\mb{T}^\tx{mid}_{\dot V}$ and $\xi^\tx{rig}_{\dot V}$ is functorial in $(F,
  \ol{F}, \dot V)$, i.e.\ any isomorphism between two such triples induces an
  isomorphism between the corresponding objects.
  In particular for $\tau \in \Gamma$, denoting $\dot V' = \tau(\dot V)$ we have
  canonical isomorphisms $\mb{T}^\tx{rig}_{\dot V} \simeq \mb{T}^\tx{rig}_{\dot
  V'}$ etc.

  Consider for a moment the local case, where $F$ is a $p$-adic field. As in the global case the formation of the gerbe $\mc{E}^\tx{rig}$ is
  functorial in $(F, \ol{F})$.
  In particular any $\sigma \in \Gamma$ induces an automorphism of
  $\mc{E}^\tx{rig}$, which stabilizes $P^\tx{rig}$ and acts on the quotient
  $\mc{E}^\tx{rig} / P^\tx{rig} = \Gamma$ by conjugation by $\sigma$, but the
  action on $P^\tx{rig}$ is \emph{not} the obvious one (action of an element of
  $\Gamma$ on $\ol{F}$-points of a scheme).
  This is essentially due to the fact that finite extensions of $F$ in $\ol{F}$
  occur in the definition of $P^\tx{rig}$.
  Thus automorphisms of $\ol{F}$ induce a priori non-trivial automorphisms of
  cohomology groups for $\mc{E}^\tx{rig}$, unlike usual Galois cohomology
  (\cite[\S VII.5 Prop.\ 3]{SerLF}).
  This is reflected by the fact that for a connected reductive group $G$ over
  $F$ and $Z$ a finite central subgroup, the natural action of $\Gamma$ on
  the finite abelian group $\ol{Y}_{+, \tx{tor}}(Z \to G)$ defined in \cite[\S
  4]{KalRI} (see Proposition 5.3 loc.\ cit.) which is the source of the
  Tate-Nakayama isomorphism for $\mc{E}^\tx{rig}$, is not trivial in general.

  The existence of non-trivial automorphisms in the local case has a global
  consequence.
  Consider a global field $F$ and two arbitrary sets of lifts $\dot V$ and $\dot
  V'$ satisfying Condition \ref{cnd:dense}.
  In general there does not seem to be any natural isomorphism between the
  corresponding gerbes $\mc{E}^\tx{rig}_{\dot V}$ and $\mc{E}^\tx{rig}_{\dot
  V'}$.
  Here ``natural'' means at least compatible with localization.
  For example assume that there is a finite place $v_0 \in V$ such that for any
  $v \in V \smallsetminus \{v_0\}$ the two lifts of $v$ in $\dot V$ and $\dot
  V'$ coincide but the two lifts $\dot v_0$ and $\dot v_0'$ do not coincide, say
  with $\dot v_0' = \dot v_0 \circ \tau$ for some $\tau \in \Gamma$.
  Let $G$ be $\tx{Res}_{E/F} \tx{SL}_2$ for some quadratic extension
  $E/F$.
  Let $Z \simeq \tx{Res}_{E/F} \mu_2$ be the center of $G$.
  We have an identification of $H^1(P^\tx{rig}_{\dot V} \to
  \mc{E}^\tx{rig}_{\dot V}, Z \to G)$ with the subset of $\bigoplus_{v \in V}
  H^1(P^\tx{rig}_v \to \mc{E}^\tx{rig}_v, Z \to G)$ consisting of classes
  $(c_v)_v$ such that $c_v$ is trivial for almost all $v$ and the corresponding
  characters $\chi_v : Z(\hat{G}_\tx{sc}) \to \C^\times$ are such that $\prod_v
  \chi_v = 1$.
  Assume that $v_0$ is not split in $E$.
  The natural isomorphism between $H^1(P^\tx{rig}_{\dot v_0} \to
  \mc{E}^\tx{rig}_{\dot v_0}, Z \to G)$ and $H^1(P^\tx{rig}_{\dot v_0'} \to
  \mc{E}^\tx{rig}_{\dot v_0'}, Z \to G)$ maps $\chi_{v_0}$ to $\chi_{v_0} \circ
  \tau$.
  But in general this ruins the product condition $\prod_v \chi_v = 1$: there is
  an isomorphism $Z(\hat{G}_\tx{sc}) \simeq \Z / 2 \times \Z / 2$ such that
  $\Gamma_{E/F}$ exchanges the two factors, so if $\chi_{v_0}(1,1) \neq 1$ then
  the product condition fails.

  We thus see that the dependence of the global gerbe on $\dot V$ is necessitated by the properties of the local gerbe.
\end{rem}

\begin{cor}
  The natural maps $\mb{T}^\tx{iso} \to \mb{T}^\tx{mid}_{\dot V} \lw P_{\dot V}$
  map the canonical classes $\xi^\tx{iso}$ and $\xi^\tx{rig}_{\dot V}$ to
  $\xi^\tx{mid}_{\dot V}$.
\end{cor}
\begin{proof}
This can be checked on finite levels, where it is the content of Proposition \ref{pro:glob_can_agree}.
\end{proof}


\subsection{The global gerbes $\mc{E}^\tx{iso}$, $\mc{E}^\tx{mid}$, $\mc{E}^\tx{rig}$}
\label{sub:glob_gerbes}

The choice of a $2$-cocycle in the canonical class $\xi^\tx{iso}$ (resp.\
$\xi^\tx{mid}_{\dot V}$, $\xi^\tx{rig}_{\dot V}$) gives an extension
$\mc{E}^\tx{iso}$ (resp.\ $\mc{E}^\tx{mid}_{\dot V}$, $\mc{E}^\tx{rig}_{\dot
V}$) of $\Gamma$ by $\mb{T}^\tx{iso}(\bar{F})$ (resp.\ $\mb{T}^\tx{mid}_{\dot
V}(\bar{F})$, $P^\tx{rig}_{\dot V}(\bar{F})$).
Using these, we define functors $H^1_\tx{alg}(\mb{T}^\tx{iso} \to
\mc{E}^\tx{iso},Z \to G)$, $H^1_\tx{alg}(\mb{T}^\tx{mid}_{\dot V} \to
\mc{E}^\tx{mid}_{\dot V},Z \to G)$, and $H^1(P_{\dot V} \to
\mc{E}^\tx{rig}_{\dot V},Z \to G)$, where $G$ is a linear algebraic group defined over $F$ and $Z \subset G$ is a central diagonalizable group. Note that in the first two cases replacing $Z$ by the central torus $T=Z^\circ$ has no effect.

By the discussion in Section \ref{sub:prelim} these functors are well-defined,
independently of the choice of $2$-cocycles:
\begin{enumerate}
  \item For $\mc{E}^\tx{rig}_{\dot V}$ we use the vanishing of
    $H^1(\Gamma,P^\tx{rig}_{\dot V})$.
  \item For $\mc{E}^\tx{iso}$, we use the vanishing of
    $H^1(\Gamma,\mb{T}^\tx{iso}_{E,S})$ for all finite levels.
  \item For $\mc{E}^\tx{mid}_{\dot V}$, we use the eventual vanishing of
    $H^1(\Gamma,\mb{T}^\tx{mid}_{E,\dot S_E})$ at finite levels (Proposition
    \ref{pro:h1gmvan}).
\end{enumerate}

Next we fix morphisms of extensions $\mc{E}^\tx{iso} \to \mc{E}^\tx{mid}_{\dot
V}$ and $\mc{E}^\tx{rig}_{\dot V} \to \mc{E}^\tx{mid}_{\dot V}$ extending the
morphisms $\mb{T}^\tx{iso} \to \mb{T}^\tx{mid}_{\dot V}$ and $P^\tx{rig}_{\dot
V} \to \mb{T}^\tx{mid}_{\dot V}$.
These exist by Proposition \ref{pro:glob_can_agree}.
For an affine algebraic group $G$ defined over $F$ and a central subgroup $Z
\subset G$ we have the cohomology pointed sets defined in Section
\ref{sub:prelim}, and comparison maps between them induced by $c^\tx{iso}$ and $c^\tx{rig}$
\begin{equation} \label{eq:maps_H1}
  H^1_\tx{alg}(\mb{T}^\tx{iso} \to \mc{E}^\tx{iso},Z \to G) \lw
  H^1_\tx{alg}(\mb{T}^\tx{mid}_{\dot V} \to \mc{E}^\tx{mid}_{\dot V}, Z \to G)
  \rw H^1(P^\tx{rig}_{\dot V} \to \mc{E}^\tx{rig}_{\dot V}, Z \to G).
\end{equation}

The morphisms of extensions $\mc{E}^\tx{iso} \to \mc{E}^\tx{mid}_{\dot V}$ and
$\mc{E}^\tx{rig}_{\dot V} \to \mc{E}^\tx{mid}_{\dot V}$ are well-defined only up
to multiplication by $H^1(\Gamma, \mb{T}^\tx{mid}_{\dot V})$.
According to Corollary \ref{cor:h1gmvan} this group equals $R^1 \varprojlim
H^0(\Gamma, \mb{T}^\tx{mid}_{\dot V})$ and thus all maps $H^1(\Gamma,
\mb{T}^\tx{mid}_{\dot V}) \to H^1(\Gamma, \mb{T}^\tx{mid}_{E_i,\dot S_i})$
vanish.
It follows that the maps on cohomology \eqref{eq:maps_H1} are independent of the
morphisms of extensions used to define them.


\subsection{The relationship between the cohomology of $\mc{E}^\tx{iso}$, $\mc{E}^\tx{mid}$, and $\mc{E}^\tx{rig}$}

In this subsection $F$ is either local or global. We omit the subscript $\dot V$ when $F$ is global in order to state the following result uniformly.

\begin{cor} \label{cor:cart}
  Let $G$ be an algebraic group and $T \subset G$ a central torus.
  Then the squares
  \[ \xymatrix{
      H^1(\mb{T}^\tx{mid} \to \mc{E}^\tx{mid}, T \to G) \ar[r]
      \ar[d] & \tx{Hom}_F(\mb{T}^\tx{mid},T) \ar[d] \\
    H^1(\mb{T}^\tx{iso} \to \mc{E}^\tx{iso},T \to G) \ar[r] &
      \tx{Hom}_F(\mb{T}^\tx{iso},T)
  } \]
  and
  \[ \xymatrix{
      H^1(\mb{T}^\tx{mid} \to \mc{E}^\tx{mid}, T \to G) \ar[r]
      \ar[d] & \tx{Hom}_F(\mb{T}^\tx{mid}, T) \ar[d] \\
      H^1(P^\tx{rig} \to \mc{E}^\tx{rig}, T \to G) \ar[r] &
      \tx{Hom}_F(P^\tx{rig}, T)
  } \]
  are Cartesian and the vertical arrows are surjective.
\end{cor}
\begin{proof}
  This follows from Fact \ref{fct:cart} and Proposition \ref{pro:surj}.
\end{proof}

\begin{rem}
  In the first square, the same is true even if $T$ is a diagonalizable central
  subgroup of $G$, since replacing $T$ by $T^\circ$ does not change any of the
  four corners of the square.
  In the second square however, the surjectivity does \emph{not} remain true
  when $T$ is disconnected.
  In particular, if $Z=T$ is finite, then the top left corner becomes
  $H^1(\Gamma, G)$ and the top right corner becomes zero, while the bottom left
  corner is usually strictly larger than $H^1(\Gamma, G)$, and the bottom right
  corner is always larger than zero.
\end{rem}

\begin{rem} \label{rem:splitting_H1}
  In the first square, we even have a splitting
  \[ H^1(\mb{T}^\tx{mid} \to \mc{E}^\tx{mid}, T \to G) \simeq
    H^1(\mb{T}^\tx{iso} \to \mc{E}^\tx{iso}, T \to G) \times
  \tx{Hom}(\mb{T}^\tx{mid} / \mb{T}^\tx{iso}, T) \]
  by the splitting \eqref{eq:siso} in the local case, and by Proposition
  \ref{pro:ciso_splits_infty} in the global case.
  In the global case, the splitting is not canonical.
\end{rem}


\subsection{Localization}

\begin{fct}
  The image of the local canonical class under the map
  $H^2(\Gamma_v,\mb{T}^\tx{mid}_v) \to H^2(\Gamma_v,\mb{T}^\tx{mid}_{\dot V})$
  induced by $\tx{loc}_v$ coincides with the image of the global canonical class
  under the restriction map $H^2(\Gamma,\mb{T}^\tx{mid}_{\dot V}) \to
  H^2(\Gamma_v,\mb{T}^\tx{mid}_{\dot V})$.
\end{fct}
\begin{proof}
This follows from Propositions \ref{pro:glob_can_agree},
\ref{pro:loc_can_agree}, and \cite[Corollary 3.3.8]{KalGRI}.
\end{proof}

This implies the existence of the dotted arrow in the commutative diagram
\[ \xymatrix{
  1 \ar[r] & \mb{T}_v^\tx{mid}(\bar F_v) \ar[r] \ar[d] & \mc{E}^\tx{mid}_v
    \ar[r] \ar@{.>}[d] & \Gamma_v \ar[r] \ar@{=}[d] & 1 \\
  1 \ar[r] & \mb{T}^\tx{mid}_{\dot V}(\bar F_v) \ar[r] & \Box_2\ar[r] &
    \Gamma_v \ar[r] \ar@{=}[d] & 1 \\
  1 \ar[r] & \mb{T}^\tx{mid}_{\dot V}(\bar F) \ar[r] \ar@{=}[d] \ar@{^(->}[u] &
    \Box_1\ar[r] \ar[d] \ar[u] & \Gamma_v \ar[r] \ar@{^(->}[d] & 1 \\
  1 \ar[r] & \mb{T}^\tx{mid}_{\dot V}(\bar F) \ar[r] & \mc{E}^\tx{mid}_{\dot V}
    \ar[r] & \Gamma \ar[r] & 1 \\
}
\]
From this diagram we obtain a localization map
\[ H^1_\tx{alg}(\mc{E}^\tx{mid}_{\dot V}, G) \to H^1_\tx{alg}(\mc{E}^\tx{mid}_v,
G_{F_v}) \]
for any algebraic group $G$ over $F$. The vanishing of $\varprojlim_i H^1(\Gamma_v, \mb{T}^\tx{mid}_{E_i,\dot S_i})$
shown in Corollary \ref{cor:h1gmvan} implies that the map
$H^1_\tx{alg}(\Gamma_v,\mb{T}^\tx{mid}_{\dot V}) \to H^1(\Gamma_v,
\mb{T}^\tx{mid}_{E_i,\dot S_i})$ is zero for every $i \geq 0$.
Thus, even though the dotted arrow above is not unique, the localization map on
cohomology that it induces is unique.
To be more precise, this argument shows that for any central torus $Z \subset G$
we have a localization map
\[ Z^1(\mb{T}^\tx{mid}_{\dot V} \to \mc{E}^\tx{mid}_{\dot V}, Z \to G) \to
  Z^1(\mb{T}^\tx{mid}_v \to \mc{E}^\tx{mid}_v, Z \to G) \]
which is uniquely determined up to coboundaries taking values in $Z(\ol{F_v})$.

Similarly we have localization maps for "iso".
For "rig", see \cite[\S 3.6]{KalGRI}.
It is formal to check that the localization maps for "mid", "iso", and "rig" are
compatible, i.e.\ that the following diagram is commutative.
\[ \xymatrix{
    H^1_\tx{alg}(\mc{E}^\tx{iso}, G) \ar[d] & \ar[l]
    H^1_\tx{alg}(\mc{E}^\tx{mid}_{\dot V}, G) \ar[r] \ar[d] &
    H^1_\tx{alg}(\mc{E}^\tx{rig}_{\dot V}, G) \ar[d] \\
    H^1_\tx{alg}(\mc{E}^\tx{iso}_v, G) & \ar[l]
    H^1_\tx{alg}(\mc{E}^\tx{mid}_v, G) \ar[r] &
    H^1_\tx{alg}(\mc{E}^\tx{rig}_v, G)
} \]

\begin{lem} \label{lem:iso_unr_ae}
  Let $G$ be a connected reductive group over $F$ and $Z$ a central torus in
  $G$.
  Choose a model $\ul{G}$ of $G$ over $\mc{O}_F[1/N]$ for some integer $N>0$.
  For any $z \in Z^1(\mb{T}^\tx{iso} \to \mc{E}^\tx{iso}, Z \to G)$, there
  exists a finite Galois extension $E/F$ and a finite set $S$ of places of $F$
  containing all Archimedean places and all finite places dividing $N$ or
  ramifying in $E$, such that for all $v \in V \smallsetminus S$ the
  localization $\tx{loc}_v(z) \in Z^1(\mb{T}^\tx{iso}_v \to \mc{E}^\tx{iso}_v, Z
  \to G)$ is the product of an element inflated from $Z^1(\Gamma_{E_{\dot
  v}/F_v}, \ul{G}(\mc{O}_{E_{\dot v}}))$ with an element inflated from a
  co-boundary $\Gamma_v \to Z(\ol{F_v})$.
\end{lem}
\begin{proof}
  The restriction of $z$ to $\mb{T}^\tx{iso}$ is defined over $F$ since $Z$ is
  central, and factors through $\mb{T}^\tx{iso}_{E,S}$ for some pair $(E,S)$
  where $E$ is a finite Galois extension of $F$ and $S$ satisfies the conditions
  of \cite{Tate66}.
  Up to enlarging $S$ we can assume that $Z \to G$ comes from a closed embedding
  $\ul{Z} \to \ul{G}$ where $\ul{Z}$ is the canonical model over $\mc{O}_{F,S}$
  of the torus $Z$, and that the restriction of $z: \mb{T}^\tx{iso}_{E,S} \to Z$
  comes from a (uniquely determined) morphism $\ul{\mb{T}^\tx{iso}_{E,S}} \to
  \ul{Z}$.
  Let $\xi^\tx{iso}_{E,S} \in Z^2(\Gamma_{E/F},
  \mb{T}^\tx{iso}_{E,S}(\mc{O}_{E,S}))$ be a representative of the canonical
  class.
  Consider a finite Galois extension $E'/F$ containing $E$ and $S' \subset V$
  finite, containing $S$ and satisfying Tate's conditions.
  Let $\mc{E}$ be the extension of $\Gamma_{E'/F}$ by
  $\ul{\mb{T}^\tx{iso}_{E,S}}(\mc{O}_{E',S'})$ built using $\xi^\tx{iso}_{E,S}$.
  We have a well-defined map
  \[ \frac{Z^1(\ul{\mb{T}^\tx{iso}_{E,S}}(\mc{O}_{E',S'}) \to \mc{E},
      \ul{Z}(\mc{O}_{E',S'}) \to \ul{G}(\mc{O}_{E',S'})) }{ B^1(\Gamma_{E'/F},
      \ul{Z}(\mc{O}_{E',S'})) } \longrightarrow Z^1(\mb{T}^\tx{iso} \to
    \mc{E}^\tx{iso}, Z \to G) / B^1(\Gamma, Z(\ol{F})) \]
  which exists because the images of $\xi^\tx{iso}_{E,S}$ (by
  inflation) and of $\xi^\tx{iso}$ (by $\mb{T}^\tx{iso} \to
  \mb{T}^\tx{iso}_{E,S}$) in $H^2(\Gamma, \mb{T}^\tx{iso}_{E,S})$ coincide, and
  is uniquely determined because $H^1(\Gamma, \mb{T}^\tx{iso}_{E,S}) = 0$.
  By continuity of $z$, up to enlarging $E'$ and $S'$ the class of $z$ modulo
  $B^1(\Gamma, Z(\ol{F}))$ belongs to the image of this map.

  Now for any $v \in V \smallsetminus S'$ we have a commutative diagram
  \[ \xymatrix{
    \frac{Z^1(\mb{T}^\tx{iso}_{E,S}(\mc{O}_{E',S'}) \to \mc{E},
      \ul{Z}(\mc{O}_{E',S'}) \to \ul{G}(\mc{O}_{E',S'})) }{ B^1(\Gamma_{E'/F},
      \ul{Z}(\mc{O}_{E',S'})) } \ar[r] \ar[d] &
    Z^1(\Gamma_{E'_{\dot v}/F_v}, \ul{G}(\mc{O}_{E'_{\dot v}})) /
      B^1(\Gamma_{E'_{\dot v}/F_v}, \ul{Z}(\mc{O}_{E'_{\dot v}})) \ar[d] \\
    Z^1(\mb{T}^\tx{iso} \to \mc{E}^\tx{iso}, Z \to G) / B^1(\Gamma, Z(\ol{F}))
      \ar[r] &
    Z^1(\mb{T}^\tx{iso}_v \to \mc{E}^\tx{iso}_{\dot V}, Z \to G) /
      B^1(\Gamma_v, Z(\ol{F_v}))
  }
  \]
  The top horizontal map exists because $H^2(\Gamma_{E'_{\dot v}/F_v},
  \mb{T}^\tx{iso}_{E,S}(\mc{O}_{E'_{\dot v}})) = 0$ and is uniquely determined
  because $H^1(\Gamma_{E'_{\dot v}/F_v}, \mb{T}^\tx{iso}_{E,S}(\mc{O}_{E'_{\dot
  v}})) = 0$.
  Commutativity follows from $H^1(\Gamma_v, \mb{T}^\tx{iso}_{E,S}) = 0$.
\end{proof}

\begin{rem}
  This is much easier than \cite[Proposition 6.1.1]{TaibGRI} (see also \cite[\S
  3.9]{KalGRI}) thanks to vanishing of $H^1$ at finite level.
  A similar ramification property could be proved for "mid" using Proposition
  \ref{pro:h1gmvan}, with an extra step.
\end{rem}

\subsection{The cohomology of $\mc{E}^\tx{iso}$ and $B(G)$}
\label{sub:BG_iso_infty}

Let $F$ be a local or global field.
For $E$ a finite Galois extension of $F$, Kottwitz introduced in \cite{KotBG} an
extension $\mc{E}^\tx{iso}_E$ (simply denoted $\mc{E}(E/F)$ loc.\ cit.) of
$\Gamma_{E/F}$ by $\mb{T}^\tx{iso}_E(E)$, where $\mb{T}^\tx{iso}_E$ is the
protorus $\varprojlim_S \mb{T}^\tx{iso}_{E,S}$.
Note that the transition maps are surjective morphisms having connected kernel
between tori split by $E$, so that they induce surjective maps between groups of
$E$-points.

For a linear algebraic group $G$ defined over $F$, he defined the pointed set
$B(G) := \varinjlim_E H^1_\tx{alg}(\mc{E}^\tx{iso}_E, G(E))$ in \S 10 loc.\ cit.
The transition maps exist thanks to the compatibility of canonical classes with
inflation maps $\mb{T}^\tx{iso}_K \to \mb{T}^\tx{iso}_E$, and are well-defined
thanks to the vanishing of $H^1(\Gamma_{K/F}, \mb{T}^\tx{iso}_E(K))$.
Define $\ol{Z}^1_\tx{alg}(\mc{E}^\tx{iso}_E, G(E))$ as the quotient of
$Z^1_\tx{alg}(\mc{E}^\tx{iso}_E, G(E))$ by the following equivalence relation:
$z \sim z'$ if and only if there exists $t \in \mb{T}^\tx{iso}_E(E)$ such that
$z'(w) = z(twt^{-1})$ for all $w \in \mc{E}^\tx{iso}_E$.
Note that $z(twt^{-1}) = z(t) z(w) \sigma_w(z(t))^{-1}$ where $\sigma_w \in
\Gamma_{E/F}$ is the image of $w$, so that we have a surjective map
$\ol{Z}^1_\tx{alg}(\mc{E}^\tx{iso}_E, G(E)) \to H^1_\tx{alg}(\mc{E}^\tx{iso}_E,
G(E))$.
The inflation maps are well-defined at the level of $\ol{Z}^1$ and letting
$\widetilde{B}(G) := \varinjlim_E \ol{Z}^1_\tx{alg}(\mc{E}^\tx{iso}_E, G(E))$,
we obtain a pointed set mapping onto $B(G)$.

\begin{lem} \label{lem:BG_iso_infty}
  The natural map $\widetilde{B}(G) \to \ol{Z}^1_\tx{alg}(\mc{E}^\tx{iso}, G)$,
  defined similarly to the inflation maps, is an isomorphism.
  In particular we have a natural isomorphism $B(G) \to
  H^1_\tx{alg}(\mc{E}^\tx{iso}, G)$.
\end{lem}
\begin{proof}
  Surjectivity is essentially the first part of the proof of Lemma
  \ref{lem:iso_unr_ae}.
  Injectivity is clear.
\end{proof}

Note that since the inflation maps $\mb{T}^\tx{iso}_K \to \mb{T}^\tx{iso}_E$ do
not have connected kernel, there is no reason why there should exists
$\xi^\tx{iso} \in Z^2(\Gamma, \mb{T}^\tx{iso}(\ol{F}))$ such that for some $E
\neq F$ its image in $Z^2(\Gamma, \mb{T}^\tx{iso}(\ol{F}))$ belongs to
$Z^2(\Gamma_{E/F}, \mb{T}^\tx{iso}_E(E))$.

It is not difficult to check that the isomorphisms in Lemma
\ref{lem:BG_iso_infty} are compatible with localization.
In fact this is the second step of the proof of Lemma \ref{lem:iso_unr_ae}.


\subsection{A Tate-Nakayama description of $H^1_\tx{alg}(\mc{E}^\tx{mid}_{\dot
V}, T)$}

The main goal of this subsection is to describe the failure of commutativity of
\eqref{eq:locnoncom}.
For this, we shall give a linear algebraic description of the group
$H^1_\tx{alg}(\mc{E}^\tx{mid}_{\dot V}, T)$ for an algebraic torus $T$, both in
the local and in the global case.
The local description will be used to give a precise formula for the failure of
commutativity of \eqref{eq:locnoncom}.
The global description will be used to show that this failure of commutativity
satisfies a product formula.

We first begin with the local case. Let $F$ be local. Let $T$ be an algebraic torus defined over $F$ and write $Y=X_*(T)$. In all three cases we have the inflation-restriction exact sequence of Fact \ref{fct:infres}
\[ 1 \to H^1(\Gamma,T) \to H^1_\tx{alg}(\mc{E},T) \to \tx{Hom}_F(\mb{D},T) \to ...\]
with $\mb{D}$ being one of $\mb{T}^\tx{mid}$, $\mb{T}^\tx{iso}$, or
$P^\tx{rig}$. We have the Tate-Nakayama isomorphism $Y_{\Gamma,\tx{tor}} \to H^1(\Gamma,T)$ and the isomorphism $(Y \otimes X^*(\mb{D}))^\Gamma \to \tx{Hom}_F(\mb{D},T)$ of Fact \ref{fct:iso_hom_D}. Compatible with these two isomorphisms is a third isomorphism whose target is
$H^1_\tx{alg}(\mc{E},T)$.
In the case of $\mc{E}^\tx{iso}$ its source is $Y_\Gamma$ according to
\cite[(13.2)]{KotBG}.
In the case of $\mc{E}^\tx{rig}$ its source is the torsion subgroup of
$Y\otimes\Q/IY$ according to \cite[\S4]{KalRI}.
We shall write $Y^\tx{iso}:=Y_\Gamma$ and $Y^\tx{rig}:=(Y\otimes\Q/IY)[\tx{tor}]$ and see both of these as functors from the category of tori to the category of $\Gamma$-modules.

Define
\[ Y^\tx{mid} := Y^\tx{mid}(T) := \{ (\lambda,\mu)| \lambda \in Y_\Gamma, \mu \in Y \otimes \Q, N^\natural(\lambda)=N^\natural(\mu)\}, \]
where $N^\natural$ is the normalized norm map, i.e. $N^\natural =
[E:F]^{-1}\sum_{\sigma \in \Gamma_{E/F}} \sigma$ for any finite Galois extension
$E/F$ splitting $T$.
We have a natural map $Y^\tx{mid} \to Y^\tx{iso}, (\lambda, \mu) \mapsto
\lambda$.

\begin{fct} \label{fct:tncartloc}
The right square in the commutative diagram
\[ \xymatrix{
  0 \ar[r] & Y_{\Gamma,\tx{tor}} \ar[r] \ar@{=}[d] & Y^\tx{mid} \ar[r] \ar[d] &
    (Y \otimes M^\tx{mid})^\Gamma \ar[d] \\
  0 \ar[r] & Y_{\Gamma,\tx{tor}} \ar[r] & Y^\tx{iso} \ar[r] & (Y \otimes
    M^\tx{iso})^\Gamma
}
\]
is Cartesian.
Here $Y^\tx{mid} \to (Y \otimes M^\tx{mid})^\Gamma$ maps $(\lambda,\mu)$ to
$\sum_\sigma \sigma(\mu) \otimes \sigma \in Y \otimes \Q[\Gamma]$, $Y^\tx{iso}
\to (Y \otimes M^\tx{iso})^\Gamma$ is given by $N^\natural$ via the natural
embedding $M^\tx{iso} \subset \Q$, and $(Y \otimes M^\tx{mid})^\Gamma \to (Y
\otimes M^\tx{iso})^\Gamma$ is $\tx{id}_Y \otimes c^\tx{iso}$,
i.e.\ via the isomorphisms \eqref{eq:iso_hom_D} it is the ``pre-composition by
$c_\tx{iso}$'' map $\tx{Hom}(\mb{T}^\tx{mid}, T) \to \tx{Hom}(\mb{T}^\tx{iso},
T)$.
\end{fct}

\begin{rem}
  Note that the $\Gamma$-invariant map $Y \otimes M^\tx{mid} \to Y \otimes
  M^\tx{iso}$ given by $\tx{id}_Y \otimes c^\tx{iso}$ has the
  $\Gamma$-equivariant splitting $\tx{id}_Y \otimes s^\tx{iso}$.
  This gives a canonical isomorphism between $Y^\tx{mid}$ and $Y^\tx{iso} \oplus
  \tx{ker}(\tx{id}_Y \otimes c^\tx{iso})^\Gamma$.
\end{rem}

\begin{pro} \label{pro:tnmidloc}
  There exists a unique isomorphism $Y^\tx{mid} \to
  H^1_\tx{alg}(\mc{E}^\tx{mid},T)$ that is functorial in $T$ and fits into
  the commutative diagrams
  \[ \xymatrix{
    Y_{\Gamma,\tx{tor}} \ar[r] \ar[d]_\wr & Y^\tx{mid} \ar[r] \ar[d]_\wr & (Y
      \otimes \Q[\Gamma])^\Gamma \ar[d]_\wr \\
    H^1(\Gamma,T) \ar[r] & H^1_\tx{alg}(\mc{E}^\tx{mid},T) \ar[r] &
      \tx{Hom}_F(\mb{T}^\tx{mid},T) \\
  } \]
  and
  \[ \xymatrix{
    Y^\tx{mid} \ar[r] \ar[d]_\wr & Y^\tx{iso} \ar[d]_\wr \\
    H^1_\tx{alg}(\mc{E}^\tx{mid},T) \ar[r] & H^1(\mc{E}^\tx{iso},T)
  } \]
\end{pro}
\begin{proof}
  This follows from Fact \ref{fct:tncartloc} and Corollary \ref{cor:cart} which
  realize $Y^\tx{mid}$ and $H^1_\tx{alg}(\mc{E}^\tx{mid}, T)$ as fiber products,
  and the functoriality of the isomorphisms $Y^\tx{iso} \to
  H^1_\tx{alg}(\mc{E}^\tx{iso},T)$, $\tx{Hom}_F(\mb{T}^\tx{iso},T) \to (Y
  \otimes \Q)^\Gamma$ and $\tx{Hom}_F(\mb{T}^\tx{mid},T) \to (Y \otimes
  \Q[\Gamma])^\Gamma$.
\end{proof}

\begin{cor} \label{cor:tnmidlocrel}
Let $Z \subset T$ be a subtorus defined over $F$. The isomorphism $Y^\tx{mid}(T) \to H^1_\tx{alg}(\mc{E}^\tx{mid},T)$ identifies
\[ Y^\tx{mid}(Z \to T) := \{ (\lambda,\mu)| \lambda \in Y_\Gamma(T), \mu \in Y(Z) \otimes \Q, N^\natural(\lambda)=N^\natural(\mu)\} \subset Y^\tx{mid}(T) \]
with $H^1(\mb{T}^\tx{mid} \to \mc{E}^\tx{mid},Z \to T)$.
\end{cor}

\begin{pro} \label{pro:tnisomid}
  For any torus $T$ over $F$, the composition
  \[ Y^\tx{iso} \simeq H^1_\tx{alg}(\mc{E}^\tx{iso},T) \to
    H^1(\mc{E}^\tx{mid},T) \simeq Y^\tx{mid} \]
  where the middle map is pullback along $s_\tx{iso} : \mc{E}^\tx{mid} \to
  \mc{E}^\tx{iso}$, maps $\lambda \in Y^\tx{iso}$ to $(\lambda,
  N^\natural(\lambda))$.
\end{pro}
\begin{proof}
  The image of $\lambda$ is of the form $(\lambda, \mu)$ since $s_\tx{iso} \circ
  c_\tx{iso} = \tx{id}_{\mc{E}^\tx{iso}}$ up to $Z^1(\Gamma, \mb{T}^\tx{iso})$.
  We can compute $\mu \in Y \otimes \Q$ as the image of $\lambda$ by the
  composition
  \[ Y^\tx{iso} \to (Y \otimes \Q)^\Gamma \xrightarrow{\tx{id}_Y \otimes
    s_\tx{iso}} (Y \otimes \Q[\Gamma])^\Gamma \simeq Y \otimes \Q \]
  where the last map is evaluation at $1 \in \Gamma$.
\end{proof}

\begin{pro} \label{pro:tnmidrigloc}
The composition
\[ Y^\tx{mid} \xrightarrow{\sim} H^1_\tx{alg}(\mc{E}^\tx{mid},T) \to
H^1(P^\tx{rig},T) \xrightarrow{\sim} Y^\tx{rig}, \]
where the middle map is induced by a map of gerbes $c_\tx{rig}$ as in
\eqref{eq:maps_gerbes_loc}, is given by $(\lambda,\mu) \mapsto \lambda-\mu$.
\end{pro}
\begin{proof}
  This composition, as well as the map $(\lambda,\mu) \mapsto \lambda-\mu$, are
  functorial homomorphisms that fit into the commutative diagram with exact rows
  \[ \xymatrix{
    0\ar[r] & Y_{\Gamma,\tx{tor}}\ar[r]\ar@{=}[d] & Y^\tx{mid}\ar[r]\ar[d] &
      Y\otimes \Q\ar[d]\\
    0\ar[r] & Y_{\Gamma,\tx{tor}}\ar[r] & Y^\tx{rig}\ar[r] &
      \frac{Y\otimes\Q}{Y}
  } \]
  where the right vertical map is $\mu \mapsto - \mu + Y$.
  If $T$ is induced, that is if $T \simeq \tx{Res}_{A/F} \tx{GL}_1$ for some
  finite \'etale $F$-algebra $A$, then $Y_{\Gamma, \tx{tor}} \simeq H^1(\Gamma,
  T) = 0$ by Shapiro's lemma and Hilbert's theorem 90 and so our two maps
  $Y^\tx{mid} \to Y^\tx{rig}$ are equal in this case.
  In general we realize $T$ as a quotient of an induced torus $\widetilde{T}$,
  by realizing $Y$ as a quotient of an induced $\Z[\Gamma_{E/F}]$-module
  $\widetilde{Y}$ for some finite Galois extension $E/F$.
  Let $K = \ker (\widetilde{Y} \to Y)$.
  To conclude it is enough to show that $\widetilde{Y}^\tx{mid} \to Y^\tx{mid}$
  is surjective.
  Let $(\lambda, \mu) \in Y^\tx{mid}$, choose $\widetilde{\lambda} \in
  \widetilde{Y}_\Gamma$ lifting $\lambda$ and $\widetilde{\mu}_0 \in
  \widetilde{Y} \otimes \Q$ lifting $\mu$.
  Then $\epsilon := N^\natural(\widetilde{\lambda}) -
  N^\natural(\widetilde{\mu}_0) \in (K \otimes \Q)^\Gamma$ and thus
  $N^\natural(\epsilon) = \epsilon$.
  Setting $\widetilde{\mu} = \widetilde{\mu}_0 + \epsilon$, we obtain that
  $(\widetilde{\lambda}, \widetilde{\mu}) \in \widetilde{Y}^\tx{mid}$ lifts
  $(\lambda, \mu)$.
\end{proof}

Recall that in the non-Archimedean case, the morphism of extensions $s_\tx{iso}
\circ c_\tx{rig} : \mc{E}^\tx{rig} \to \mc{E}^\tx{iso}$ equals
\cite[(3.13)]{KalRIBG} up to $Z^1(\Gamma, \mb{T}^\tx{iso})$.
Note that Proposition 3.2 loc.\ cit.\ (as well as its Archimedean analogue)
follows from Propositions \ref{pro:tnisomid} and \ref{pro:tnmidrigloc} above.
This is not surprising since the proof of Proposition \ref{pro:tnmidrigloc} is
very similar to that of Proposition 3.2 loc.\ cit.
For later use in Section \ref{sec:mult_formula} we also record the following
consequence.

\begin{cor} \label{cor:tnlcononcomm}
Consider the two homomorphisms $H^1_\tx{alg}(\mc{E}^\tx{mid},T) \to H^1(\mc{E}^\tx{rig},T)$ obtained by pulling back along
\begin{enumerate}
	\item A homomorphism $c_\tx{rig} : \mc{E}^\tx{rig} \to \mc{E}^\tx{mid}$,
  \item The composition of a morphism $s_\tx{iso} \circ c_\tx{rig} :
    \mc{E}^\tx{rig} \to \mc{E}^\tx{iso}$ with $c_\tx{iso} :
    \mc{E}^\tx{iso} \to \mc{E}^\tx{mid}$.
\end{enumerate}
Their difference, when pre-composed with the Tate-Nakayama isomorphism
$Y^\tx{mid} \to H^1_\tx{alg}(\mc{E}^\tx{mid},T)$ and post-composed with the
inverse of the Tate-Nakayama isomorphism $Y^\tx{rig} \to
H^1(\mc{E}^\tx{rig},T)$, is given by the map
\[ Y^\tx{mid} \to Y^\tx{rig},\qquad (\lambda,\mu) \mapsto \mu-N^\natural(\mu). \]
\end{cor}
\begin{proof}
  This follows immediately from Propositions \ref{pro:tnisomid} and
  \ref{pro:tnmidrigloc} and the equality $N^\natural(\mu) =
  N^\natural(\lambda)$.
\end{proof}

Now turn to the global case: Let $F$ be global. Let $T$ be a torus defined over $F$ and $Z \subset T$ a subtorus defined over $F$. As before we write $Y=X_*(T)$.
Denote $Y^\tx{iso}_{E,S} := (Y[S_E]_0)_\Gamma$ and $Y^\tx{iso}_E = \varinjlim_S
Y^\tx{iso}_{E,S}$.
For $K/F$ a finite Galois extension containing $E$ define $j : Y^\tx{iso}_K \to
Y^\tx{iso}_E$ by $j(f)(v) = \sum_{\substack{w \in V_K \\ w \mapsto v}} f(w)$ for
$v \in V_E$.
It turns out that this map is an isomorphism.
Choose a section $s$ of $V_K \to V_E$ whose image contains $\dot V_K$, then by
\cite[Lemma 3.1.7]{KalGRI} the unique right inverse $s_! : Y[V_E]_0 \to
Y[s(V_E)]_0 \subset Y[V_K]_0$ to $Y[V_K]_0 \to Y[V_E]_0$ induces a well-defined
map $! : Y^\tx{iso}_E \to Y^\tx{iso}_K$ which does not depend on the choice of
$s$ (in fact $\dot V$ is irrelevant here) and of course is injective.
It is also surjective thanks to Lemma \ref{lem:repr_Sdot_supp}, and so $j$ is
an isomorphism with inverse $!$.
Denote $Y^\tx{iso} = \varinjlim_E Y^\tx{iso}_E = \varprojlim_E Y^\tx{iso}_E$.
By \cite[Lemma 4.1]{KotBG} we have a Tate-Nakayama isomorphism $Y^\tx{iso}
\simeq H^1_\tx{alg}(\mc{E}^\tx{iso},T)$.
Note that \cite[Lemma 8.4]{KotBG} also shows a posteriori that the maps $j$ are
isomorphisms.
For any $(E,S)$ such that $E$ splits $T$ and $S$ satisfies Tate's axioms the
following diagram is commutative:
\[ \xymatrix{
    0 \ar[r] & (Y[S_E]_0)_{\Gamma,\tx{tor}} \ar[r] \ar[d] & Y^\tx{iso}_{E, S}
      \ar[r]^{N_{E/F}} \ar[d] & (Y[S_E]_0)^\Gamma \ar[d] \\
    0 \ar[r] & H^1(F, T) \ar[r] & H^1_\tx{alg}(\mc{E}^\tx{iso}, T) \ar[r] &
    \tx{Hom}(\mb{T}^\tx{iso}, T)
} \]
where the left vertical arrow is the usual Tate-Nakayama map and the right
vertical arrow is the obvious map.
For a subtorus $Z$ of $T$ defined over $F$ denote $Y_T = X_*(T)$ and $Y_Z =
X_*(Z)$ and let
\begin{eqnarray*}
  Y^\tx{mid}(Z \to T)_{E,\dot S_E} & := & \Bigg\{(\lambda,\mu) \Bigg| \lambda
    \in Y^\tx{iso}_{E,S}(T), \mu \in (M^\tx{mid}_{E,\dot S_E}\otimes
    Y_Z)^\Gamma, \\
  && \sum_{\sigma \in \Gamma_{E/F}} \sigma(\lambda(\sigma^{-1}w)) =
    \sum_\sigma\mu(\sigma,w) \Bigg\}.\\
  Y^\tx{mid}_{\dot V}(Z \to T)&:=&\varinjlim_{E,S} Y^\tx{mid}(Z \to T)_{E,\dot
    S_E}
\end{eqnarray*}
where the transition maps are given by $! = j^{-1}$ on $\lambda$ and using the
inflation maps defined in Section \ref{sub:globinf} on $\mu$.
It is easy to see that one could also define $Y^\tx{mid}_{\dot V}(Z \to T)$ as
the set of pairs $(\lambda, \mu)$ with $\lambda \in Y^\tx{iso}(T)$ and $\mu \in
(M^\tx{mid}_{\dot V} \otimes Y_Z)^\Gamma$ satisfying the above relation at any
level $E$.
More concretely, using the description of $M^\tx{mid}_{\dot V}$ given in Section
\ref{sub:def_pro_glob} we may also see $\mu$ as a function $V \to \Q \otimes
Y_Z$ with finite support such that $\sum_{v \in V} \mu(v) = 0$ and for any
Archimedean place $v$ of $F$, $N_{\ol{F_v} / F_v}(\mu(v)) \in Y_Z$.

\begin{fct} \label{fct:tncartglob}
The right square below is Cartesian
\[ \xymatrix{
  0 \ar[r] & Y^\tx{iso}_\tx{tor}(T) \ar[r] \ar@{=}[d] & Y^\tx{mid}_{\dot V}(Z
    \to T) \ar[r] \ar[d] & (M^\tx{mid}_{\dot V} \otimes Y_Z)^\Gamma \ar[d] \\
  0 \ar[r] & Y^\tx{iso}_\tx{tor}(T) \ar[r] & Y^\tx{iso}(Z \to T) \ar[r] &
    (M^\tx{iso} \otimes Y_Z)^\Gamma
} \]
\end{fct}
\begin{proof}
  This follows directly from the definition.
\end{proof}

\begin{rem}
  Using the same argument as in the proof of Proposition
  \ref{pro:ciso_splits_infty} one can show that the natural transformation
  $Y^\tx{mid}_{\dot V} \to Y^\tx{iso}$ admits a splitting, but as we already
  observed in Remark \ref{rem:no_can_rig_iso} this splitting is not canonical.
\end{rem}

As in the local case we simply write $Y^\tx{mid}_{\dot V}(T)$ for
$Y^\tx{mid}_{\dot V}(T \to T)$.

\begin{pro} \label{pro:tnmidglob}
There is a unique functorial isomorphism
\[  Y^\tx{mid}_{\dot V}(T) \to H^1_\tx{alg}(\mc{E}^\tx{mid}_{\dot V},T) \]
that fits into the commutative diagrams
\[ \xymatrix{
  Y^\tx{iso}_\tx{tor} \ar[r]\ar[d] & Y^\tx{mid}_{\dot V} \ar[r]\ar[d] &
  (M^\tx{mid}_{\dot V} \otimes Y)^\Gamma\ar[d]\\
  H^1(\Gamma,T)\ar[r] & H^1(\mc{E}^\tx{mid}_{\dot V}, T) \ar[r] &
  \tx{Hom}_F(\mb{T}^\tx{mid}_{\dot V},T)
} \]
and
\[ \xymatrix{
  Y^\tx{mid}_{\dot V}\ar[r]\ar[d]&Y^\tx{iso}\ar[d]\\
  H^1(\mc{E}^\tx{mid}_{\dot V},T)\ar[r]&H^1(\mc{E}^\tx{iso},T).
} \]
\end{pro}
\begin{proof}
Analogous to the proof of Proposition \ref{pro:tnmidloc}, but with Fact \ref{fct:tncartloc} now replaced by its global analog Fact \ref{fct:tncartglob}.
\end{proof}

\begin{cor} \label{cor:tnmidglobrel}
  Let $Z \subset T$ be a subtorus defined over $F$.
  The isomorphism $Y^\tx{mid}_{\dot V}(T) \to H^1_\tx{alg}(\mc{E}^\tx{mid}_{\dot
  V},T)$ identifies $Y^\tx{mid}_{\dot V}(Z \to T)$ with
  $H^1_\tx{alg}(\mc{E}^\tx{mid}_{\dot V},Z \to T)$.
\end{cor}

\begin{cor} \label{cor:surj_H1_mid_iso}
  For $Z \to T$ an injective map between tori over $F$ both maps
  \[ H^1(\mb{T}^\tx{mid}_{\dot V} \to \mc{E}^\tx{mid}_{\dot V}, Z \to T) \to
  H^1(\mb{T}^\tx{iso} \to \mc{E}^\tx{iso}, Z \to T) \to H^1(\Gamma, T/Z) \]
  are surjective.
\end{cor}
\begin{proof}
  By Corollary \ref{cor:cart} the map $H^1(\mb{T}^\tx{mid}_{\dot V} \to
  \mc{E}^\tx{mid}_{\dot V}, Z \to T) \to H^1(\mb{T}^\tx{iso} \to
  \mc{E}^\tx{iso}, Z \to T)$ is surjective, and it is clear that the map
  \[ H^1(\mb{T}^\tx{iso} \to \mc{E}^\tx{iso}, Z \to T) \simeq Y^\tx{iso}(Z \to
  T) \to Y^\tx{iso}(T/Z)_\tx{tor} \simeq H^1(\Gamma, T/Z) \]
  is surjective.
\end{proof}

\begin{cor} \label{cor:surj_H1_mid_iso_nonabelian}
  For $G$ a connected reductive group over $F$ and $Z$ a central torus in $G$
  both maps
  \[ H^1(\mb{T}^\tx{mid}_{\dot V} \to \mc{E}^\tx{mid}_{\dot V}, Z \to G) \to
  H^1(\mb{T}^\tx{iso} \to \mc{E}^\tx{iso}, Z \to G) \to H^1(\Gamma, G/Z) \]
  are surjective.
\end{cor}
\begin{proof}
  Lemma A.1 in \cite{KalGRI} reduces this to the previous Corollary.
\end{proof}

\begin{pro} \label{pro:tnmidglobloc}
  Given $w \in \dot S \subset \dot V$, the composition of the localization map
  $\tx{loc}_w : H^1(\mc{E}^\tx{mid}_{\dot V},T) \to H^1(\mc{E}^\tx{mid}_w,T)$
  with $Y^\tx{mid}_{E,\dot S_E} \to H^1(\mc{E}^\tx{mid}_{\dot V},T)$ and
  $H^1(\mc{E}^\tx{mid}_w,T) \to Y^\tx{mid}_w$ sends $(\lambda,\mu) \in
  Y^\tx{mid}_{E,\dot S_E}$ to $(\lambda_w,\mu_w) \in Y^\tx{mid}_w$ determined by
  \[ \lambda_w = \sum_{\Gamma_{E_w/F_v} \lmod \Gamma_{E/F}}
  \sigma\lambda(\sigma^{-1}w),\qquad \mu_w=\mu(1,w). \]
\end{pro}
\begin{proof}
The element $\lambda_w \in Y_{\Gamma_w}$ is the image of the pair $(\lambda_w,\mu_w)$ under the natural map $Y^\tx{mid}_w \to Y_{\Gamma_w}$. Therefore the formula for $\lambda_w$ follows from the commutativity of
\[ \xymatrix{
  Y^\tx{mid}_{E,\dot S_E} \ar[r] \ar[d] & H^1_\tx{alg}(\mc{E}^\tx{mid}_{\dot
    V}, T) \ar[r]^{\tx{loc}_w} \ar[d] & H^1_\tx{alg}(\mc{E}^\tx{mid}_w,T) \ar[r]
    \ar[d] & Y^\tx{mid}_w \ar[d] \\
  Y^\tx{iso}_{E,S} \ar[r] & H^1_\tx{alg}(\mc{E}^\tx{iso},T) \ar[r]^{\tx{loc}_w}
    & H^1_\tx{alg}(\mc{E}^\tx{iso}_w,T)\ar[r]&Y^\tx{iso}_w
} \]
and the formula for the bottom horizontal map described in \cite[\S7.7]{KotBG}.

Analogously, the element $\mu_w \in Y \otimes \Q$ is the image of the pair $(\lambda_w,\mu_w)$ under the natural map $Y^\tx{mid}_w(T) \to Y \otimes \Q$. The formula for $\mu_w$ follows from the commutativity of
\[ \xymatrix{
  Y^\tx{mid}_{E,\dot S_E} \ar[r] \ar[d] & H^1_\tx{alg}(\mc{E}^\tx{mid}_{\dot V},
    T) \ar[r]^{\tx{loc}_w} \ar[d] & H^1_\tx{alg}(\mc{E}^\tx{mid}_w,T) \ar[r]
    \ar[d] & Y^\tx{mid}_w \ar[d] \\
  (M^\tx{mid}_{E,\dot S_E} \otimes Y)^\Gamma \ar[r] &
    \tx{Hom}_F(\mb{T}^\tx{mid}_{\dot V}, T) \ar[r]^{\tx{loc}_w} &
    \tx{Hom}_F(\mb{T}^\tx{mid}_w,T) \ar[r] & Y \otimes \Q
} \]
and the formula for $\mu_w$ follows from the formula for the bottom horizontal map obtained by composing the localization map $M^\tx{mid}_{E,\dot S_E} \to M_{E_w,[E:F]}^\tx{mid}$ described in \S\ref{sub:locmap} with the evaluation at $1$ map $M_{E_w,[E:F]}^\tx{mid} \to \frac{1}{[E:F]}\Z \hrw \Q$.
\end{proof}

Although we will not need it in the paper, there is a global analogue of
Proposition \ref{pro:tnmidrigloc}.

\begin{pro} \label{pro:tnmidrigglob}
  For any torus $T$ defined over $F$, the composition
  \[ Y^\tx{mid}_{\dot V} \xrightarrow{\sim} H^1_\tx{alg}(\mc{E}^\tx{mid}_{\dot
    V},T) \to H^1(P^\tx{rig}_{\dot V},T) \xrightarrow{\sim} Y^\tx{rig}_{\dot
    V}, \]
  where the middle map is induced by a map of gerbes $c_\tx{rig}$ as discussed
  in Section \ref{sub:glob_gerbes}, is given by $(\lambda,\mu) \mapsto
  \lambda-\mu$.
\end{pro}
\begin{proof}
  The proof is similar to the local case (Proposition \ref{pro:tnmidrigloc}),
  except that we cannot take $\widetilde{\mu} = \widetilde{\mu}_0 + \epsilon$
  since $\epsilon$ is not supported on $\dot V$, but thanks to Lemma
  \ref{lem:repr_Sdot_supp} we may find $\epsilon' \in (K \otimes \Q)[\dot V]_0$
  such that $N^\natural(\epsilon') = \epsilon$ and set $\widetilde{\mu} =
  \widetilde{\mu}_0 + \epsilon'$.
  Details are left to the reader.
\end{proof}


\section{The global multiplicity formula}
\label{sec:mult_formula}


\subsection{An obstruction} \label{sub:obs}

Let $G^*$ be a quasi-split connected reductive group over a global field $F$,
$\psi : G^* \to G$ an inner twist.
We consider a strongly regular semi-simple element $\delta \in G(\A)$ with the
property that there is an element of $G^*(F)$ stably conjugate to $\delta$.
In this situation, Langlands has defined a cohomological obstruction to the existence of an $F$-point in the $G(\A)$-conjugacy class of $\delta$. We shall now review its definition and properties, following material from \cite{LS87} and \cite{Kot86}. We will then reinterpret this obstruction in terms of the global gerbes.

We shall first assume that $G$ satisfies the Hasse principle, as the obstruction takes a more transparent form in that case.

The condition on $\delta$ is that there exists $\delta^* \in G^*(F)$ and $g \in G^*(\bar \A)$ so that $\delta=\psi(g\delta^*g^{-1})$. Let $T^*$ be the centralizer of $\delta^*$. Let $u \in C^1(\Gamma,G^*(\bar F))$ be a lift of the element of $Z^1(\Gamma,G^*_\tx{ad}(\bar F))$ corresponding to $\psi$. For $\sigma \in \Gamma$ the element $g^{-1}u(\sigma)\sigma(g)$ lies in $T^*(\bar\A)$.
Its image in $T^*(\bar\A)/T^*(\bar F)$ is independent of the choice of $u$ and
this gives a 1-cocycle $\Gamma \to T^*(\bar\A)/T^*(\bar F)$.
Its cohomology class is independent of the choice of $g$ and will be denoted by $\tx{obs}(\delta) \in H^1(\Gamma,T^*(\bar\A)/T^*(\bar F))$.

Let $\delta' \in G(\A)$ be stably conjugate to $\delta$.
That is, there exists $g \in G(\bar\A)$ s.t. $\delta'=g\delta g^{-1}$.
We can define $\tx{inv}(\delta,\delta') \in H^1(\Gamma,T(\bar\A))$ just as in
the local case, namely as the class of $\sigma \mapsto g^{-1}\sigma(g)$, and
$\tx{inv}(\delta, \delta') = \tx{inv}(\delta, \delta'')$ if $\delta'$ and
$\delta''$ are conjugate in $G(\A)$.
Note that centralizer $T$ of $\delta$ in $G_\A$ is only defined over $\A$.
Under the isomorphism $H^1(\Gamma,T(\bar\A)) = \bigoplus_v H^1(\Gamma_v,T(\bar
F_v))$ the class $\tx{inv}(\delta,\delta')$ corresponds to $\sum_v
\tx{inv}(\delta_v,\delta'_v)$.

\begin{lem} \label{lem:obs1}
\begin{enumerate}
	\item The class $\tx{obs}(\delta)$ depends only on the $G(\A)$-conjugacy class of $\delta$.
	\item The class $\tx{obs}(\delta)$ is independent of the choice of $\delta^*$ in the following sense: If $\delta^{**} \in G^*(\bar F)$ is another choice then the unique admissible isomorphism $\varphi_{\delta^{**},\delta^*} : T^{**} \to T^*$ sending $\delta^{**}$ to $\delta^*$ identifies the two versions of $\tx{obs}(\delta)$ obtained from $\delta^{**}$ and $\delta^*$, respectively. 
	\item If $\delta' \in G(\A)$ is stably conjugate to $\delta \in G(\A)$ then
\[ \tx{obs}(\delta') = \tx{obs}(\delta) \cdot \varphi_{\delta^*,\delta}^{-1}(\tx{inv}(\delta,\delta')) \]
where $\varphi_{\delta^*, \delta} : T^*_{\A} \simeq T$ is the admissible
isomorphism mapping $\delta^*$ to $\delta$.
\end{enumerate}
\end{lem}
\begin{proof}
The first claim is immediate. For the second, let $h \in G^*(\bar F)$ be s.t. $h\delta^{**}h^{-1}=\delta^*$. Then $\tx{Ad}(h) : T^{**} \to T^*$ is the admissible isomorphism sending $\delta^{**}$ to $\delta^*$. The version of the obstruction obtained from $\delta^{**}$ is represented by the 1-cocycle 
\[ h^{-1}g^{-1}u(\sigma)\sigma(gh) = \tx{Ad}(h^{-1})(g^{-1}u(\sigma)\sigma(g)) \cdot h^{-1}\sigma(h), \]
and $h^{-1}\sigma(h)$ lies in $Z^1(\Gamma,T^{**}(\bar F))$.
The third claim follows from a similar direct computation.
\end{proof}

\begin{pro} \label{pro:obs1}
The class $\tx{obs}(\delta)$ vanishes if and only if the $G(\A)$-conjugacy class of $\delta$ contains an $F$-point.
\end{pro}
\begin{proof}
If the $G(\A)$-conjugacy class of $\delta$ contains an $F$-point, Lemma \ref{lem:obs1} allows us to replace $\delta$ by that $F$-point without changing $\tx{obs}(\delta)$. Then $g$ can be chosen in $G^*(\bar F)$ and so $g^{-1}u(\sigma)\sigma(g) \in T^*(\bar F)$, showing that $\tx{obs}(\delta)$ vanishes.

Conversely, if the class of $g^{-1}u(\sigma)\sigma(g)$ in $H^1(\Gamma,T^*(\bar\A)/T^*(\bar F))$ is trivial there exists $t \in T^*(\bar \A)$ s.t. $(gt)^{-1}u(\sigma)\sigma(gt) \in T^*(\bar F)$ for all $\sigma \in \Gamma$. We may replace $g$ by $gt$ and drop $t$ from the notation. Now $z(\sigma) := \psi(g^{-1}u(\sigma)\sigma(g)u(\sigma)^{-1})$ is an element of $Z^1(\Gamma,G(\bar F))$ whose image in $Z^1(\Gamma,G(\bar \A))$ is cohomologically trivial, namely the coboundary of $\psi(g)$. By the Hasse principle for $G$ there exists $h \in G^*(\bar F)$ s.t. $\psi((gh^{-1})^{-1}u(\sigma)\sigma(gh^{-1})u(\sigma))=1$. This means $\psi(gh^{-1}) \in G(\A)$. Therefore $\delta'=\psi(gh^{-1})^{-1}\delta\psi(gh^{-1})$ lies in the $G(\A)$-conjugacy class of $\delta$. At the same time, $\delta'=\psi(h\delta^*h^{-1}) \in G(F)$.
\end{proof}

We now drop the condition that $G$ satisfies the Hasse principle. Then it turns out that $H^1(\Gamma,T^*(\bar\A)/T^*(\bar F))$ is not a suitable home for the obstruction any more. By work of Kneser, Harder, and Chernousov, $G_\tx{sc}$ does satisfy the Hasse principle. This will lead to a slight modification of $H^1(\Gamma,T_\tx{sc}^*(\bar\A)/T_\tx{sc}^*(\bar F))$ that will serve as a replacement for $H^1(\Gamma,T^*(\bar\A)/T^*(\bar F))$.

\begin{lem} \label{lem:adelic_sc_conj}
Let $\delta^* \in G^*(F)$. There exists $g \in G^*(\bar\A)$ s.t. $\delta=\psi(g\delta^*g^{-1})$ if and only if there exists $g_\tx{sc} \in G_\tx{sc}^*(\bar\A)$ s.t. $\delta=\psi(g_\tx{sc}\delta^*g_\tx{sc}^{-1})$.
\end{lem}
\begin{proof}
This is not immediate because the map $G^*_\tx{sc}(\bar \A) \to G^*(\bar \A)$ need not be surjective. However, letting $T^*$ be the centralizer of $\delta^*$, we have $G^*_\tx{sc}(\bar \A)\cdot T^*(\bar \A) = G^*(\bar\A)$. Indeed, letting $E$ be a sufficiently large finite Galois extension of $F$, for almost all places $w$ of $E$ we have $G^*_\tx{sc}(O_{E_w}) \cdot T^*(O_{E_w}) = G^*(O_{E_w})$ by \cite[(3.3.4)]{Kot84S}.
\end{proof}

Let $g_\tx{sc} \in G_\tx{sc}^*(\bar \A)$ be so that $\delta=\psi(g_\tx{sc}\delta^*g_\tx{sc}^{-1})$. Let $u_\tx{sc} \in C^1(\Gamma,G^*_\tx{sc}(\bar F))$ be a lift of the element of $Z^1(\Gamma,G^*_\tx{ad}(\bar F))$. For $\sigma \in \Gamma$ the element $g_\tx{sc}^{-1}u_\tx{sc}(\sigma)\sigma(g_\tx{sc})$ lies in $T_\tx{sc}^*(\bar\A)$. Its image in $T_\tx{sc}^*(\bar\A)/T_\tx{sc}^*(\bar F)$ is independent of the choice of $u_\tx{sc}$ and is a 1-cocycle. Its cohomology class is independent of the choice of $g_\tx{sc}$ and will be denoted by $\tx{obs}_\tx{sc}(\delta) \in H^1(\Gamma,T_\tx{sc}^*(\bar\A)/T_\tx{sc}^*(\bar F))$. We can also refine the invariant of two stably conjugate $\delta,\delta' \in G(\A)$ to $\tx{inv}_\tx{sc}(\delta,\delta') \in H^1(\Gamma,T_\tx{sc}(\bar\A))$ in the obvious way. Then we have

\begin{lem} \label{lem:obs2}
\begin{enumerate}
	\item The class $\tx{obs}_\tx{sc}(\delta)$ depends only on the $G_\tx{sc}(\A)$-conjugacy class of $\delta$.
	\item The class $\tx{obs}_\tx{sc}(\delta)$ is independent of the choice of $\delta^*$ in the following sense: If $\delta^{**} \in G^*(\bar F)$ is another choice then the admissible isomorphism $\varphi_{\delta^{**},\delta^*} : T_\tx{sc}^{**} \to T_\tx{sc}^*$ identifies the two versions of $\tx{obs}_\tx{sc}(\delta)$ obtained from $\delta^{**}$ and $\delta^*$, respectively. 
	\item If $\delta' \in G(\A)$ is stably conjugate to $\delta \in G(\A)$ then
\[ \tx{obs}_\tx{sc}(\delta') = \tx{obs}_\tx{sc}(\delta) \cdot \varphi_{\delta^*,\delta}^{-1}(\tx{inv}_\tx{sc}(\delta,\delta')). \]
\end{enumerate}
\end{lem}
\begin{proof}
The same as for Lemma \ref{lem:obs1}.
\end{proof}

\begin{pro} \label{pro:obs2}
The class $\tx{obs}_\tx{sc}(\delta)$ vanishes if and only if the $G_\tx{sc}(\A)$-conjugacy class of $\delta$ contains an $F$-point.
\end{pro}
\begin{proof}
The same as for Proposition \ref{pro:obs1}, but now the assumption that $G_\tx{sc}$ satisfies the Hasse principle is automatically satisfied by the work of Kneser, Harder, and Chernousov, \cite{Kne62}, \cite{Har65}, \cite{Har66}, \cite{Cher89}.
\end{proof}

We now define 
\begin{enumerate}
	\item $\Delta$ to be the image under $H^1(\Gamma,T^*_\tx{sc}(\bar\A)) \to H^1(\Gamma,T^*_\tx{sc}(\bar\A)/T^*_\tx{sc}(\bar F))$ of 
	\[ \tx{ker}(H^1(\Gamma,T^*_\tx{sc}(\bar\A)) \to H^1(\Gamma,T^*(\bar\A))) \cap \tx{ker}(H^1(\Gamma,T^*_\tx{sc}(\bar\A)) \to H^1(\Gamma,G_\tx{sc}^*(\bar\A))) \]
	
	\item $\mf{K}(T/F)^D$ to be the quotient of $H^1(\Gamma,T^*_\tx{sc}(\bar\A)/T^*_\tx{sc}(\bar F))$ by $\Delta$
	\item $\tx{obs}(\delta) = \tx{obs}(\delta)_\tx{sc}+\Delta \subset H^1(\Gamma,T^*_\tx{sc}(\bar\A)/T^*_\tx{sc}(\bar F))$, or equivalently the image of $\tx{obs}(\delta)_\tx{sc}$ in $\mf{K}(T/F)^D$.
\end{enumerate}

\begin{rem} Note that under the map 
\[ H^1(\Gamma,T^*_\tx{sc}(\bar\A)/T^*_\tx{sc}(\bar F)) \to H^1(\Gamma,T^*(\bar\A)/T^*(\bar F)) \] 
the subgroup $\Delta$ goes to $0$ and $\tx{obs}_\tx{sc}(\delta)$ maps to the element that we denoted by $\tx{obs}(\delta)$ when the Hasse principle holds. Thus, the new definition of $\tx{obs}(\delta)$ is a refinement of the old definition.
\end{rem}

\begin{pro} \label{pro:obs3}
The element $\tx{obs}(\delta)$ depends only on the $G(\A)$-conjugacy class of $\delta$. It is independent of the choice of $\delta^*$. It vanishes if and only if the $G(\A)$-conjugacy class of $\delta$ contains an $F$-point.
\end{pro}

\begin{proof}
The independence of $\delta^*$ is immediate from the second point in Lemma \ref{lem:obs2}. For the other two statements we note that all elements in the $G(\A)$-conjugacy class of $\delta$ are $G_\tx{sc}(\bar \A)$-conjugate to each other by Lemma \ref{lem:adelic_sc_conj}.
Therefore the set of $G_\tx{sc}(\A)$-conjugacy classes inside of the
$G(\A)$-conjugacy class of $\delta$ is in bijection with 
\[ \ker \left(
H^1(\Gamma,T^*_\tx{sc}(\bar\A)) \to H^1(\Gamma, G^*_\tx{sc}(\bar\A)) \right) \cap \ker \left(
H^1(\Gamma,T^*_\tx{sc}(\bar\A)) \to H^1(\Gamma,T^*(\bar\A)) \right), \] 
namely via
$\delta' \leftrightarrow
\varphi_{\delta^*,\delta}^{-1}(\tx{inv}_\tx{sc}(\delta,\delta'))$.
With this, the outstanding two statements of this proposition follow from the third point in Lemma \ref{lem:obs2} and Proposition \ref{pro:obs2}.
\end{proof}

We now assume again that $G$ satisfies the Hasse principle. It is clear from the definitions that if $G=G^*$ then $\tx{obs}(\delta)$ is simply the image of $\tx{inv}(\delta^*,\delta)$ under $H^1(\Gamma,T^*(\bar\A)) \to H^1(\Gamma,T^*(\bar\A)/T^*(\bar F))$. When $G \neq G^*$ then $\tx{inv}(\delta^*,\delta)$ doesn't make sense. However, using the cohomology of the global gerbe $\mc{E}^\tx{iso}$ we can make sense of it. More precisely, we need the versions of $\mc{E}^\tx{iso}$ with $\A$-coefficients and $\A/F$-coefficients. These are denoted by $\mc{E}_2(K/F)$ and $\mc{E}_1(K/F)$ respectively in \cite[\S1.5]{KotBG}. Here $K/F$ is any sufficiently large finite Galois extension. We shall write $H^1_\tx{alg}(\mc{E}^\tx{iso}_2,T(\bar\A))$ for $\varinjlim_K H^1_\tx{alg}(\mc{E}_2^\tx{iso}(K/F),T^*(\A_K))$ and $H^1_\tx{alg}(\mc{E}^\tx{iso}_1,T^*(\bar\A)/T^*(\bar F))$ for $\varinjlim_K H^1_\tx{alg}(\mc{E}_1^\tx{iso}(K/F),T^*(\A_K)/T^*(K))$.

Assume that the element of $Z^1(\Gamma,G^*_\tx{ad})$ corresponding to $\psi$ has a lift $z^\tx{iso} \in Z^1_\tx{bas}(\mc{E}^\tx{iso},G^*)$. Define $\tx{inv}[z^\tx{iso}](\delta^*,\delta)$ to be the class in $H^1_\tx{alg}(\mc{E}_2^\tx{iso},T^*(\bar\A))$ of the 1-cocycle
\[ e \mapsto g^{-1}z^\tx{iso}(e) \sigma_e(g). \]

\begin{fct} \label{fct:obsiso}
The image in $H^1_\tx{alg}(\mc{E}_1^\tx{iso},T^*(\bar\A)/T^*(\bar F))$ of $\tx{inv}[z^\tx{iso}](\delta^*,\delta)$ lies in the subgroup $H^1(\Gamma,T^*(\bar\A)/T^*(\bar F))$ and equals $\tx{obs}(\delta)$.	
\end{fct}
\begin{proof}
Immediate from the fact that $z^\tx{iso}$ takes values in $T
^*(\bar F)$.
\end{proof}

This image can be computed in terms of Tate-Nakayama isomorphisms.

\begin{pro} \label{pro:obs_iso}
  Denote $Y = X_*(T^*)$.
  The compositions
  \[ H^1_\tx{alg}(\mc{E}_2^\tx{iso}, T^*(\ol{\A})) \to \bigoplus_{v \in V}
    H^1_\tx{alg}(\mc{E}^\tx{iso}_v, T^*) \simeq \bigoplus_{v \in V} Y/I_v(Y) \to
    Y/IY, \]
  where the isomorphism is Kottwitz' local Tate-Nakayama isomorphism and the
  last map is ``sum over all places'', and
  \[ H^1_\tx{alg}(\mc{E}_2^\tx{iso}, T^*(\ol{\A})) \to H^1(\Gamma, T^*(\ol{\A}) /
  T^*(\ol{F})) \simeq (Y/IY)[\tx{tor}] \subset Y/IY \]
  where the isomorphism is the Tate-Nakayama isomorphism defined in
  \cite{Tate66}, are equal.
\end{pro}
\begin{proof}
  This is a special case of a more general compatibility.
  In \cite{KotBG} Kottwitz defined a generalization $H^1(\mc{E}^\tx{iso}_1,
  T^*(\ol{\A}) / T^*(\ol{F})) \simeq Y/IY$ of Tate's isomorphism, and proved
  local-global compatibility ((6.6) loc.\ cit.): the compositions
  \[ H^1_\tx{alg}(\mc{E}^\tx{iso}_2, T^*(\ol{\A})) \to \bigoplus_{v \in V}
  H^1_\tx{alg}(\mc{E}^\tx{iso}_v, T^*) \simeq \bigoplus_{v \in V} Y/I_v(Y) \to
  Y/IY  \]
  and
  \[ H^1_\tx{alg}(\mc{E}^\tx{iso}_2, T^*(\ol{\A})) \to
  H^1_\tx{alg}(\mc{E}^\tx{iso}_1, T^*(\ol{\A}) / T^*(\ol{F})) \simeq I/IY \]
  are equal.
\end{proof}


\subsection{Global transfer factors} \label{sub:adelic_tf}

We continue with $\psi : G^* \to G$ from the previous subsection.
Let $(H,\mc{H},s,\eta)$ be an endoscopic datum for $G$ and $(H_1,\eta_1)$ a
$z$-pair as in \cite[\S2.2]{KS99}.
We assume that for every place $v$ of $F$ there exists a pair of related
strongly $G$-regular elements $\gamma_{0,v}^H \in H(F_v)$ and $\delta_{0,v} \in
G(F_v)$.

\begin{lem} \label{lem:loc_glob_tf}
  Under this assumption there exists a pair of related strongly $G$-regular
  elements $\gamma_0^H \in H(F)$ and $\delta_0 \in G(\A)$.
\end{lem}
\begin{proof}
  The assumption is equivalent to the following one: for every place $v$ of $F$,
  there exists a maximal torus $T_{H,v}$ of $H_{F_v}$ and an admissible
  embedding $T_{H,v} \hookrightarrow G_{F_v}$, i.e.\ an isomorphism of $T_{H,v}$
  with a maximal torus of $G_{F_v}$ as in \cite[Lemma 3.3.B]{KS99}.
  Note that this assumption is automatically satisfied at every place $v$ such
  that $G_{F_v}$ is quasi-split, by essentially the same argument as in the
  proof of this lemma.
  By \cite[Theorem 7.9]{BorelSpringer_rat2} the variety $X$ of maximal tori of
  $H$ is rational, in particular it satisfies weak approximation.
  Let $S$ be the finite set of places $v$ such that $G_{F_v}$ is not
  quasi-split.
  For any $v \in S$ the $H(F_v)$-conjugacy class of $T_{H,v}$ is a neighbourhood
  of $T_{H,v}$ in $X(F_v)$ (for the natural topology).
  Therefore there exists a maximal torus $T_H$ of $H$, that is an element of
  $X(F)$, such that for every $v \in S$ the maximal tori $(T_H)_{F_v}$ and
  $T_{H,v}$ of $H_{F_v}$ are conjugate under $H(F_v)$.

  Take any $G$-regular semisimple $\gamma_0^H \in T_H(F)$. Since $G^*$ is quasi-split over $F$ there exists a strongly regular $\delta_0^* \in G^*(F)$ that is related to $\gamma^H$. We have arranged that for any place $v$ of $F$, there exists a
  strongly regular element $\delta_{0,v}$ of $G(F_v)$ that is stably conjugate to $\delta_{0,v}^*$, and we are left to show
  that we may take $\delta_{0,v} \in \ul{G}(\mc{O}_{F_v})$ for almost all places
  $v$, where $\ul{G}$ is any model of $G$ over $\mc{O}_F[1/N]$ for some integer
  $N>0$.

  We can find $N>0$ and models $\ul{G^*}$ and $\ul{G}$ over $\mc{O}_F[1/N]$ s.t. $\psi$ extends to a map over $\mc{O}[1/N]$, where $\mc{O}$ is the ring of integers of the maximal extension of $F$ unramified at all places prime to $N$. For a place $v$ prime to $N$ the element of $Z^1(\Gamma_v,G^*)$ corresponding to $\psi$ lies in $Z^1(\Gamma_v/I_v,G^*(O_v))$, so there is an element $g_v \in G^*(O_v)$ s.t. $\psi\circ\tx{Ad}(g_v)$ is an isomorphism defined over $O_{F_v}$. We let $\delta_{0,v}=\psi(g_v\delta_0^*g_v^{-1}) \in \ul{G}(O_{F_v})$. For the finitely many places $v$ that are not prime to $N$ we choose $\delta_{0,v} \in G(F_v)$ to be stably conjugate to $\psi(\delta_0) \in G(\bar F)$. That is, there is $g_v \in G^*(\bar F_v)$ s.t. $\delta_v=\psi(g_v\delta_0g_v^{-1})$. The resulting collection $\delta_0$ of local elements lies in $G(\A)$.
\end{proof}

Under that assumption an adelic transfer factor $\Delta_\A'(\gamma^{H_1},
\delta)$ is defined for all pairs of strongly $G$-regular elements $\gamma^{H_1}
\in H_1(\A)$ and $\delta \in G(\A)$.
It is defined in \cite[\S6.3]{LS87} and \cite[\S7.3]{KS99} without the prime
decoration, and in discussed in \cite[\S5.4]{KS12} with the prime decoration.

By the construction in \cite[\S7.3]{KS99} the factor $\Delta_\A'(\gamma^{H_1},\delta)$ is a product of local transfer factors over all places. Since at each place the local transfer factor is canonical up to a scalar multiple, their product is also. What makes the global transfer factor completely canonical is the following property: If $\gamma^H \in H(F)$ and $\delta \in G(\A)$ are related, then
\begin{equation} \label{eq:tfobs} \Delta'_\A(\gamma^{H_1},\delta) = \<\tx{obs}(\delta),\hat\varphi^{-1}_{\gamma^H,\delta^*}(s)\>. \end{equation}
To explain the notation, note first that the condition $\gamma^H \in H(F)$ implies the existence of $\delta^* \in G^*(F)$ in the $G^*(\bar\A)$ conjugacy class of $\psi^{-1}(\delta)$, which was assumed in the definition of $\tx{obs}(\delta)$ that was reviewed in \S\ref{sub:obs}. Let $T^* \subset G^*$ be the centralizer of $\delta^*$ and let $T^H \subset H$ be the centralizer of $\gamma^H$. There is a unique admissible isomorphism $\varphi_{\gamma^H,\delta^*} : T^H \to T^*$ sending $\gamma^H$ to $\delta^*$. It induces an isomorphism $\hat T^H \to \hat T$ via which we transport $s \in [Z(\hat H)/Z(\hat G)]^\Gamma$ to an element of $[\hat T/Z(\hat G)]^\Gamma$. Then we use the Tate-Nakayama pairing $\<-,-\>$ between $H^1(\Gamma,T_\tx{sc}^*(\bar\A)/T_\tx{sc}^*(\bar F))$ and $\pi_0([\hat T^*/Z(\hat G)]^\Gamma)$. Even though the obstruction $\tx{obs}(\delta)$ was not just an element of $H^1(\Gamma,T_\tx{sc}^*(\bar\A)/T_\tx{sc}^*(\bar F))$, but rather a set of elements there, the pairing is well-defined, because $s$ has the property that its image under the connecting homomorphism $[Z(\hat H)/Z(\hat G)]^\Gamma \to H^1(\Gamma,Z(\hat G))$ is everywhere locally trivial. The subgroup of $[Z(\hat H)/Z(\hat G)]^\Gamma$ of all elements with this property is denoted by $\mf{K}(T^*/F)$ and is dual to the quotient $\mf{K}(T^*/F)^D$ recalled in \S\ref{sub:obs}.

\begin{rem} \label{rem:inverses}
We need to be careful with the normalizations of the various pairings we are using. Equation \eqref{eq:tfobs} appears optically compatible with \cite[Corollary 7.3.B]{KS99}. However, the latter is stated for the transfer factor $\Delta_\A$ constructed as in \cite{LS87}, rather than the factor $\Delta_\A'$ that we are using here, whose construction differs from $\Delta_\A$ by inverting the element $s$, see \cite[\S5.1]{KS12}. The reason for preferring $\Delta'$ over $\Delta$ is that $\Delta$ itself doesn't properly generalize to the twisted setting, see \cite{KS12}. 

Thus now \eqref{eq:tfobs} would seem optically at odds with the $\Delta'$-version of \cite[Corollary 7.3.B]{KS99}, which is the last displayed equation in \cite{KS12} before \S5.5 there. This discrepancy stems from yet another clash of conventions between \cite{LS87} and \cite{KS99}. First we recall that in the twisted setting $\tx{obs}(\delta)$ is an element of a hypercohomology group $H^1(\A/F,T^*_\tx{sc} \to T^*)$, while the element $\kappa$ lies in $H^1(W_F,\hat T^* \to \hat T^*/Z(\hat G))$. The maps in the two complexes are $1-\theta^*$ and $1-\hat\theta^*$ respectively, where $\theta^*$ is the twisting automorphism. In the untwisted setting $\theta^*=1$ and the two hypercohomology groups each break up into direct products $H^1(\A/F,T^*_\tx{sc} \to T^*)=H^1(\A/F,T^*_\tx{sc}) \oplus H^0(\A/F,T^*)$ and $H^1(W_F,\hat T^* \to \hat T^*/Z(\hat G))=H^1(W_F,\hat T^*) \oplus H^0(W_F,\hat T^*/Z(\hat G))$. The pairing between the two hypercohomology groups becomes the product of the standard normalization of the Tate-Nakayama pairing between $H^1(\A/F,T^*_\tx{sc})$ and $H^0(W_F,\hat T^*/Z(\hat G))$ and the \emph{negative} of the standard normalization of the Langlands pairing between $H^0(\A/F,T^*)$ and $H^1(W_F,\hat T^*)$. The occurrence of this negative is forced by the anti-commutativity of the cup product.

The reason why the $\Delta'$-version of \cite[Corollary 7.3.B]{KS99} is compatible with \eqref{eq:tfobs} is that the projection of the element $\tx{obs}(\delta) \in H^1(\A/F,T^*_\tx{sc} \to T^*)$ constructed in \cite{KS99} onto the direct factor $H^1(\A/F,T^*_\tx{sc})$ is the \emph{inverse} of the element $\tx{obs}(\delta)$ constructed in \S\ref{sub:obs} above. For this we direct the reader to \cite[pp.82-83]{KS99} and point out that the element $v(\sigma)$ constructed there is the 1-cocycle that represents our element $\tx{obs}(\delta)$ here (despite the element $g$ there being the inverse of our element $g$ here), yet the class $\tx{obs}(\delta)$ constructed there contains the inverse of $v(\sigma)$. The converntions we have used here are those used in \cite[(3.4)]{LS87} in the quasi-split case, which are the opposite of those used in \cite[\S5.3]{KS99} in the quasi-split case.

Finally, we remark that the same issue occurs with the local and adelic invariant $\tx{inv}(\delta,\delta')$ used in this paper -- it follows the quasi-split convention in \cite{LS87}, and is the opposite of the quasi-split convention in \cite{KS99}, so any equation involving $\tx{inv}(\delta,\delta')$ in this paper will contain an inverse when compared to the $\Delta'$-version of the corresponding equation of \cite{KS12}, and would thus appear optically identical to the corresponding equation in \cite{KS99} for the factor $\Delta$ in place of $\Delta'$.
\end{rem}


\subsection{Global transfer factors in terms of isocrystals} \label{sub:gtf_iso}

In this subsection we assume that $Z(G)$ is connected and $G$ satisfies the Hasse principle. We will show how the canonical adelic transfer factor reviewed in \S\ref{sub:adelic_tf} can be written as a product of local transfer factors that are normalized using $\widetilde{B}(G)$.

Let again $(H,\mc{H},s,\eta)$ be an endoscopic datum for $G$ and $(H_1,\eta_1)$ a $z$-pair as in \cite[\S2.2]{KS99}. By assumption $s \in [Z(\hat H)/Z(\hat G)]^\Gamma$ and the image of $s$ in $Z^1(\Gamma,Z(\hat G))$ under the connecting homomorphism has cohomologically trivial localization at each place of $F$. The Hasse principle for $G$, reinterpreted as \cite[Lemma 11.2.2]{Kot84}, implies that the image of $s$ in $Z^1(\Gamma,Z(\hat G))$ is already cohomologically trivial and thus $s$ lifts to an element $s^\natural \in Z(\hat H)^\Gamma$. We shall refer to $(H,\mc{H},s^\natural,\eta)$ as an \emph{isocrystal-refined} endoscopic datum. An isomorphism of such data is required to identify the two elements $s^\natural$, and not simply their images $s$.

By Corollary \ref{cor:surj_H1_mid_iso_nonabelian} we can choose $z^\tx{iso} \in Z^1_\tx{bas}(\mc{E}^\tx{iso},G^*)$ such that $\psi^{-1}\sigma(\psi)=\tx{Ad}(\bar z_\sigma)$. Let $\mf{w}$ be a global Whittaker datum for $G^*$. At each place $v$ of $F$ we now have the normalized local transfer factor $\Delta[\mf{w}_v,z^\tx{iso}_v] : H_1(F_v)_\tx{sr} \times G(F_v)_\tx{sr} \to \C$, defined as
\begin{equation} \label{eq:tf_iso}
\Delta[\mf{w}_v,z^\tx{iso}_v](\gamma^{H_1},\delta) = \Delta[\mf{w}_v](\gamma^{H_1},\delta^*) \cdot \<\tx{inv}[z^\tx{iso}_v](\delta^*,\delta),\hat\varphi_{\delta^*,\gamma^H}(s^\natural)\>. \end{equation}
We need to explain the notation.
On the left we have $\gamma^{H_1} \in H_1(F_v)_\tx{sr}$ and $\delta \in
G(F_v)_\tx{sr}$.
We choose arbitrarily $\delta^* \in G^*(F_v)$ that is conjugate in $G^*(\bar F_v)$ to $\psi^{-1}(\delta)$.
For $\Delta[\mf{w}_v](\gamma^{H_1},\delta^*)$ to be non-zero it is necessary
that $\gamma_1$ be a norm of $\delta$, so we make this assumption.
Let $T^* \subset G^*$ be the centralizer of $\delta^*$, a maximal torus of $G^*$. Given any $g \in G^*(\bar F)$ with $\delta=\psi(g \delta^*g^{-1})$, the element $g^{-1}z^\tx{iso}(e)\sigma_e(g)$ of $G^*(\bar F_v)$ belongs to $T^*(\bar F_v)$, for all $e \in \mc{E}^\tx{iso}_v$, where $\sigma_e \in \Gamma_v$ denotes the image of $e$ under $\mc{E}^\tx{iso}_v \to \Gamma_v$. For formal reasons the map $e \mapsto g^{-1}z^\tx{iso}(e)\sigma_e(g)$ is an element of $Z^1_\tx{alg}(\mc{E}^\tx{iso}_v,T^*)$ and its cohomology class is independent of the choice of $g$. We denote by $\tx{inv}[z^\tx{iso}_v](\delta^*,\delta)$ this cohomology class.

The element $\gamma^H \in H(F_v)_\tx{sr}$ is the image of $\gamma^{H_1}$ under $H_1 \to H$. Letting $T^H \subset H$ be its centralizer, a maximal torus of $H$, there is a unique admissible isomorphism $\varphi_{\delta^*,\gamma^H} : T^* \to T^H$ mapping $\delta^*$ to $\gamma^H$. Its dual, when composed with the canonical embedding $Z(\hat H) \to \hat T^H$, transports $s^\natural \in Z(\hat H)^{\Gamma_v}$ into $\hat T^{*,\Gamma_v}$. The pairing $\<-,-\> : H^1_\tx{alg}(\mc{E}^\tx{iso}_v,T^*) \times \hat T^{*,\Gamma_v}$ is given by \cite[Lemma 8.1]{KotBG}.

The transfer factor $\Delta[\mf{w}_v] : H_1(F_v)_\tx{sr} \times G^*(F_v)_\tx{sr}$ is the Whittaker normalization of the factor $\Delta'$, as defined in \cite[(5.5.2)]{KS12}.

\begin{pro} \label{pro:local_tf_iso}
The function $\Delta[\mf{w}_v,z^\tx{iso}_v]$ is an absolute transfer factor.
\end{pro}
\begin{proof}
This is proved in \cite[Proposition 2.2.1]{KalIso}, under the assumption that the $z$-pair is trivial, and with the inverse of the pairing used here. Nonetheless, the proof given there carries over to this situation with trivial modifications.
\end{proof}

\begin{pro} \label{pro:adelic_tf_iso}
Let $\gamma^{H_1} \in H_1(\A)_\tx{sr}$ and $\delta \in G(\A)_\tx{sr}$. For almost all places $v$ the factor $\Delta[\mf{w}_v,z^\tx{iso}_v](\gamma^{H_1}_v,\delta_v)$ is equal to $1$ and the product
\[ \prod_{v \in V} \Delta[\mf{w}_v,z^\tx{iso}_v](\gamma^{H_1}_v,\delta_v) \]
is equal to the canonical adelic transfer factor $\Delta_\A'(\gamma^{H_1},\delta)$.
\end{pro}
\begin{proof}
The adelic transfer factor was reviewed in \S\ref{sub:adelic_tf}. By construction it is a product of local transfer factors, and so is the product in this proposition. Therefore one is a scalar multiple of the other. It is enough to show that they give the same value on one pair $(\gamma^{H_1},\delta)$ of related elements. We choose the pair so that $\gamma^{H_1} \in H_1(F)$. The existence of such an element is assumed in the definition of the global transfer factor. Then we can choose $\delta^* \in G(F)$ related to $\gamma^{H_1}$.

We first claim that almost all factors in the product $\prod_v \Delta[\mf{w}_v](\gamma^{H_1},\delta^*)$ are equal to $1$ and the product itself is equal to $1$. For this we recall that this factor is the product of terms $\epsilon_v \cdot \Delta_I \cdot \Delta_{II} \cdot \Delta_{III_1} \cdot \Delta_{III_2} \cdot \Delta_{IV}$. The terms $\epsilon_v$ are the local components of a global root number of an orthogonal Artin representation of degree $0$. Therefore almost all terms are equal to $1$ and their product is equal to $1$. For the remaining terms we apply Theorem 6.4.A and Corollary 6.4.B of \cite{LS87}.

This reduces to showing that almost all factors in the product
\[ \prod_v \<\tx{inv}[z^\tx{iso}_v](\delta^*,\delta),\hat\varphi_{\delta^*,\gamma^H}(s^\natural)\> \]
are equal to $1$ and this product equals $\Delta_\A'(\gamma^{H_1},\delta)$. This follows from Fact \ref{fct:obsiso} and the equality
\begin{eqnarray*}
\prod_v \<\tx{inv}[z^\tx{iso}_v](\delta^*,\delta),\hat\varphi_{\delta^*,\gamma^H}(s^\natural)\>&=&\<\sum_v \tx{inv}[z^\tx{iso}_v](\delta^*,\delta),\hat\varphi_{\delta^*,\gamma^H}(s^\natural)\>\\
&=&\<\tx{inv}[z^\tx{iso}](\delta^*,\delta),\hat\varphi_{\delta^*,\gamma^H}(s^\natural)\>\\
\end{eqnarray*}
where the pairings in the first term are the local pairings $H^1_\tx{alg}(\mc{E}^\tx{iso}_v,T^*(\bar F_v)) \times X^*(\hat T^{*,\Gamma_v}) \to \C$, the pairing in second term is the pairing $H^1_\tx{alg}(\mc{E}^\tx{iso}_2,T^*(\bar\A)) \times (X^*(\hat T^*)\otimes \Z[V])_\Gamma \to \C$, and that in the third term is $H^1_\tx{alg}(\mc{E}^\tx{iso}_1,T^*(\bar\A)/T^*(\bar F)) \times X^*(\hat T^{*,\Gamma}) \to \C$.
The last equality follows from Proposition \ref{pro:obs_iso}.
\end{proof}

\begin{rem} \label{rem:adelic_transfer_can}
  Every $f \in \mc{C}^\infty_c(G(\A))$ has a \emph{canonical} transfer $f^{H_1}
  \in \mc{C}^\infty_c(H_1(\A))$, up to functions having identically vanishing
  stable orbital integrals.
Namely, choose an arbitrary decomposition $\Delta'_\A = \prod_v \Delta'_v$ of the canonical adelic transfer factor as a product of local transfer factors. Assume without loss of generality that $f$ is a decomposable function $f=\otimes_v f_v$. Let $f^{H_1}_v \in \mc{C}^\infty_c(H_1(F_v))$ be the transfer of $f_v$ relative to the factor $\Delta'_v$ and let $f^{H_1}=\otimes_v f^{H_1}_v$. Each individual local component $f^{H_1}_v$ depends on the choice of $\Delta'_v$, which is well-defined up to a complex scalar, but the formula $\Delta'_\A = \prod_v \Delta'_v$ implies that $f^{H_1}$ is independent of these choices.

In particular, Proposition \ref{pro:adelic_tf_iso} implies that if we take $f^{H_1}_v$ to be the transfer relative to $\Delta[\mf{w}_v,z^\tx{iso}_v]$, then $\otimes_v f^{H_1}_v$ will be the canonical adelic transfer. Analogously, \cite[Proposition 4.4.1]{KalGRI} implies the same statement with $\Delta[\mf{w}_v,z^\tx{iso}_v]$ replaced by the transfer factor normalized using $\mc{E}^\tx{rig}$ in place of $\mc{E}^\tx{iso}$.
\end{rem}

\begin{rem}
  According to \cite[\S 2.4]{LS90} the product formula proved in Proposition
  \ref{pro:adelic_tf_iso} generalizes from the case of strongly regular pairs
  $(\delta, \gamma^{H_1})$ to the case of $(G,H)$-regular pairs, i.e.\ the
  assumption of \cite[\S 6.10]{Kot86} holds.
\end{rem}

\subsection{Global transfer factors in the rigid setting}
\label{sub:global_transfer_rig}

In this section we do not assume that the connected reductive group $G$ over $F$
satisfies the Hasse principle, nor that $Z(G)$ is connected.
The analogue of Proposition \ref{pro:adelic_tf_iso} in the rigid setting was
proved in \cite[\S 4.4]{KalGRI}, under the assumption that there exists a pair of related elements $\gamma^H \in H(F)$ and $\delta \in G(F)$. Note that this is stronger than the assumption on the existence of pair of related elements $\gamma^H \in H(F)$ and $\delta \in G(\A)$. The reason this stronger assumption was made in \cite[\S 4.4]{KalGRI} is that it is also made on p.\ 268 of \cite{LS87}, where the authors declare that the
global adelic transfer factors should vanish if this stronger assumption is not satisfied.

On the other hand, the constructions in \cite{KS99} operate under the weaker
assumption on the existence of related elements $\gamma^H \in H(F)$ and $\delta
\in G(\A)$. It is clear to us that if the weaker assumption is not satisfied
then the endoscopic datum is not needed for the stabilization of the trace
formula and its transfer factors can be declared zero, see Lemma
\ref{lem:loc_glob_tf}. However, it is not clear to us that this is true if the
stronger assumption is not satisfied. It could be that a Hasse principle for
this assumption holds, but we do not know if it holds, let alone of any
published proof.
It may be possible to approach this problem using the classification of Dynkin
diagrams (similar to \cite{Lan79}), after reducing to the case where $G$ is
absolutely simple and simply connected and the endoscopic datum is elliptic.

Instead, we have decided to generalize here the proof of \cite[Proposition 4.4.1]{KalGRI} by dropping the stronger assumption and only keeping the weaker assumption. To replicate the proof of Proposition \ref{pro:adelic_tf_iso} in the rigid
setting, we want an analogue of Proposition \ref{pro:obs_iso}.
The natural strategy would be to introduce extensions of $\Gamma$ bound by
$\ol{\A}$- and ``$\ol{\A} / \ol{F}$-points'' of certain projective limits of
finite multiplicative groups over $F$ (analogous to the gerbes $\mc{E}^1$ and
$\mc{E}^2$ of \cite{KotBG}), prove Tate-Nakayama isomorphisms and local-global
compatibility (analogous to \cite[(6.6)]{KotBG}). This is certainly possible and was in fact mostly done in the preparations to \cite{KalGRI}. However, it takes a fair amount of pages to set up. We have thus chosen an alternative approach here, which uses the result in the ``iso'' setting (Proposition
\ref{pro:obs_iso}) and the comparison with ``rig'' via ``mid'' to deduce
the result in the rigid setting (Proposition \ref{pro:obs_rig}).

\begin{lem} \label{lem:cart_rig_semi_adelic}
  Let $T$ be a torus over $F$, $Z \subset T$ a finite multiplicative subgroup.
  Then the diagram of Abelian groups
  \[ \xymatrix{
    H^1(P^\tx{rig}_{\dot V} \to \mc{E}^\tx{rig}_{\dot V}, Z(\ol{F}) \to
    T(\ol{\A})) \ar[r] \ar[d] & \bigoplus_{v \in V} H^1(P^\tx{rig}_v \to
    \mc{E}^\tx{rig}_v, Z \to T) \ar[d] \\
    \tx{Hom}(P^\tx{rig}_{\dot V}, Z) \ar[r] & \bigoplus_{v \in V}
    \tx{Hom}(P^\tx{rig}_v, Z)
  } \]
  is Cartesian.
\end{lem}
\begin{proof}
  We have to prove that the map from $H^1(P^\tx{rig}_{\dot V} \to
  \mc{E}^\tx{rig}_{\dot V}, Z(\ol{F}) \to T(\ol{\A}))$ to the fiber product is
  injective and surjective.
  Injectivity follows easily from the well-known isomorphism
  \[ H^1(\Gamma, T(\ol{\A})) \simeq \bigoplus_{v \in V} H^1(\Gamma_v,
  T(\ol{F_v})). \]

  To prove surjectivity, let $\mu \in \tx{Hom}(P^\tx{rig}_{\dot V}, Z)$ and
  $([z_v]) \in \bigoplus_{v \in V} H^1(P^\tx{rig}_v \to \mc{E}^\tx{rig}_{\dot
  v}, Z \to T)$ be such that for any $v \in V$, the restriction of $z_v$ to
  $P^\tx{rig}_v$ equals $\mu$ composed with $\tx{loc}_v : P^\tx{rig}_v \to
  P^\tx{rig}_{\dot V}$.
  Recall that for any $v \in V$ there is $\eta_v \in C^1(\Gamma_v,
  P^\tx{rig}_{\dot V})$ unique up to $Z^1(\Gamma_v, P^\tx{rig}_{\dot V}) =
  B^1(\Gamma_v, P^\tx{rig}_{\dot V})$ such that $\xi^\tx{rig}_{\dot
  V}|_{\Gamma_v^2} = (\tx{loc}_v \circ \xi^\tx{rig}_v ) \times \tx{d}(\eta_v)$.
  Let $S$ be a finite set of places of $F$ containing all Archimedean places and
  such that $[z_v]=1$ for $v \in V \smallsetminus S$, so that in particular $\mu
  \circ \tx{loc}_v$ is trivial.
  Recall that for any $v \in V$ we can represent $z_v$ as an element of
  $C^1(\Gamma_v, T(\ol{F_v}))$ such that $\tx{d}(z_v) = \mu \circ
  \tx{loc}_v \circ \xi^\tx{rig}_v$.
  There exists a finite Galois extension $E/F$ such that the action of
  $\Gamma$ on $X^*(Z)$ factors through $\Gamma_{E/F}$ and $E$ contains all roots
  of unity of order dividing the exponent of $Z$ (in particular
  $Z(E)=Z(\ol{F})$) and for all $v \in S$ $z_v$ is inflated from an element of
  $C^1(\Gamma_{E_{\dot v}/F_v}, T(E_{\dot v}))$ and $\mu \circ \eta_v$ is
  inflated from an element of $C^1(\Gamma_{E_{\dot v}/F_v}, Z(E_{\dot v}))$.
  For $v \in S$ let $\tilde{z}^{(0)}_v = \dot{S}^1( z_v \times (\mu \circ
  \eta_v)) \in C^1(\Gamma_{E/F}, T(E \otimes_F F_v))$ where $\dot{S}^1$ is a
  Shapiro corestriction map associated to a section of $\Gamma_{E/F} \to
  \Gamma_{E_{\dot v}/F_v} \backslash \Gamma_{E/F}$ as explained in
  \cite[Appendix B]{KalGRI}.
  We have $[\tx{d}(\tilde{z}^{(0)}_v)] = \mu_* [\xi^\tx{rig}_{\dot V}]$ in
  $H^2(\Gamma_{E/F}, Z(E \otimes_F F_v)) \simeq H^2(\Gamma_{E_{\dot v}/F_v},
  Z(E_{\dot v}))$ (Shapiro isomorphism), therefore there exists $a_v \in
  C^1(\Gamma_{E/F}, Z(E \otimes_F F_v))$ such that $\tilde{z}_v :=
  \tilde{z}^{(0)}_v \times \tx{d}(a_v)$ satisfies $\tx{d}(\tilde{z}_v) =
  \mu_*(\xi^\tx{rig}_{\dot V})$ (via the embedding $Z(E) \subset Z(E \otimes_F
  F_v)$).
  For $v \in V \smallsetminus S$ we have $\mu_*([\xi^\tx{rig}_{\dot V}]) = 1$ in
  $H^2(\Gamma_{E/F}, Z(E \otimes_F F_v))$, thanks to
  \begin{itemize}
    \item the Shapiro isomorphism $H^2(\Gamma_{E/F}, Z(E \otimes_F F_v)) \simeq
      H^2(\Gamma_{E_{\dot v}/F_v}, Z(E_{\dot v}))$,
    \item the compatibility between the local and global canonical classes which
      implies that the image of $\mu_*([\xi^\tx{rig}_{\dot V}])$ in
      $H^2(\Gamma_v, Z)$ equals $(\mu \circ
      \tx{loc}_v)_*(\xi^\tx{rig}_v)$,
    \item the fact that $\mu \circ \tx{loc}_v : P^\tx{rig}_v \to Z$ is trivial
      because $[z_v]=1$,
    \item the local analogue of \cite[Lemma 3.2.7]{KalGRI} (similar proof) which
      says that the inflation map $H^2(\Gamma_{E_{\dot v}/F_v}, Z(E_{\dot v}))
      \to H^2(\Gamma_v, Z)$ is injective.
  \end{itemize}
  So for $v \in V \smallsetminus S$ we can find $\tilde{z}_v \in
  C^1(\Gamma_{E/F}, Z(E \otimes_F F_v))$ satisfying $\tx{d}(\tilde{z}_v) =
  \mu_*(\xi^\tx{rig}_{\dot V})$.
  Now $(\tilde{z}_v)_{v \in V}$ represents an element of $H^1(P^\tx{rig}_{\dot
  V} \to \mc{E}^\tx{rig}_{\dot V}, Z(\ol{F}) \to T(\ol{\A}))$ whose restriction
  to $P^\tx{rig}_{\dot V}$ is $\mu$ and mapping to $([z_v])_v$.
\end{proof}

\begin{cor} \label{cor:TN_rig_semi_adelic}
  In the same setting, denote $Y = X_*(T)$ and $\ol{Y} = X_*(T/Z)$.
  Then there is a unique Tate-Nakayama isomorphism between
  \[ Y^\tx{rig}_\tx{sa}(Z \to T) := \left\{ (\lambda_v)_v \in \bigoplus_{v \in
    V} (\ol{Y}/I_{\dot v}(Y))[\tx{tor}] \middle | \sum_{v \in V} (\lambda_v +
    IY) \in (Y/IY)[\tx{tor}] \right\} \]
  and $H^1(P^\tx{rig}_{\dot V} \to \mc{E}^\tx{rig}_{\dot V}, Z(\ol{F}) \to
  T(\ol{\A}))$ which is compatible with the local Tate-Nakayama isomorphisms and
  the identification
  \[ \tx{Hom}(P^\tx{rig}_{\dot V}, Z) \simeq \varinjlim_E
  ((\ol{Y}/Y)[\dot{V}_E]_0)^{N_{E/F}=0}. \]
\end{cor}
\begin{proof}
  We have an obvious Cartesian diagram
  \[ \xymatrix{
    Y^\tx{rig}_\tx{sa}(Z \to T) \ar[r] \ar[d] & \bigoplus_{v \in V}
    (\ol{Y}/I_{\dot v}(Y))[\tx{tor}] \ar[d] \\
    \varinjlim_E ((\ol{Y}/Y)[\dot{V}_E]_0)^{N_{E/F}=0} \ar[r] & \bigoplus_{v \in
    V} (\ol{Y}/Y)^{N_{\dot v}=0}
  } \]
  and comparing with the Cartesian diagram in the previous lemma gives the
  sought isomorphism.
\end{proof}

\begin{pro} \label{pro:obs_rig}
  We have a commutative diagram
  \[ \xymatrix{
      Y^\tx{rig}_\tx{sa}(Z \to T) \ar[r]^-\sim \ar[d] & H^1(P^\tx{rig}_{\dot V}
      \to \mc{E}^\tx{rig}_{\dot V}, Z(\ol{F}) \to T(\ol{\A})) \ar[d] \\
      (Y/IY)[\tx{tor}] \ar[r]^-\sim & H^1(\Gamma, T(\ol{\A}) / T(\ol{F}))
  } \]
  where the top horizontal isomorphism was defined in the previous corollary,
  the bottom isomorphism was defined in \cite{Tate66}, the left vertical map is
  $(\lambda_v)_v \mapsto \sum_{v \in V} (\lambda_v + IY)$ and the right vertical
  map is the obvious one.
\end{pro}
\begin{proof}
  The kernel of the right vertical map is the image of $H^1(P^\tx{rig}_{\dot V}
  \to \mc{E}^\tx{rig}_{\dot V}, Z \to T)$ and by the local and global
  Tate-Nakayama isomorphisms for ``rig'' and their compatibility, we have an
  embedding $\iota^\tx{rig}$ from the cokernel $C^\tx{rig}$ of
  \[ H^1(P^\tx{rig}_{\dot V} \to \mc{E}^\tx{rig}_{\dot V}, Z \to T) \to
    H^1(P^\tx{rig}_{\dot V} \to \mc{E}^\tx{rig}_{\dot V}, Z(\ol{F}) \to
  T(\ol{\A})) \]
  into $(Y/IY)[\tx{tor}]$, mapping $(\lambda_v)_v$ to $\sum_{v \in V} (\lambda_v
  + IY)$.
  For the rest of this proof we take the inductive limit over all finite
  multiplicative subgroups $Z$ of $T$, as we may since all morphisms in sight
  are compatible with the transition maps induced by any inclusion $Z \subset
  Z'$.
  This has the effect of replacing $\ol{Y}$ by $\Q Y$.
  Using Lemma \ref{lem:repr_Sdot_supp} it is easy to check that in the limit
  $\iota^\tx{rig}$ is also surjective, but this will also be a consequence of
  the rest of the proof.
  It is however \emph{not} obvious that the map $C^\tx{rig} \to
  (Y/IY)[\tx{tor}]$ induced by the top horizontal, right vertical and bottom
  horizontal maps is the identity.

  Lemma \ref{lem:cart_rig_semi_adelic} and Corollary
  \ref{cor:TN_rig_semi_adelic} admit ``iso'' and ``mid'' analogues, with similar
  proofs and with
  \[ Y^\tx{iso}_\tx{sa} = \left\{ (\lambda_v)_v \in \bigoplus_{v \in V}
    Y/I_{\dot v}(Y) \middle| N_{E/F} \left( \sum_{v \in V} \lambda_v \right) = 0
    \right\} \]
  and
  \begin{multline*}
    Y^\tx{mid}_\tx{sa} = \Big\{ ((\lambda_v)_v, (\mu_v)_v) \Big| (\lambda_v)_v
    \in Y^\tx{mid}_\tx{sa}, (\mu_v)_v \in \bigoplus_{v \in V} \Q Y, \sum_{v \in
    V} \mu_v = 0 \\
    \text{and } \forall v \in V, N_{E_{\dot v}/F_v}(\lambda_v - \mu_v) = 0
    \Big\} 
  \end{multline*}
  where $E/F$ is any finite Galois extension splitting $T$.
  We also have an embedding $\iota^\tx{iso}$ (resp.\ $\iota^\tx{mid}$) from
  $C^\tx{iso}$ (resp.\ $C^\tx{mid}$) into $(Y/IY)[\tx{tor}]$, mapping the class
  of $(\lambda_v)_v$ (resp.\ $((\lambda_v)_v, \mu)$) to $\sum_{v \in V}
  \lambda_v + IY$.
  We have natural maps between cokernels
  \[ C^\tx{iso} \longleftarrow C^\tx{mid} \longrightarrow C^\tx{rig} \]
  compatible with $\iota^\tx{iso}$, $\iota^\tx{mid}$, $\iota^\tx{rig}$ thanks to
  the compatibility of local Tate-Nakayama isomorphisms in the three settings.
  It follows from Proposition \ref{pro:surj} (or a direct argument) that the
  natural map $Y^\tx{mid}_\tx{sa} \to Y^\tx{iso}_\tx{sa}$ is surjective, thus
  $C^\tx{mid} \to C^\tx{iso}$ is bijective.
  It is obvious that $\iota^\tx{iso}$ is surjective, and by Proposition
  \ref{pro:obs_iso} it is the identity map.
  Therefore $\iota^\tx{mid}$ and $\iota^\tx{rig}$ are also bijective and
  identity maps.
\end{proof}

\begin{cor}
  Proposition 4.4.1 of \cite{KalGRI} holds without the assumption that there
  exists a pairs of related elements in $H_1(F)_\tx{sr} \times G(F)_\tx{sr}$.
\end{cor}
\begin{proof}
  Similar to the proof of Proposition \ref{pro:adelic_tf_iso}, applying
  Proposition \ref{pro:obs_rig} to $T^*_\tx{sc}$ instead of Proposition
  \ref{pro:obs_iso} to $T^*$.
  Details are left to the reader.
\end{proof}


\subsection{Isocrystal local Langlands correspondence in the Archimedean case}
\label{sub:iso_LLC_archi}

In order to formulate the isocrystal version of the conjectural multiplicity formula for a connected reductive group over a number field we will need the isocrystal version of the refined local Langlands conjecture at all places. In the non-archimedean case this was formulated in \cite{KalRIBG}. Fortunately the
Archimedean case is similar, and we shall formulate it here, as well as the comparison between the isocrystal and rigid versions.

The complex case $F \simeq \C$ is very simple: the group $P^\tx{rig}$ is trivial
and so is $\mc{E}^\tx{rig}$.
For any connected reductive group $G$ over $F$ and any tempered parameter
$\varphi : F^\times \to \hat G$ the group $S_\varphi^\natural$ defined as in
\cite[\S 4.1]{KalRIBG} is canonically isomorphic to $\hat{G} / \hat{G}_\tx{der}
= \hat{C}$ where $C = Z(G)^0$.
Since $Z(\hat{G}) \to \hat{C}$ is an isogeny, if we fix $z^\tx{iso} \in
B(G)_\tx{bas} = X_*(C)$ then the set of characters of $S_\varphi^\natural$ whose
restriction to $Z(\hat{G})$ is $[z^\tx{iso}] = z^\tx{iso}|_{Z(\hat{G})}$ is just
$\{z^\tx{iso}\}$.
Thus the isocrystal version of the local Langlands correspondence for $G$ is
simply the usual correspondence with the extra datum of an element of $X_*(C)$.

We are left to consider the real case $F=\R$.
Recall that the analogue of the morphism of extensions $\mc{E}^\tx{rig} \to
\mc{E}^\tx{iso}$ of \cite[(3.13)]{KalRIBG} is the composition $s^\tx{iso} \circ
c^\tx{rig}$ as defined in Section \ref{sub:loc_gerbes}.
Note that the analogue of \cite[Proposition 3.2]{KalRIBG} is a direct
consequence of Propositions \ref{pro:tnisomid} and \ref{pro:tnmidrigloc}.
Recall that Kottwitz defined a map $\kappa_G : B(G)_\tx{bas} \to
X^*(Z(\hat{G})^\Gamma)$, whose image was characterized in \cite[Proposition
13.4]{KotBG} as the group of $\chi|_{Z(\hat{G})^\Gamma}$ (note that
$X^*(Z(\hat{G}))$ maps onto $X^*(Z(\hat{G})^\Gamma)$) such that
the element $N_{\C/\R}(\chi)$ of $X^*(Z(\hat{G}))$ belongs to $X^*(\hat{G})$.
The analogue of \cite[Proposition 3.3]{KalRIBG} is that the following diagram
commutes, although the horizontal maps are not bijective in general.
\[ \xymatrix{
    B(G)_\tx{bas} \ar[r]_{\kappa_G} \ar[d] & X^*(Z(\hat{G})^\Gamma) \ar[d] \\
    H^1_\tx{alg}(\mc{E}^\tx{rig}, G) \ar[r] & \tx{Hom}(\pi_0((Z(\hat{G}_\tx{sc})
    \times Z(\hat{G})^0_\infty)^+), \C^\times)
}\]
Here the left vertical map is induced by $s^\tx{iso} \circ
c^\tx{rig}:\mc{E}^\tx{rig} \to \mc{E}^\tx{iso}$,
the bottom horizontal map is obtained as the composition of \cite[Theorem 4.8
and Proposition 5.3]{KalRI} and the right vertical map is as in
\cite[Proposition 3.3]{KalRIBG}, i.e.\ it is dual to the map
\begin{align} \label{eq:pi0_to_Gammainv}
  \pi_0((Z(\hat{G}_\tx{sc}) \times Z(\hat{G})^0_\infty)^+) & \longrightarrow
  Z(\hat{G}^\Gamma) \\
  (a, (b_n)_{n>0}) & \longmapsto a b_1 N_{\C/\R}(b_2)^{-1} \nonumber
\end{align}
Note that this maps $N_{\C/\R}(\pi_0(Z(\hat{G}_\tx{sc}) \times
Z(\hat{G})^0_\infty))$ to $N_{\C/\R}(Z(\hat{G}_\tx{der}))$.
The proof of \cite[Proposition 3.3]{KalRIBG} applies almost verbatim, replacing
``elliptic torus'' by ``fundamental torus'' and using \cite[Lemma 13.2]{KotBG}
instead of \cite[Proposition 5.3]{Kot85}.
Note that the first argument of the proof, showing that $B(S)$ maps to
$B(G)_\tx{bas}$, does not hold in the real case but this is not necessary if one
uses $B(S)_{G-\tx{bas}} \simeq H^1(\mb{T}^\tx{iso} \to \mc{E}, C \to S)$ instead
of $B(S)$ as in \cite[\S 13.5]{KotBG}.

For $\varphi : W_F \to {}^L G$ a Langlands parameter denote by $S_\varphi$ its
centralizer in $\hat{G}$ and define its quotient $S_\varphi^\natural$ (a complex
reductive group) as in \cite[\S 4.1]{KalRIBG}.
We can define $\pi_0(S_\varphi^+) \to S_\varphi^\natural$ similarly to
\eqref{eq:pi0_to_Gammainv}.
The proof of \cite[Lemma 4.1]{KalRIBG} does not use anything specific to the
non-Archimedean case, so it still holds.

Now let $G^*$ be a quasi-split connected reductive group over $F$.
Fix a Whittaker datum $\mf{w}$.
Consider an inner form $(G, \psi)$ of $G^*$.
Let $z \in \widetilde{B}(G)_\tx{bas} \simeq \ol{Z}^1_\tx{alg}(\mb{T}^\tx{iso}
\to \mc{E}^\tx{iso}, C \to G^*)$ be a lift of the cocycle $\Gamma \to
G_\tx{ad}(\ol{F}), \sigma \mapsto \psi^{-1} \sigma(\psi)$.
Note that in general such a lift may not exist.

\begin{thm}
  Fix a connected reductive quasi-split group $G^*$ over $\R$ and a Whittaker
  datum $\mf{w}$.
  There is a unique bijection between isomorphism classes of quadruples $(G,
  \psi, z, \pi)$ and isomorphism classes of pairs $(\varphi, \rho)$ where
  $\varphi$ is a tempered Langlands parameter and $\rho$ is an algebraic
  irreducible representation of $S_\varphi^\natural$, such that
  \begin{itemize}
    \item for given $(G, \psi, z)$ and $\varphi$ the L-packet $\Pi_\varphi$ of
      isomorphism classes of $\pi$ such that $(G, \psi, z, \pi)$ corresponds to
      $(\varphi, \rho)$ for some $\rho$ equals the one defined by Langlands in
      \cite{Lan89}, and
    \item the endoscopic character relations \cite[(4.3)]{KalRIBG} hold with respect to the transfer factor \eqref{eq:tf_iso}.
  \end{itemize}
  This correspondence is compatible with the rigid version proved in
  \cite[\S 5.6]{KalRI}, in the same sense as in \cite[\S 4.2]{KalRIBG}.
\end{thm}
\begin{proof}
  This is deduced from the rigid version of the local Langlands correspondence
  exactly as in \cite[\S 4.2]{KalRIBG}.
\end{proof}

If $\mf{w}$, $(G, \psi, z)$ and $\varphi$ are fixed we will denote, for $\pi \in
\Pi_{\varphi}$, $\< \pi, \cdot \>_{\psi, z, \mf{w}}$ for the character $\rho$ of
$S_\varphi^\natural$ corresponding to $\pi$.


\subsection{Multiplicity formula in the isocrystal setting}
\label{sub:mult_formula_BG}

In this section we assume that $G$ satisfies the Hasse principle.
We shall formulate a version of the formula \cite[(12.3)]{Kot84} for the
multiplicity of an irreducible admissible representation of $G(\A)$ in the
discrete automorphic spectrum of $G$ using Kottwitz's global set $B(G)$.

Assume the existence of a global Langlands group $L_F$, and consider a
continuous semi-simple global parameter $\varphi : L_F \to {^LG}$.
For simplicity we do not consider more general Arthur-Langlands parameters,
although they do not present additional difficulty for the following discussion,
only requiring a slightly more complicated formulation.
Recall from \cite[\S10.2]{Kot84} the group $S_\varphi$ of self-equivalences of
$\varphi$ is defined as
\begin{multline*}
  \Big\{ g \in \hat G \,\Big|\, \forall x
  \in L_F:\  g^{-1}\varphi(x)g\varphi(x)^{-1}=:z(x) \in Z(\hat G) \\
  \text{and } \forall v \in V, 0=[z_v] \in H^1(L_{F_v},Z(\hat G)) \Big\}.
\end{multline*}
Note that $z$ is an element of $Z^1(L_F,Z(\hat G))$ for formal reasons. The Hasse principle, reinterpreted as \cite[(4.2.2)]{Kot84}, together with \cite[Lemma 11.2.2]{Kot84}  imply $S_\varphi=C_\varphi \cdot Z(\hat G)$, where
\[ C_\varphi = \{g \in \hat G| g^{-1}\varphi(x)g\varphi(x)^{-1}=1\}. \]
We conclude that, if $\varphi$ is discrete, then the finite group
$\mc{S}_\varphi=\pi_0(S_\varphi/Z(\hat G))$ equals $C_\varphi/Z(\hat G)^\Gamma$.

Assume from now on that $\varphi$ is discrete.
For each place $v$ we assume the isocrystal version of the refined local Langlands conjecture, as stated in \cite[\S4.1]{KalRIBG} when $v$ is finite, and in \S\ref{sub:iso_LLC_archi} when $v$ is infinite. 
In particular at each place we have the L-packet $\Pi_{\varphi_v}$.
Choose a reductive model $\ul{G}$ of $G$ over $\mc{O}_F[1/N]$ for some integer
$N>0$.
For almost all finite places $v$ of $F$, the L-packet $\Pi_{\varphi_v}$ contains
a unique unramified representation (with respect to $\ul{G}(\mc{O}(F_v)$).
Given a collection of $\pi_v \in \Pi_{\varphi_v}(G)$, unramified for almost all
$v$ so that the restricted tensor product $\pi = \otimes'_v \pi_v$ is
well-defined, we now define the class-function $\<\pi,-\>$ on $\mc{S}_\varphi$
as follows.
Let $G^*$ be the quasi-split inner form of $G$ and let $\psi : G^* \to G$ be an
inner twist.
By Corollary \ref{cor:surj_H1_mid_iso_nonabelian} we can choose $z^\tx{iso} \in
Z^1_\tx{bas}(\mc{E}^\tx{iso},G^*)$ such that $\psi^{-1}\sigma(\psi)=\tx{Ad}(\bar
z_\sigma)$.
The choice of $\psi$ gives an identification $^L\psi : {^LG} \to {^LG^*}$ and
hence a parameter $\varphi^*={^L\psi}\circ\varphi$ for $G^*$.
Choose also a global Whittaker datum $\mf{w}$ for $G^*$.
For each place $v$ we have the complex number
$\<\pi_v,s^*\>_{\psi_v, z^\tx{iso}_v, \mf{w}_v}$ for $s^* \in C_{\varphi^*}$.
For almost all finite places $v$ this complex number equals one: this follows
from the endoscopic character relations \cite[(4.3)]{KalRIBG} and Lemma
\ref{lem:iso_unr_ae} by the same argument as in \cite[Lemma 4.5.1]{KalGRI}.
Given $s \in C_\varphi$ define
\[ \<\pi,s\>_\tx{iso} = \prod_v \<\pi_v,{^L\psi}(s)\>_{\psi_v, z^\tx{iso}_v,
\mf{w}_v}. \]

\begin{lem} \label{lem:can_pairing_iso}
  The complex number $\<\pi,s\>_\tx{iso}$ is independent of the choice of
  $\mf{w}$, $z^\tx{iso}$ and $\psi$ and the map
  \[ C_\varphi \to \C,\quad s \mapsto \<\pi,s\>_\tx{iso} \]
  is invariant under $Z(\hat G)^\Gamma$.
\end{lem}
\begin{proof}
  The parameter $\varphi$ and $s$ induce an elliptic endoscopic datum $(H,
  \mc{H}, s, \eta)$ with $\eta(\widehat{H}) = \tx{Cent}(s, \widehat{G})^0)$ and
  $\eta(\mc{H}) = \varphi(L_F) \eta(\widehat{H})$.
  Choose a z-extension $H_1$ of $H$ and an L-embedding $\eta_1 : \mc{H} \to
  {^LH_1}$ as in \cite[Lemma 2.2.A]{KS99}.
  Denote $\varphi' = \eta_1 \eta^{-1} \varphi : L_F \to {^LH_1}$.
  Fix a maximal compact subgroup $K_\infty$ of $G(\R \otimes_\Q F)$.
  For $f(g) dg$ a smooth bi-$K_{\infty}$-finite compactly supported distribution
  on $G(\A)$, there exists a transfer $f'(h) dh$ on $H(\A)$ in the sense of
  \cite[(5.5)]{KS99}.
  Note that the global transfer factors for $(H, \mc{H}, s, H_1, \eta_1)$ are
  canonical, and functorial under isomorphisms of endoscopic data. As discussed in Remark \ref{rem:adelic_transfer_can}, the canonical adelic transfer $f'$ of $f$ satisfies the endoscopic character identities at each place $v$ with respect to the normalized transfer factor $\Delta[\mf{w}_v,z^\tx{iso}_v]$. Therefore, taking the product of the local endoscopic character relations
  \cite[(4.3)]{KalRIBG} at all places, we have
  \[ \sum_{\pi \in \Pi_\varphi} \langle \pi, s \rangle_\tx{iso} \tr \pi(f(g)dg)
  = \sum_{\pi' \in \Pi_{\varphi'}} \tr \pi(f'(h) dh) \]
  and since distributions $f(g) dg$ as above separate elements of $\Pi_\varphi$
  this shows that $\langle \pi, s \rangle_\tx{iso}$ does not depend on any
  choice, and by functoriality of global transfer factors invariant under
  translation of $s$ by $Z(\widehat{G})^\Gamma$.

  Independence of $\psi$: Note first that we can replace $\psi$ by
$\psi\circ\tx{Ad}(g)$ and $z^\tx{iso}$ by $g^{-1}z^\tx{iso}\sigma(g)$ for any $g
\in G^*(\bar F)$ without changing the numbers $\<\pi_v,s^*\>_{z_v^\tx{iso}}$.
Let $\psi' : G^* \to G$ be another inner twist.
Consider the automorphism $\theta = \psi^{-1}\circ\psi'$ of $G^*$.
Changing $\psi$ by $\psi\circ\tx{Ad}(g)$ if necessary we may assume that
$\theta$ fixes an $F$-pinning of $G^*$.
For any $\sigma \in \Gamma$ the automorphism $\theta^{-1}\sigma(\theta)$ is
inner and fixes an $F$-pinning, hence trivial, and we conclude that $\theta$ is
defined over $F$.
Thus replacing $\psi$ by $\psi'=\psi\circ\theta$ has the effect of replacing
$^L\psi$ by ${^L\psi}\circ{^L\theta}$, $\varphi^*$ by ${^L\theta}\circ\varphi^*$, and $s^*$ by ${^L\theta}(s^*)$. The functoriality of the refined local Langlands conjecture \cite[Appendix A]{KalLLCD} implies that $\<\pi_v,s^*\>_{z^\tx{iso}_v}$ remains unchanged.
\end{proof}

We thus obtain a function $\<\pi,-\>_\tx{iso}$ on $\mc{S}_\varphi=C_\varphi/Z(\hat G)^\Gamma$ that depends only on the Langlands parameter $\varphi$. The conjectural multiplicity formula \cite[(12.3)]{Kot84} can now be stated using that function.


\subsection{Comparison between the global pairings}

In the last subsection we defined the pairing $\<\pi,-\>_\tx{iso}$ under the
assumptions that $G$ has connected center and satisfies the Hasse principle,
using the cohomology of $\mc{E}^\tx{iso}$.
On the other hand, in \cite[\S4.5]{KalGRI} we defined a pairing $\<\pi,-\>$
without assumptions on $G$, using the cohomology of $\mc{E}^\tx{rig}$.
In \cite[\S4.5]{KalGRI} we did not discuss the independence of the pairing
defined in \cite[Proposition 4.5.2]{KalGRI} of the choice of $\psi$.
The proof of Lemma \ref{lem:can_pairing_iso} applies verbatim to the rigid
version, and the following proposition follows.

\begin{pro} \label{pro:paircomp}
  Assume that the connected reductive group $G$ has connected center.
  Then the pairings $\langle \pi, \cdot \rangle_\tx{iso}$ and $\langle \pi,
  \cdot \rangle$ are equal.
\end{pro}

Using the Galois gerbe $\mc{E}^\tx{rig}_{\dot V}$ introduced in this paper, we
can obtain a finer comparison result. Namely, we can see how the local pairings are related at each place $v$. This also implies Proposition \ref{pro:paircomp}, as follows.

\begin{proof}[Alternate proof of Proposition \ref{pro:paircomp}]
  By Corollary \ref{cor:surj_H1_mid_iso_nonabelian} (or Remark
  \ref{rem:splitting_H1}) we can choose $z^\tx{mid}
  \in Z^1(\mc{E}^\tx{mid}_{\dot V}, Z(G^*) \to G^*)$ lifting $z^\tx{iso}$.
  Up to pre-composing $\psi$ with an inner automorphism of $G^*_{\ol{F}}$ we can
  assume that there exists a maximal torus $T^*$ of $G^*$ such that $z^\tx{mid}
  \in Z^1(\mc{E}^\tx{mid}_{\dot V}, Z(G^*) \to T^*)$.
  Let $(\lambda,\mu) \in Y^\tx{mid}(Z(G^*) \to T^*)$ be its linear algebraic
  description given by Proposition \ref{pro:tnmidglob} and Corollary
  \ref{cor:tnmidglobrel}.
Thus there is a finite level $(E,S,\dot S_E)$ such that $\lambda \in Y[S_E]_0/IY[S_E]_0$ and $\mu \in (M^\tx{mid}_{E,\dot S_E}\otimes X_*(Z(G^*)))^\Gamma$ satisfy $\sum_{\sigma \in \Gamma_{E/F}}\sigma(\lambda(\sigma^{-1}w))=\sum_\sigma \mu(\sigma,w)$. Note that $\mu$ depends only on $z_\tx{mid}|_{\mb{T}^\tx{mid}} \in \tx{Hom}_F(\mb{T}^\tx{mid},Z(G^*))$ and is thus independent of the choice of $T^*$.

  Let $\bar s \in \mc{S}_\varphi$.
  Choose a semi-simple lift $s \in C_\varphi$ and furthermore a lift $\dot s \in
  S_\varphi^+$.
  We use the images $z^\tx{iso} \in Z^1_\tx{bas}(\mc{E}^\tx{iso},G^*)$ and
  $z^\tx{rig} \in Z^1(\mc{E}^\tx{rig},G^*)$ of $z^\tx{mid}$ and a global
  Whittaker datum $\mf{w}$ for $G^*$ to write as products over all places the
  complex numbers $\<\pi,\bar s\>_\tx{iso}=\prod_v \<\pi_v,s\>_{\psi_v,
  z^\tx{iso}}, \mf{w}_v$ and $\<\pi,\bar s\>_\tx{rig} = \prod_v \<\pi_v,\dot
  s\>_{\psi_v, z^\tx{rig}, \mf{w}_v}$.

By Lemma \ref{lem:locparcomp} below we have
\[  \<\pi,\bar s\>_\tx{iso} \cdot \<\pi,\bar s\>_\tx{rig}^{-1} = \prod_{v \in V} \<\mu_v,\dot s\>, \]
where for each place $v$ the localization $\mu_v \in X_*(Z(G^*)) \otimes \Q$ of $\mu$ is given by $\mu(1,\dot v)$ if $v \in S$ and equals $0$ if $v \notin S$, according to Proposition \ref{pro:tnmidglobloc}. We are thus pairing $\dot s$ with $\sum_{\dot v \in \dot S_E} \mu(1,\dot v)= \sum_{w \in S_E} \mu(1,w) = 0$.
\end{proof}


\subsection{Comparison between local pairings}

In this subsection we will compare the pairings between the local compound
$L$-packet of a tempered Langlands parameter and the centralizer of that
parameter that are guaranteed to exist by the isocrystal and rigid versions of
the refined local Langlands correspondence.
Let $G$ be a connected reductive group defined over a local
field $F$, $G^*$ its quasi-split inner form, $\psi : G^* \to G$ an inner twist,
$\varphi : L_F \to {^LG}$ a tempered Langlands parameter.
These local pairings are normalized by choices of elements $z^\tx{iso} \in
Z^1_\tx{bas}(\mc{E}^\tx{iso},G^*)$ and $z^\tx{rig} \in
Z^1_\tx{bas}(\mc{E}^\tx{rig},G^*)$ that lift the element $\psi^{-1}\sigma(\psi)
\in Z^1(\Gamma,G^*_\tx{ad})$, and of a Whittaker datum $\mf{w}$ for $G^*$.

When $z^\tx{iso} \mapsto z^\tx{rig}$ under the comparison homomorphism \cite[(3.14)]{KalRIBG} the two pairings were compared in \cite[\S4.2]{KalRIBG}.
The global setting of Proposition \ref{pro:paircomp} imposes a different
relationship between $z^\tx{iso}$ and $z^\tx{rig}$ -- they are the images of a
an element $z^\tx{mid}$ under the maps $c^\tx{iso}$ and $c^\tx{mid}$,
respectively.
Thus we will now combine the results of \cite[\S4.2]{KalRIBG} with the analysis of the non-commutativity of \eqref{eq:locnoncom} that was quantified in Corollary \ref{cor:tnlcononcomm}.

Choose an arbitrary maximal torus $S \subset G^*$ such that the class of
$z^\tx{mid}$ is in the image of $H^1(\mc{E}^\tx{mid},Z(G^*) \to S) \to
H^1(\mc{E}^\tx{mid},Z(G^*) \to G^*)$.
Let $(\lambda,\mu) \in Y^\tx{mid}$ denote the Tate-Nakayama element corresponding to $[z^\tx{mid}]$ under the isomorphism of Proposition \ref{pro:tnmidloc}. Thus $\lambda \in Y_\Gamma$, $\mu \in Y \otimes \Q$, and $N^\natural(\lambda)=N^\natural(\mu)$.
Since $z^\tx{mid}$ sends $\mb{T}^\tx{mid}$ into $C := Z(G^*)$, we have $\mu \in
X_*(C) \otimes\Q \subset Y \otimes \Q$, and this is independent of the
choice of $S$.

Denote by $p : \hat{G} \to \hat{C}$ the surjection dual to $C \subset G$.

\begin{lem} \label{lem:locparcomp}
  For any semi-simple element $\dot s \in S_\varphi^+$ and $\pi \in
  \Pi_\varphi(G)$ we have
  \[ \<\pi,s\>_{\psi, z^\tx{iso}, \mf{w}} = \<\mu, p(\dot s)\>\<\pi,\dot
  s\>_{\psi, z^\tx{rig}, \mf{w}}. \]
  where $p(\dot s)$ is the image of $\dot s$ in $[\hat{\bar C}]^+$.
\end{lem}
\begin{proof}
  In the proof we suppress $\psi$ and $\mf{w}$ from the notation since they are
  fixed.
The local pairings are conjectural. However, we know from \cite[\S4.2]{KalRIBG} that the following identity is implied by the assumed validity of the isocrystal and rigid versions of the refined local Langlands conjecture:
\[ \<\pi,t\>_{z^\tx{iso}} = \<\pi,\dot s\>_{x^\tx{rig}}, \]
where $t \in S_\varphi$ is equal to the image of $\dot s$ under the map $\pi_0(S_\varphi^+) \to S_\varphi$ of \cite[(4.7)]{KalRIBG} , and $x^\tx{rig}$ is the image of $z^\tx{iso}$ under the map $Z^1_\tx{bas}(\mc{E}^\tx{iso},G^*) \to Z^1_\tx{bas}(\mc{E}^\tx{rig},G^*)$ of \cite[(3.14)]{KalRIBG}. This reduces the question to comparing $\<\pi,t\>_{z^\tx{iso}}$ with $\<\pi,s\>_{z^\tx{iso}}$ and $\<\pi,\dot s\>_{x^\tx{rig}}$ with $\<\pi,\dot s\>_{z^\tx{rig}}$.

For the first comparison, we use the notation of \cite[\S4.2]{KalRIBG} and represent $\dot s$ as $(a_\tx{sc},(b_n)_n)$ with $a_\tx{sc} \in \hat G_\tx{sc}$ and $b_n \in Z(\hat G)^\circ$, $b_m^{m/n}=b_n$.
Then $s=a_\tx{der}b_1$ while $t=a_\tx{der}b_1N_{E/F}(b_{[E:F]})^{-1}$ for $E$ any
finite Galois extension of $F$ such that $\Gamma_{\ol{F}/E}$ acts trivially on
$Z(\hat G)$.
Thus $\<\pi,s\>_{z^\tx{iso}}=\<\pi,t\>_{z^\tx{iso}}\cdot \<\pi,N_{E/F}(b_{[E:F]}))_{z^\tx{iso}}$. Now $N_{E/F}(b_{[E:F]}) \in Z(\hat G)^{\circ,\Gamma}$, and the restriction of $\<\pi,-\>_{z^\tx{iso}}$ to $Z(\hat G)^\Gamma$ is the character $\<[z^\tx{iso}],-\>$, so we obtain
\[ \<\pi,s\>_{z^\tx{iso}}=\<\pi,t\>_{z^\tx{iso}}\cdot \<[{z^\tx{iso}}],N_{E/F}(b_{[E:F]})\>. \]

For the second comparison, we remind ourselves that $z^\tx{rig}$ is pulled back from $z^\tx{mid}$ via
\[ \xymatrix{
  1\ar[r]&\mb{T}^\tx{rig}\ar[r]\ar[d]&\mc{E}^\tx{rig}\ar[r]\ar[d]&\Gamma\ar[r]\ar@{=}[d]&1\\
  1\ar[r]&\mb{T}^\tx{mid}\ar[r]&\mc{E}^\tx{mid}\ar[r]&\Gamma\ar[r]&1
}
\]
while $x^\tx{rig}$ is pulled back from $z^\tx{mid}$ via
\[ \xymatrix{
  1\ar[r]&\mb{T}^\tx{rig}\ar[r]\ar[d]&\mc{E}^\tx{rig}\ar[r]\ar[d]&\Gamma\ar[r]\ar@{=}[d]&1\\
  1\ar[r]&\mb{T}^\tx{iso}\ar[r]\ar[d]&\mc{E}^\tx{iso}\ar[r]\ar[d]&\Gamma\ar[r]\ar@{=}[d]&1\\
  1\ar[r]&\mb{T}^\tx{mid}\ar[r]&\mc{E}^\tx{mid}\ar[r]&\Gamma\ar[r]&1
}
\]
Since the maps on $\Gamma$ are all the identity, both $z^\tx{rig}$ and $x^\tx{rig}$ map to the same element of $Z^1(\Gamma,G^*_\tx{ad})$, so their difference $x^\tx{rig}/z^\tx{rig}$ lies in $Z^1(\mc{E}^\tx{rig},Z(G^*))$, and in fact in $Z^1(\mc{E}^\tx{rig},Z)$ for some finite $Z \subset Z(G^*)$. Then we have
\[ \<\pi,\dot s\>_{x^\tx{rig}}=\<x^\tx{rig}/z^\tx{rig},(-d)\dot s\>\<\pi,\dot s\>_{z^\tx{rig}} = \<x^\tx{rig}/z^\tx{rig}, p(\dot s)\>\<\pi,\dot s\>_{z^\tx{rig}}. \]
The first equality is due to \cite[Lemma 6.2]{KalRIBG}, where $d : S_\varphi^+ \to Z^1(\Gamma,\hat Z)$ is the differential, and we are using the pairing $Z^1(\Gamma,\hat Z) \otimes H^1(\mc{E}^\tx{rig},Z) \to \C^\times$ of \cite[\S6.2]{KalRIBG}. The second equality is due to the commutative diagrams (6.1) and (6.2) in \cite{KalRIBG}, applied to the torus $C=Z(G^*)$, and we are using the pairing between $H^1(\mc{E}^\tx{rig},C)$ and $\pi_0([\hat{\bar C}]^+)$.

We now combine the comparisons and arrive at
\[ \<\pi,s\>_{z^\tx{iso}} = \<[z^\tx{iso}],N_{E/F}(b_{[E:F]})\>
\<x^\tx{rig}/z^\tx{rig}, p(\dot s)\> \<\pi,\dot s\>_{z^\tx{rig}}. \]
Recall the Tate-Nakayama element $(\lambda,\mu)$ corresponding to $z^\tx{mid}$.
We can see $\lambda \in Y_{\Gamma}$ as an element of $X^*(\hat{S}^\Gamma)$,
which via restriction to $Z(\hat G)^\Gamma$ maps to $[z^\tx{iso}]$.
For $m>0$ a sufficiently divisible integer we have that $N_{E/F} \lambda =
N_{E/F} \mu$ belongs to $m^{-1} X_*(C)$, and so
\[ \< [z^\tx{iso}], N_{E/F}(b_{[E:F]}) \> = \<N_{E/F}\lambda,b_{[E:F]}\> = \< m
N_{E/F}\mu, p(b_{m[E:F]})\> \]
where we have used the isogeny $p|_{Z(\hat G)^0} : Z(\hat G)^0 \to \hat C$.

Consider now the factor $\< x^\tx{rig} / z^\tx{rig}, p(\dot s) \>$.
According to Corollary \ref{cor:tnlcononcomm} it equals $\<\mu-N^\natural(\mu),
p(\dot s)\>$.
To evaluate this pairing, choose an integer $n>0$ divisible by $m[E:F]$ so that
$\mu \in n^{-1} X_*(C)$.
Note that $N^\natural(\mu) \in n^{-1} X_*(C)$ also.
Then $\< x^\tx{rig} / z^\tx{rig}, p(\dot s) \> = \<n(\mu-N^\natural(\mu)),
p(b_n)\>$.
We have
\[ \<nN^\natural(\mu), p(b_n)\> = \< mN_{E/F}(\mu), b_{m[E:F]}\>. \]
Thus
\[
\<\pi,s\>_{z^\tx{iso}} = \<\mu,\dot s\>\<\pi,\dot s\>_{z^\tx{rig}}\]
and the lemma is proved.
\end{proof}

\bibliographystyle{amsalpha}
\bibliography{bibliography}

\end{document}